\documentclass[oneside,article]{memoir}
\usepackage{amsmath}         % align-like environments
\usepackage{amsthm,thmtools,thm-restate} % theorem-like environments
\usepackage{enumitem}        % enumerate-like environments

\usepackage{tikz}
  \usetikzlibrary{matrix,arrows,positioning}
  
\usepackage{tikz-cd}

\usepackage{float} %to force figure into right place
\usepackage{caption, subcaption} % for subfigures

\usepackage[nobysame,alphabetic]{amsrefs} % bibliography tools
\usepackage[hidelinks]{hyperref}                   % hyperlinks
\usepackage[capitalize,nosort]{cleveref}           % named references

\usepackage{amssymb}
\usepackage[mathscr]{euscript}
\usepackage{relsize}
\usepackage{stmaryrd}
\usepackage{stackengine}

\usepackage{indentfirst} % indent the first paragraph
\usepackage{multicol}

\usepackage[inline,nomargin]{fixme} % fixme notes
  \fxsetup{targetlayout=color}

\usepackage{float} %to force figure into right place

\isopage[12]

\setlrmargins{*}{*}{1}
\checkandfixthelayout

\counterwithout{section}{chapter}
\usepackage{appendix}

  \newcommand{\iso}{\mathrel{\cong}}

  \newcommand{\op}{{\mathord\mathrm{op}}}

  \newcommand{\ob}{\operatorname{ob}}

  \newcommand{\push}{\cup}

  \newcommand{\colim}{\operatorname{colim}}

  \newcommand{\bd}{\partial}

  \newcommand{\gprod}{\otimes}

  \newcommand{\id}[1][]{\operatorname{id}_{#1}}
  
  \makeatletter
  \def\horn#1{\expandafter\horn@i#1,,\@nil}
  \def\horn@i#1,#2,#3\@nil{\Lambda^{#2}[#1]}
  \makeatother

  \newcommand{\uvar}{\mathord{\relbar}}

  \renewcommand{\tilde}{\widetilde}

  \renewcommand{\hat}{\widehat}

  \renewcommand{\bar}{\widebar}

  \newcommand{\nat}{{\mathord\mathbb{N}}}

  \newcommand{\cat}[1]{\mathscr{#1}}

  \newcommand{\ncat}[1]{\mathsf{#1}}

  \newcommand{\from}{\colon}

  \newcommand{\ito}{\hookrightarrow}

  \declaretheorem[style=definition,within=section]{definition}
  \declaretheorem[style=definition,numberlike=definition]{example}
  \declaretheorem[style=definition,numberlike=definition]{remark}

  \declaretheorem[style=plain,numberlike=definition]{corollary}
  \declaretheorem[style=plain,numberlike=definition]{lemma}
  \declaretheorem[style=plain,numberlike=definition]{proposition}
  \declaretheorem[style=plain,numberlike=definition]{theorem}

  \declaretheorem[style=plain,numbered=no,name=Theorem]{theorem*}

  \Crefname{corollary}{Corollary}{Corollaries}
  \Crefname{definition}{Definition}{Definitions}
  \Crefname{lemma}{Lemma}{Lemmas}
  \Crefname{proposition}{Proposition}{Propositions}
  \Crefname{remark}{Remark}{Remarks}
  \Crefname{theorem}{Theorem}{Theorems}

  \newlist{axioms}{enumerate}{1}
  \Crefname{axiomsi}{}{}

  \newenvironment{tikzeq*}
  {
    \begingroup
    \begin{equation*}
    \begin{tikzpicture}[baseline=(current bounding box.center)]
  }
  {
    \end{tikzpicture}
    \end{equation*}
    \endgroup
    \ignorespacesafterend
  }

  \tikzset
  {
    diagram/.style=
    {
      matrix of math nodes,
      column sep={4.3em,between origins},
      row sep={4em,between origins},
      text height=1.5ex,
      text depth=.25ex
    },
    over/.style={preaction={draw=white,-,line width=6pt}},
    every to/.style={font=\footnotesize},
    inj/.style={right hook->},
    surj/.style={-{Latex[open]}},
    cof/.style={>->},
    fib/.style={->>},
  }

  \DeclareFontFamily{U}{mathx}{\hyphenchar\font45}

  \DeclareFontShape{U}{mathx}{m}{n}{
    <5> <6> <7> <8> <9> <10>
    <10.95> <12> <14.4> <17.28> <20.74> <24.88>
    mathx10}{}

  \DeclareSymbolFont{mathx}{U}{mathx}{m}{n}

  \DeclareFontFamily{U}{mathb}{\hyphenchar\font45}

  \DeclareFontShape{U}{mathb}{m}{n}{
    <5> <6> <7> <8> <9> <10>
    <10.95> <12> <14.4> <17.28> <20.74> <24.88>
    mathb10}{}

  \DeclareSymbolFont{mathb}{U}{mathb}{m}{n}

  \DeclareMathAccent{\widebar}{0}{mathx}{"73}

  \DeclareMathSymbol{\Rsh}{\mathrel}{mathb}{"E9}

  \DeclareFontFamily{U}{MnSymbolA}{}

  \DeclareFontShape{U}{MnSymbolA}{m}{n}{
    <-6> MnSymbolA5
    <6-7> MnSymbolA6
    <7-8> MnSymbolA7
    <8-9> MnSymbolA8
    <9-10> MnSymbolA9
    <10-12> MnSymbolA10
    <12-> MnSymbolA12}{}

  \DeclareSymbolFont{MnSyA}{U}{MnSymbolA}{m}{n}

  \DeclareMathSymbol{\twoheaddownarrow}{\mathrel}{MnSyA}{27}

  \newcommand{\MSC}[1]{%
    \let\thempfn\relax
    \footnotetext[0]{2020 Mathematics Subject Classification: #1.}
  }

\declaretheorem[style=definition,numberlike=definition]{construction}
\Crefname{construction}{Construction}{Constructions}

\tikzstyle{vertex}=[circle, draw, minimum size=7pt, inner sep=0pt]

\newcommand{\cSet}{\mathsf{cSet}} 

\newcommand{\Graph}{\mathsf{Graph}}
\newcommand{\Set}{\mathsf{Set}}
\newcommand{\Ch}{\mathsf{Ch}} % category of bounded chain complexes
\newcommand{\Ab}{\mathsf{Ab}} % category of abelian groups
\newcommand{\Z}{\mathbb{Z}} % integers

\newcommand{\fcat}[2]{{#2}^{#1}} % functor category (source first then target)
\newcommand{\arr}[1]{\fcat{[1]}{#1}} % arrow category
\newcommand{\Top}{\ncat{Top}}
\newcommand{\adjunct}[4]{#1 \from #3 \rightleftarrows #4 : \! #2} % adjunction with left adjoint on left - arguments are: left adjoint, right adjoint, domain of #1, codomain of #1
\newcommand{\Mono}[1]{{#1}_{2}} % full subcat of arrow cat spanned by monos

\newcommand{\G}{\ncat{G}}
\newcommand{\Ghat}{\hat{\G}}
\newcommand{\im}{\mathsf{Im}} % reflector from presheaf category to graph
\newcommand{\gtimes}{\otimes} % box product of graphs
\newcommand{\gexp}[2]{\ensuremath{#1}^{\gtimes #2}} % Box-power of a graph
\newcommand{\ghom}[2]{\operatorname{hom}^{\gtimes}(#1, #2)} % internal hom in Graph

\newcommand{\cmap}[1][n]{\partial_{#1}} % chain map
\newcommand{\image}{\operatorname{im}} % image
\newcommand{\f}{\Phi} % the map (I_3m)^n -> |(i,e)-obox|^m in the proof that NG is Kan
\newcommand{\lstep}[2][n]{\lambda^{#1}_{#2}}
\newcommand{\lstepbar}[2][n]{\bar{\lambda}^{#1}_{#2}}
\newcommand{\rstep}[2][n]{\rho^{#1}_{#2}}
\newcommand{\rstepbar}[2][n]{\bar{\rho}^{#1}_{#2}}
\newcommand{\indP}[1]{P_{#1}} % colimit of ghom(|X|_m, -)
\newcommand{\loopsp}[1][]{\Omega^{#1}}

\newcommand{\boxcat}{\mathord{\square}} % box category

\newcommand{\face}[2]{\partial^{#1}_{#2}} % face operator on a cubical set
\newcommand{\degen}[2]{\sigma^{#1}_{#2}} % degeneracy operator on a cubical set
\newcommand{\conn}[2]{\gamma^{#1}_{#2}} % negative connection operator on a cubical set
\newcommand{\cube}[1]{\mathord{\square^{#1}}} % standard n-cube
\newcommand{\obox}[2]{\mathord{\sqcap^{#1}_{#2}}} % open box (place i and epsilon in same argument)
\newcommand{\dfobox}[1][n]{\mathord{\sqcap^{#1}_{i,\varepsilon}}} % open box with default argument (optional argument is n)
\newcommand{\Kan}{\mathsf{Kan}}
\newcommand{\reali}[2][]{\lvert #2 \rvert_{#1}} % realization from cSet to Graph - optional argument is for length-m realization
\newcommand{\sing}[2][]{\operatorname{Sing}_{#1} #2} % Sing of a graph (optional parameter for n as in Sing_n G)
\newcommand{\gnerve}[1][]{\mathrm{N}_{#1}} % nerve of a graph - optional parameter for m-nerve
\newcommand{\lhom}{\operatorname{hom}_{L}} % left hom-cSet
\newcommand{\rhom}{\operatorname{hom}_{R}} % right hom-cSet
\newcommand{\pp}[1][\gprod]{\mathrel{\hat{#1}}} % pushout product - optional argument for monoidal product symbol
\newcommand{\restr}[2]{{#1}|_{#2}} % restriction of a map
\newcommand{\bound}[1]{\beta[#1]} % bounding graph map I_inf -> I_{t}
\newcommand{\noproof}{\hfill\qedsymbol}
\newcommand{\sat}[1]{\operatorname{Sat} #1} % saturation of a class of maps

\author{Daniel Carranza \and Krzysztof Kapulkin}
\title{Cubical setting for discrete homotopy theory, revisited}
\date{\today}

\begin{document}

  \maketitle

  \begin{abstract}
    We construct a functor associating a cubical set to a (simple) graph.
    We show that cubical sets arising in this way are Kan complexes, and that the A-groups of a graph coincide with the homotopy groups of the associated Kan complex.
    
    We use this to prove a conjecture of Babson, Barcelo, de Longueville, and Laubenbacher from 2006, and a strong version of the Hurewicz theorem in discrete homotopy theory.
    
     \MSC{05C25 (primary), 55U35, 18N40, 18N45 (secondary)}
  \end{abstract}

\tableofcontents*
  
\section*{Introduction}

Discrete homotopy theory, introduced in \cite{barcelo-kramer-laubenbacher-weaver,babson-barcelo-longueville-laubenbacher}, is a homotopy theory in the category of simple graphs.
It builds on the earlier work of Atkin \cite{atkin:i,atkin:ii}, made precise in \cite{kramer-laubenbacher}, on the homotopy theory of simplicial complexes, and can also be generalized to the homotopy theory of finite metric spaces \cite{barcelo-capraro-white}.
It has found numerous applications \cite[\S5--6]{barcelo-laubenbacher}, including in matroid theory, hyperplane arrangements, combinatorial time series analysis, and, more recently, in topological data analysis \cite{memoli-zhou}.

The key invariants associated to graphs in discrete homotopy theory are the A-groups $A_n(G, v)$, named after Atkin, which are the discrete analogue of homotopy groups $\pi_n(X, x)$, studied in the homotopy theory of topological spaces.
In \cite{babson-barcelo-longueville-laubenbacher}, Babson, Barcelo, de Longueville, and Laubenbacher construct an assignment $G \mapsto X_G$, taking a graph $G$ to a topological space, constructed as a certain cubical complex, and conjecture that the A-groups of $G$ coincide with the homotopy groups of $X_G$.
They further prove \cite[Thm.~5.2]{babson-barcelo-longueville-laubenbacher} their conjecture under an assumption of a cubical approximation property \cite[Prop.~5.1]{babson-barcelo-longueville-laubenbacher}, a cubical analogue of the simplicial approximation theorem, which remains open.

The assignment $G \mapsto X_G$ arises as a composite 
\[ \begin{tikzcd}
    \Graph \ar[r, "{\gnerve[1]}"] & \cSet \ar[r, "{\reali{\uvar}}"] & \Top\text{.}
\end{tikzcd} \]
Here $\cSet$ denotes a particular category of cubical sets, which are well studied combinatorial models for the homotopy theory of spaces \cite{cisinski:presheaves}.
(Specifically, cubical sets used in this paper are cubical sets with positive and negative connections, cf.~\cite{brown-higgins:algebra-of-cubes,doherty-kapulkin-lindsey-sattler,buchholtz-morehouse}.)
Informally, a cubical set consists of a family of sets $\{ X_n \}_{n \in \mathbb{N}}$ of $n$-cubes together with a family of structure maps indicating how different cubes `fit together,' e.g., that a certain $(n-1)$-cube is a face of another $n$-cube.
More formally, it is a presheaf on the category $\square$ of combinatorial cubes.
The functor $\gnerve[1] \from \Graph \to \cSet$ is obtained by taking the $n$-cubes of $\gnerve[1]{G}$ to be maps $\gexp{I_1}{n} \to G$, where $\gexp{I_1}{n}$ denotes the $n$-dimensional hypercube graph.
The functor $\reali{\uvar} \colon \cSet \to \Top$, called the geometric realization, assigns the topological $n$-cube $[0,1]^n$ to each (formal) $n$-cube of $X$ and then glues these cubes together according to structure maps.
(We will of course give precise definitions of all these notions later in the paper.)

A reader familiar with \cite{babson-barcelo-longueville-laubenbacher} will recognize a change in notation: the functor $\gnerve[1]$ above was in \cite[\S4]{babson-barcelo-longueville-laubenbacher} denoted by $M_*$. This change is intentional, as we wish to emphasize that this is only the first in a sequence of functors $\Graph \to \cSet$ and natural transformations:
\begin{equation} \label{eq:sequence-of-N}
 \begin{tikzcd}
    \gnerve[1] \ar[r] & \gnerve[2] \ar[r] & \gnerve[3] \ar[r] & \cdots \text{.}
\end{tikzcd}
\end{equation}
It is this sequence and its colimit
\[ \gnerve := \colim (\gnerve[1] \to \gnerve[2] \to \gnerve[3] \to \cdots)\]
that we investigate.
Our main technical results can be summarized by the following theorem.
\begin{theorem*}[cf.~\cref{thm:main,thm:a_eq_cset_pi}] \leavevmode
  \begin{enumerate}
    \item For a graph $G$, the cubical set $\gnerve G$ is a Kan complex.
    \item The natural transformations in (\ref{eq:sequence-of-N}) are natural weak equivalences.
    \item For a based graph $(G, v)$, there is a natural group isomorphism $A_n(G, v) \iso \pi_n(\gnerve G, v)$.
  \end{enumerate}
\end{theorem*}
A few words of explanation are in order.
A Kan complex is a cubical set satisfying a certain lifting property making it particularly convenient for the development of homotopy theory.
In particular, in a companion paper \cite{carranza-kapulkin:homotopy-cubical}, we develop the theory of homotopy groups of Kan complexes and show that these agree with their topological analogues under the geometric realization functor.

This establishes the key ingredient required for the proof of the conjecture of Babson, Barcelo, de Longueville, and Laubenbacher, which asks for the commutativity of the outer square in the diagram:
\[ \begin{tikzcd}
     & \cSet_* \ar[rd, bend left, "|\uvar|"] & \\
    \Graph_* \ar[dr, bend right, swap, "A_n"] \ar[r, "\gnerve"] \ar[ru, bend left, "{\gnerve[1]}"] & \Kan_* \ar[d, "\pi_n"] \ar[u, hook] \ar[r, "|\uvar|"] & \Top_* \ar[ld, bend left, "\pi_n"]\\
    & \mathsf{Grp}  & 
\end{tikzcd} \]
By the theorem above, the two triangles on the left commute, with the upper one commuting up to a natural weak equivalence
(which is sent by the composite functor $\pi_n \circ |\uvar| \from \cSet_* \to \mathsf{Grp}$ to a natural isomorphism). 
The upper right triangle commutes on the nose and the lower right triangle commutes since it expresses the compatibility between homotopy groups of cubical sets and those of topological spaces, established in \cite{carranza-kapulkin:homotopy-cubical}.
Thus the $A$-groups of $(G, v)$ agree with those of $(\gnerve G, v)$, which in turn agree with those of $(\reali{\gnerve[1] G}, v)$, since the map $\gnerve[1] G \to \gnerve G$ is a weak equivalence.

\begin{theorem*}[Conjectured in \cite{babson-barcelo-longueville-laubenbacher}; cf.~\cref{thm:conj-bbdll}]
    There is a natural group isomorphism $A_n(G, v) \cong \pi_n(\reali{\gnerve[1]{G}}, v)$. 
\end{theorem*}

The insight of \cite{babson-barcelo-longueville-laubenbacher} cannot be overstated.
It is a priori not clear, or at the very least it was not clear to us, that the A-groups of a graph should correspond to the homotopy groups of any space.
The fact that the space can be obtained from a graph in such a canonical and simple way is what drove us to the problem and to the field of discrete homotopy theory.
The title of our paper, ``Cubical setting for discrete homotopy theory, revisited'' pays tribute to this insight by alluding to the title ``A cubical set setting for the A-theory of graphs'' of \cite[\S3]{babson-barcelo-longueville-laubenbacher}.

It is also worth noting that the case of $n=1$ of \cref{thm:conj-bbdll} was previously proven in \cite[Prop.~5.12]{barcelo-kramer-laubenbacher-weaver} and perhaps helped inspire the statement of the conjecture in the general case.

Our main theorem allows us to derive a few more consequences of interest in discrete homotopy theory.
The first of those is a strong form of the Hurewicz theorem for graphs.
The Hurewicz theorem relates the first non-trivial homotopy group of a sufficiently connected space to its homology.
In discrete homotopy theory, it relates the first non-trivial A-group of a graph to its reduced discrete homology, introduced in \cite{barcelo-capraro-white}.

\begin{theorem*}[Discrete Hurewicz Theorem; cf.~\cref{graph-hurewicz}]
    Let $n \geq 2$ and $(G,v)$ be a connected pointed graph.
    Suppose $A_i (G,v) = 0$ for all $i \in \{ 1, \dots, n-1 \}$.
    Then the Hurewicz map $A_n (G,v) \to \tilde{DH}_n (G, v)$ from the $n$-th A-group to the $n$-th reduced discrete homology group is an isomorphism.
\end{theorem*}

This generalizes the results of Lutz \cite[Thm.~5.10]{lutz}, who proves surjectivity of the Hurewicz map for a more restrictive class of graphs; and complements the result of Barcelo, Capraro, and White \cite[Thm.~4.1]{barcelo-capraro-white}, who prove the $1$-dimensional analogue of the Hurewicz theorem, namely that the Hurewicz map $A_1 (G,v) \to \tilde{DH}_1 (G, v)$ is surjective with kernel given by the commutator $[A_1 (G,v), A_1 (G,v)]$ subgroup.

Lastly, our main theorem allows us equip the category of graphs with additional structure making it amenable to techniques of abstract homotopy theory and higher category theory.

By our theorem, the functor $\gnerve \colon \Graph \to \cSet$ takes values in the full subcategory $\Kan$ of $\cSet$ spanned by Kan complexes.
This subcategory is known to carry the structure of a fibration category in the sense of Brown \cite{brown}.
In brief, a fibration category is a category equipped with two classes of maps: fibrations and weak equivalences, subject to some axioms.
Fibration categories are one of the main frameworks used in abstract homotopy theory (see, for instance, \cite{szumilo}).
We declare a map $f$ of graphs to be a fibration/weak equivalence if $\gnerve f$ is one in the fibration category of Kan complexes.
Our theorem guarantees that this gives a well-defined fibration category structure on $\Graph$ (\cref{thm:fibcat}), hence allowing for the use of techniques from abstract homotopy theory in discrete homotopy theory.
Furthermore, it follows that the weak equivalences of this fibration category are precisely graph maps inducing isomorphisms on all A-groups for all choices of the basepoint.

This also allows us to put the results of \cite[\S6]{babson-barcelo-longueville-laubenbacher} in the context of abstract homotopy theory by proving that the loop graph functor constructed there is an exact functor in the sense of fibration category theory (\cref{thm:graph_loopsp_exact}).

In a different direction, we observe that the functor $\gnerve \from \Graph \to \cSet$ is lax monoidal (\cref{thm:nerve_monoidal}).
The category of graphs is enriched over itself, meaning that the collection of graph maps between two graphs forms not merely a set, but a graph, and that the composition of graph maps defines a graph homomorphism.
Enrichment can be transferred along lax monoidal functors, which means that the category of graphs is, via $\gnerve$, canonically enriched over cubical sets, and, more precisely, over Kan complexes (\cref{thm:graph-locally-kan}).
This establishes a presentation of the $(\infty, 1)$-category of graphs, thus allowing for the use of techniques of higher category theory, developed extensively by Joyal and Lurie \cite{htt}, in discrete homotopy theory, via the cubical homotopy coherent nerve construction of \cite[\S2]{kapulkin-voevodsky}.

\textbf{Organization of the paper.}
This paper is organized as follows.
In \cref{sec:graphs,sec:cset}, we review the background on discrete homotopy theory and cubical sets, respectively.
In \cref{sec:nerves}, we explain the link between graphs and cubical sets by defining the functors $\gnerve[m]$ and $\gnerve$, and proving their basic properties.
The technical heart of the paper is contained in \cref{sec:main}, where we prove our main results.
In \cref{sec:consequences}, we proceed to deduce the consequences of our main theorem, as described above.

\textbf{Acknowledgements.}
We thank Bob Lutz for comments on an earlier draft of this paper and for the beautiful talks on discrete homotopy theory (at MSRI in Spring 2020 and at Western University in Fall 2020) got us interested in this area.
We are also very grateful to Yeonjoon Choi, Udit Mavinkurve, and Mohabat Tarkeshian, members of an informal seminar on discrete homotopy theory that we organized in Spring 2021.
In particular, Choi was the first to observe \cref{thm:nerve_monoidal}.
Finally, we would like to thank Eric Babson for suggesting to us the notion of a fiber bundle (cf.~\cref{def:fiber-bundle}).

During the work on this paper, the first author was partially supported by an NSERC Undergraduate Student Research Award and the second author was partially supported by an NSERC Discovery Grant.
We thank NSERC for its generosity.

\section{Discrete homotopy theory} \label{sec:graphs}

\subsection*{The category of graphs}
We define the category of simple undirected graphs without loops as a reflective subcategory of a presheaf category.

Let $\G$ be the category generated by the diagram
\[ \begin{tikzcd}
    V \ar[r, yshift=1ex, "s"] \ar[r, yshift=-1ex, "t"'] & \ar[l, "r" description] E \ar[loop right, "\sigma"]
\end{tikzcd} \]
subject to the identities
\[ \begin{array}{l l}
    rs = rt = \id[V] & \sigma^2 = \id[E] \\
    \sigma s = t & \sigma t = s \\
    r \sigma = r.
\end{array} \]
We write $\Ghat$ for the functor category $\Set^{\G^\op}$.

For $G \in \Ghat$, we write $G_V$ and $G_E$ for the sets $G(V)$ and $G(E)$, respectively.
Explicitly, such a functor consists of sets $G_V$ and $G_E$ together with the following functions between them
\[ \begin{tikzcd}
    G_V \ar[r, "Gr" description]  &  G_E \ar[l, yshift=1ex, "Gs"'] \ar[l, yshift=-1ex, "Gt"] \ar[loop right, "G\sigma"]
\end{tikzcd} \]
subject to the dual versions of identities in $\G$.

\begin{definition} \label{def:graph}
    A \emph{graph} is a functor $G \in \Ghat$ such that the map $(Gs, Gt) \colon G_E \to G_V \times G_V$ is a monomorphism.
\end{definition}

Let $\Graph$ denote the full subcategory of $\Ghat$ spanned by graphs.

In more concrete terms, a graph $G$ consists of a set $G_V$ of vertices and a set $G_E$ of `half-edges.'
A half-edge $e \in G_E$ has source and target, and these are given by the maps $Gs$ and $Gt$, respectively.
Each half-edge is paired with its other half via the map $G\sigma \colon G_E \to G_E$.
Note that the edges paired by $G \sigma$ have swapped source and target, making the pair (i.e.~whole edge) undirected.
The map $Gr \colon G_V \to G_E$ takes a vertex to an edge whose source and target is that vertex (i.e.~a loop).
Finally, the condition that $(Gs, Gt)$ is a monomorphism ensures that there is at most one (whole) edge between any two vertices.
This is equivalent to specifying a binary ``incidence'' relation on $G_V$ which is reflexive and symmetric.

A map of graphs $f \colon G \to H$ is a natural transformation between such functors.
However, since $(Hs, Ht)$ is a monomorphism, such a map is completely determined by a function $f_V \colon G_V \to H_V$ that preserves incidence relation.

We may therefore assume that our graphs have no loops, but the maps between them, rather than merely preserving edges, are allowed to contract them to a single vertex.
That is, a graph map $f \colon G \to H$ is a function $f \colon G_V \to H_V$ such that if $v, w \in G_V$ are connected by an edge, then either $f(v)$ and $f(w)$ are connected by an edge or $f(v) = f(w)$.

%We think of a graph as a collection of vertices where any two distinct vertices may or may not be connected by an edge.
%When depicting a graph, vertices will be depicted by points, occassionally identified by a label or color.
%If two vertices are connected by an edge then we draw a line between the two points.
\begin{figure}[ht]
    \centering
    \begin{tikzpicture}[node distance=20pt]
        \node[vertex] (A) {};
        \node[vertex] (B) [below right=of A] {};
        \node[vertex] (C) [below left=of A] {};
        \node[vertex] (D) [below=of A] {};
        \node[vertex] (E) [below left=of C] {};

        \draw (A) to (B) to (D) to (E);
        \draw (A) to (C) to (D);
    \end{tikzpicture}
    \caption{An example depiction of a graph.}
\end{figure}
%A map between graphs $f \from G \to H$ is a function $G_V \to H_V$ between the sets of vertices such that if $v, w \in G$ are connected by an edge then either $f(v) = f(w)$ or $f(v)$ and $f(w)$ are connected by an edge.

%For a graph $G \in \Ghat$, an element $v \in G_V$ is a \emph{vertex} of $G$, written $v \in G$.
%A vertex $v \in G$ is \emph{connected} to a vertex $w \in G$ if there exists an element $e \in G_E$ such that $(Gs)(e) = v$ and $(Gt)(e) = w$.
%If such an $e$ exists then $(G\sigma)(e) \in G_E$ witnesses that $w$ is connected to $v$ by the identities of $\G$.
%Thus, the relation of connectedness is symmetric.
%We refer to the pair $\langle e, (G\sigma)(e) \rangle \in G_E$ as an \emph{edge} in $G$.
%As the map $(Gs, Gt) \from G_E \to G_V \times G_V$ is a monomorphism, there is at most one pair $\langle e, (G\sigma)(e) \rangle$ connecting any two vertices.
%The identites of $\G$ imply every vertex $v \in G$ is connected to itself by the pair $\langle (Gr)(v), (Gr)(v) \rangle$.

%We refer to a map $f \from G \to H$ as a \emph{graph map}.
%The data of a graph map consists of a function $f_V \from G_V \to H_V$ on vertices and a function $f_E \from G_E \to H_E$ on edges.
%Formally, $f$ is a natural transformation between functors $\G^\op \to \Set$.
%Naturality ensures that if $\langle e, (G\sigma)(e) \rangle$ is an edge between two vertices $v, w \in G$ then $f_E((G\sigma)(e)) = (H\sigma)(f_E(e))$ and $\langle f_E(e), f_E((G\sigma)(e) \rangle$ is an edge between $f_V(v), f_V(w) \in H$.

\begin{proposition} \label{thm:graph_complete} \leavevmode
    \begin{enumerate}
        \item The inclusion $\Graph \ito \Ghat$ admits a left adjoint.
        \item The category $\Graph$ is (co)complete.
        \item The functor $\Graph \to \Set$ mapping a graph $G$ to its set of vertices $G_V$ admits both adjoints.
    \end{enumerate}
\end{proposition}
\begin{proof}
    \begin{enumerate}
        \item The left adjoint is $\im \colon \Ghat \to \Graph$ given by
        \[ (\im G)_V = G_V \quad (\im G)_E = (Gs, Gt) (G_E) \]
        where $(Gs, Gt) (G_E)$ is the image of $G_E$ under the map $(Gs, Gt) \from G_E \to G_V \times G_V$.
        \item The category $\Ghat$ is (co)complete as a presheaf category.
        The conclusion follows from (1) as $\Graph$ is a reflective subcategory of a (co)complete category.
        \item The left adjoint $\Set \to \Graph$ takes a set $A$ to the discrete graph with vertex set $A$.
        The right adjoint takes a set $A$ to the complete graph with vertex set $A$. \qedhere
    \end{enumerate}
\end{proof}
\begin{remark} \label{rem:graph_lim_colim}
    \cref{thm:graph_complete} gives a procedure for constructing limits and colimits in $\Graph$.
    Given a diagram $F \from J \to \Graph$, the set of vertices of $\lim F$ is the limit $\lim UF$ of the diagram $UF \from J \to \Set$.
    The set of edges is the largest set such that the limit projections $\lim F \to F_j$ are graph maps.
    The set of vertices of $\colim F$ is the colimit $\colim UF$ of the diagram $UF \from J \to \Set$ and the set of edges is the smallest set such that the colimit inclusions $F_j \to \colim F$ are graph maps.
\end{remark}

\subsection*{Examples of graphs}
\begin{definition}
    For $m \geq 0$,
    \begin{enumerate}
        \item the \emph{$m$-interval} $I_m$ is the graph which has
        \begin{itemize}
            \item as vertices, integers $0 \leq i \leq m$;
            \item an edge between $i$ and $i+1$.
        \end{itemize}
        \item the \emph{$m$-cycle} $C_m$ is the graph which has
        \begin{itemize}
            \item as vertices, integers $0 \leq i \leq m-1$;
            \item an edge between $i$ and $i + 1$ and an edge between $m-1$ and $0$.
        \end{itemize}
        \item the \emph{infinite interval} $I_\infty$ is the graph which has
        \begin{itemize}
            \item as vertices, integers $i \in \Z$;
            \item an edge between $i$ and $i+1$ for all $i \in \Z$.
        \end{itemize}
    \end{enumerate}
\end{definition}

% For $m \geq 0$, we define the graph $C_m$ to be the coequalizer of the maps $0, m \from I_0 \to I_m$ which pick out the 0-th and $m$-th vertex, respectively.
% Explicitly, the graph $C_m$ has

\begin{figure}[ht]
    \centering
    \begin{tikzpicture}[node distance=20pt]
        % vertices
        \node[vertex, minimum size=12pt] (0) {0};
        \node[vertex, minimum size=12pt] (1) [right=of 0] {1};
        \node[vertex, minimum size=12pt] (2) [right=of 1] {2};
        \node[vertex, minimum size=12pt] (3) [right=of 2] {3};

        \node[vertex, minimum size=12pt] (1c) [right=50pt of 3] {0};
        \node[vertex, minimum size=12pt] (2c) [above right=of 1c] {1};
        \node[vertex, minimum size=12pt] (3c) [below right=of 2c] {2};
        % edges
        \draw (0) to (1)
              (1) to (2)
              (2) to (3);
        \draw (1c) to (2c) to (3c) to (1c);
    \end{tikzpicture}
    \caption{The graphs $I_3$ and $C_3$, respectively.}
\end{figure}

\begin{remark}
  The graphs $I_0$ and $I_1$ are representable when regarded as functors $\G^\op \to \Set$, represented by $V$ and $E$, respectively.
\end{remark}

For $m \geq 0$, we have a map $l \from I_{m+1} \to I_{m}$ defined by 
\[ l(v) = \begin{cases}
    0 & \text{if } v = 0 \\
    v-1 & \text{otherwise.}
\end{cases} \]
As well, we have a map $r \from I_{m+1} \to I_m$ defined by
\[ r(v) = \begin{cases}
    m & \text{if } v = m + 1 \\
    v & \text{otherwise.}
\end{cases} \]

We write $c \from I_{m+2} \to I_m$ for the composite $lr = rl$.
Explicitly, this map is defined by
\[ c(v) = \begin{cases}
    0 & \text{if } v = 0 \\
    m & \text{if } v = m + 2 \\
    v - 1 & \text{otherwise.}
\end{cases} \]

We show the inclusion $\Graph \ito \Ghat$ preserves filtered colimits and use this to show all finite graphs are compact, i.e.~if $G$ is a finite graph then the functor $\Graph(G, \uvar) \from \Graph \to \Set$ preserves filtered colimits.
\begin{proposition} \label{thm:graph_ghat_colim_filtered}
    The inclusion $\Graph \ito \Ghat$ preserves filtered colimits.
\end{proposition}
\begin{proof}
    Fix a filtered category $J$ and a diagram $D \from J \to \Graph$.
    Let $i$ denote the inclusion $\Graph \ito \Ghat$.
    Recall that $\colim D$ is computed by $\im (\colim (iD))$.
    It suffices to show $\colim (iD) \in \Ghat$ is a graph, since then the unit map $\colim (iD) \to i \im (\colim (iD))$ is an isomorphism $\colim (iD) \cong i (\colim D)$ natural in $D$.

    Let $\lambda \from iD \to \colim(iD)$ denote the colimit cone.
    Suppose two edges $e, e' \in \colim(iD)_E$ have the same source and target.
    Regarding $e, e'$ as maps $I_1 \to \colim(iD)$, these maps factor as
    \[ \begin{tikzcd}
        I_1 \ar[rr, "{e, e'}"] \ar[rd, "{\overline{e}, \overline{e}'}"'] & {} & \colim(iD) \\
        {} & iDx \ar[ur, "\lambda_x"'] & {}
    \end{tikzcd} \]
    for some $x \in J$ since $I_1 \in \Ghat$ is representable.
    Let $s, s' \in (iDx)_V$ denote the sources of $\overline{e}, \overline{e}'$ and $t, t' \in (iDx)_V$ denote the targets, respectively.
    Using an explicit description of the colimit (\cref{rem:graph_lim_colim}), since $\lambda_x(s) = \lambda_x(s')$ and $\lambda_x(t) = \lambda_x(t')$, there exist arrows $f, g \from y \to x$ and $h, k \from z \to x$ with vertices $v \in iDy$ and $w \in iDz$ such that
    \[ \begin{array}{c c c c}
        iDf(v) = s & iDg(v) = s' & iDh(w) = t & iDk(w) = t'.
    \end{array} \]
    As $J$ is filtered, there exists an arrow $l \from x \to w$ in $J$ such that $lf = lg$ and $lh = lk$.
    This implies the edges $iDl(\overline{e}), iDl(\overline{e}') \in (iDw)_E$ have the same source and target.
    As $iDw$ is a graph, it follows that $\overline{e} = \overline{e}'$, thus $e = e'$.
\end{proof}
\begin{corollary} \label{thm:fin_graph_compact}
    For a finite graph $G$, the functor $\Graph(G, \uvar) \from \Graph \to \Set$ preserves filtered colimits.
\end{corollary}
\begin{proof}
    Given a filtered category $J$ and a diagram $D \from J \to \Graph$, we have a natural isomorphism
    \[ \Graph(G, \colim D) \cong \Ghat(G, \colim D) \]
    by \cref{thm:graph_ghat_colim_filtered} since $\Graph$ is a full subcategory.
    This then follows since $G \in \Ghat$ is a finite colimit of representable presheaves. 
\end{proof}

\subsection*{Monoidal structure on the category of graphs}
Define a functor $\gtimes \from \G \times \G \to \Ghat$ by
\[ \begin{array}{c c}
    (V, V) \mapsto I_0 & (V, E) \mapsto I_1 \\
    (E, V) \mapsto I_1 & (E, E) \mapsto C_4 
\end{array} \]
Left Kan extension along the Yoneda embedding yields a monoidal product $\gtimes \from \Ghat \times \Ghat \to \Ghat$.
\[ \begin{tikzcd}
    \G \times \G \ar[d] \ar[r] & \Ghat \\
    \Ghat \times \Ghat \ar[ur, dotted, "\gtimes"']
\end{tikzcd} \]
Note that $\Graph$ is closed with respect to this product i.e.~if $G, H \in \Graph$ then $G \gtimes H \in \Ghat$ is a graph.
Thus, this product descends to a monoidal product $\gtimes \from \Graph \times \Graph \to \Graph$, called the \emph{cartesian product}.

Explicitly, the graph $G \gtimes H$ has
\begin{itemize}
    \item as vertices, pairs $(v, w)$ where $v \in G$ is a vertex of $G$ and $w \in H$ is a vertex of $H$;
    \item an edge from $(v, w)$ to $(v', w')$ if either $v = v'$ and $w$ is connected to $w'$ in $H$ or $w = w'$ and $v$ is connected to $v'$ in $G$. 
\end{itemize}
\begin{definition}
    Let $G$ and $H$ be graphs.
    The graph $\ghom{G}{H}$ has
    \begin{itemize}
        \item as vertices, morphisms $G \to H$ in $\Graph$
        \item an edge from $f$ to $g$ if there exists $H \colon G \gtimes I_1 \to H$ such that $\restr{H}{G \gtimes \{ 0 \}} = f$ and $\restr{H}{G \gtimes \{ 1 \}} = g$.
    \end{itemize}
\end{definition}

This structure makes the category of graphs into a closed symmetric monoidal category.

\begin{proposition}
    $(\Graph, \gtimes, I_0, \ghom{\uvar}{\uvar})$ is a closed symmetric monoidal category. \noproof
\end{proposition}

\subsection*{Homotopy theory of graphs}
We now review the basics of discrete homotopy theory.
Our treatment is brief and more categorically-oriented, but the reader can find full details in any of the references \cite{malle,barcelo-kramer-laubenbacher-weaver,babson-barcelo-longueville-laubenbacher,barcelo-laubenbacher}.
\begin{definition}
    Let $f, g \from G \to H$ be graph maps.
    An \emph{A-homotopy} (or just a \emph{homotopy}) from $f$ to $g$ is a map $\eta \from G \gtimes I_m \to H$ for some $m \geq 0$ such that $\restr{\eta}{G \gtimes \{ 0 \}} = f$ and $\restr{\eta}{G \gtimes \{ m \}} = g$.
    % $\eta(v,0) = fv$ and $\eta(v,1) = gv$ for any vertex $v \in G$.
\end{definition}
Note that one needs to allow the parameter $m$ appearing above to vary, as otherwise the notion of homotopy is not transitive.
This is reminiscent of the notion of \emph{Moore path} in topological spaces, which is a path parameterized by an interval of arbitrary length \cite{may,barthel-riehl,berg-garner}.
\begin{proposition}[{\cite[Prop.~2.1]{malle}}]
    For graphs $G$ and $H$, homotopy is an equivalence relation on the set of graphs maps $G \to H$. \noproof
\end{proposition}
We also define the A-homotopy groups of a graph.
As in topological spaces, this requires the definition of based graph maps and based homotopies between them.
\begin{definition}
    Let $A \ito G$ and $B \ito H$ be graph monomorphisms.
    A \emph{relative graph map}, denoted $(G, A) \to (H, B)$, is a morphism from $A \ito G$ to $B \ito H$ in the arrow category $\arr{\Graph}$, where $[1]$ denotes the poset $\{ 0 \leq 1 \}$ viewed as a category.
\end{definition}
Explicitly, this data consists of maps $G \to H$ and $A \to B$ such that the following square commutes.
\[ \begin{tikzcd}
    A \ar[d, hook] \ar[r] & B \ar[d, hook] \\
    G \ar[r] & H
\end{tikzcd} \]
That is, $A$ is a subgraph of $G$ and the map $A \to B$ is the restriction of the map $G \to H$ to $A$, whose image is contained in the subgraph $B$.

In the context of relative graph maps, we denote a monomorphism $A \ito G$ by $(G, A)$, supressing the data of the map itself.
An exception to this is for monomorphisms of the form $I_0 \to G$. 
This is exactly the data of a pointed graph, which we denote by $(G, v)$, where $v$ is the unique vertex in the image of the map $I_0 \to G$.

For a relative graph map $(f, g) \from (G, A) \to (H, B)$, the map $g$ is uniquely determined by $f$: if a map $h \from A \to B$ also forms a commutative square with $f$ then $g = h$ as the map $B \ito Y$ is monic.
Thus, we denote a relative graph map by $f \from (G, A) \to (H, B)$.
We additionally write $f \from G \to H$ for the bottom map in the square and $\restr{f}{A} \from A \to B$ for the top map.

% Given a relative graph map $f \from (G, A) \to (H, B)$, we write $\restr{f}{A}$ for the map from $A$ to $B$.
We define relative homotopies between relative graph maps as follows.
\begin{definition}
    Let $f, g \from (G, A) \to (H, B)$ be relative graph maps.
    A \emph{relative homotopy} from $f$ to $g$ is a relative map $(G \gtimes I_m, A \gtimes I_m) \to (H, B)$ for some $m \geq 0$ such that
    \begin{itemize}
        \item the map $A \gtimes I_m \to B$ is a homotopy from $\restr{f}{A}$ to $\restr{g}{A}$;
        \item the map $G \gtimes I_m \to H$ is a homotopy from $f$ to $g$.
    \end{itemize}
\end{definition}
Explicitly, a relative homotopy from $f$ to $g$ consists of 
\begin{itemize}
    \item a homotopy $\eta \from G \gtimes I_m \to H$ from $f$ to $g$
    \item a homotopy $\restr{\eta}{A} \from A \gtimes I_m \to B$ from $\restr{f}{A}$ to $\restr{g}{A}$
\end{itemize}
such that the following square commutes.
\[ \begin{tikzcd}
    A \gtimes I_m \ar[d, hook] \ar[r, "\restr{\eta}{A}"] & B \ar[d, hook] \\
    G \gtimes I_m \ar[r, "\eta"] & H
\end{tikzcd} \]
\begin{remark}
    For graph maps $f, g \from G \to H$, a path of length $m$ in $\ghom{G}{H}$ from the vertex $f$ to the vertex $g$ is exactly a homotopy $G \gtimes I_m \to H$ from $f$ to $g$.
    For relative graph maps $f, g \from (G, A) \to (H, B)$, a path of length $m$ from $f$ to $g$ in the pullback graph
    \[ \begin{tikzcd}
        \bullet \ar[r] \ar[d] \ar[rd, phantom, "\ulcorner" very near start, yshift=-1ex] & \ghom{A}{B} \ar[d, hook] \\
        \ghom{G}{H} \ar[r] & \ghom{A}{H}
    \end{tikzcd} \]
    is exactly a relative homotopy from $f$ to $g$.
\end{remark}
It follows that relative homotopy is an equivalence relation on relative graph maps.
\begin{proposition}%[\cite{barcelo-kramer-laubenbacher-weaver}]
    Relative homotopy is an equivalence relation on the set of relative graph maps $(G, A) \to (H, B)$. \noproof
\end{proposition}

Given a relative graph map $f \from (G, A) \to (H, B)$, if the subgraph $B$ consists of a single vertex $v$ then we refer to $f$ as a graph map \emph{based at $v$} or a \emph{based} graph map.
We refer to a homotopy between two such maps as a \emph{based} homotopy.

\begin{definition}
    Let $G$ be a graph and $n \geq 1$.
    For $i = 0, \dots, n$ and $\varepsilon = 0, 1$, a map $f \from \gexp{I_\infty}{n} \to G$ is \emph{stable in direction $(i,\varepsilon)$} if there exists $M \geq 0$ so that for $v_i > M$, we have 
    \[ f(v_1, \dots, v_{i-1}, (2\varepsilon - 1)v_i, v_{i+1} \dots, v_n) = f(v_1, \dots, v_{i-1}, (2\varepsilon - 1)M, v_{i+1}, \dots, v_n). \] 
\end{definition}
In words, a map $f \from \gexp{I_\infty}{n} \to G$ is stable in direction $(i, 0)$ if it becomes constant (with respect to change in the $i$-th coordinate) once the $i$-th coordinate is sufficiently large in the negative direction.
It is stable in direction $(i, 1)$ if it becomes constant (again, with respect to change in the $i$-th coordinate) once the $i$-th coordinate is sufficiently large in the positive direction.

For $n, M \geq 0$, let $\gexp{I_{\geq M}}{n}$ denote the subgraph of $\gexp{I_\infty}{n}$ consisting of vertices $(v_1, \dots, v_n)$ such that $|v_i| \geq M$ for some $i = 1, \dots n$.
Given a based graph map $f \from (\gexp{I_\infty}{n}, \gexp{I_{\geq M}}{n}) \to (G,v)$ we may also regard $f$ as a based graph map $(\gexp{I_\infty}{n}, \gexp{I_{\geq K}}{n}) \to (G,v)$ for any $K \geq M$.
This gives a notion of based homotopy between maps $(\gexp{I_\infty}{n}, \gexp{I_{\geq M}}{n}) \to (G, v)$ which, for some $M \geq 0$, are based at $v$.
\begin{proposition}[{\cite[Prop.~3.2]{malle}}]
    Based homotopy is an equivalence relation on the set of based maps 
    \[ \{ (\gexp{I_\infty}{n}, \gexp{I_{\geq M}}{n}) \to (G, v) \mid M \geq 0 \}. \eqno{\qed} \]
\end{proposition}

\begin{definition}
    Let $n \geq 0$ and $v \in G$ be a vertex of a graph $G$.
    The $n$-th \emph{A-homotopy group} of $G$ at $v$ is the set of based homotopy classes of maps $(\gexp{I_\infty}{n}, \gexp{I_{\geq M}}{n}) \to (G,v)$ based at $v$ for some $M \geq 0$.
\end{definition}
Let $n \geq 1$ and $i = 1, \dots, n$.
Given $f \from (\gexp{I_\infty}{n}, \gexp{I_{\geq M}}{n}) \to (G, v)$ and $g \from (\gexp{I_\infty}{n}, \gexp{I_{\geq M'}}{n}) \to (G, v)$, we define a binary operation $f \cdot_i g \from (\gexp{I_\infty}{n}, \gexp{I_{\geq M + M'}}{n}) \to (G, v)$ by
\[ (f \cdot_i g)(v_1, \dots, v_n) := \begin{cases}
    f(v_1, \dots, v_n) & v_i \leq M \\
    g(v_1, \dots, v_i - M - M', \dots, v_n) & v_i > M.
\end{cases} \]
This induces a group operation on homotopy groups $\cdot_i \from A_n(G, v) \times A_n(G, v) \to A_n(G, v)$.
For $n \geq 2$ and $1 \leq i < j \leq n$, it is straightforward to construct a homotopy witnessing that
\[ [[f] \cdot_i [g]] \cdot_j [[f] \cdot_i [g]] = [[f] \cdot_j [g]] \cdot_i [[f] \cdot_j [g]] \]
for any $[f], [g] \in A_n(G,v)$.
The Eckmann-Hilton argument gives $[f] \cdot_i [g] = [f] \cdot_j [g]$ and that this operation is abelian.

\subsection*{Path and loop graphs}
\begin{definition}
    For a graph $G$, we define the \emph{path graph} $PG$ to be the induced subgraph of $\ghom{I_\infty}{G}$ consisting of maps which stabilize in the $(1,0)$ and $(1,1)$ directions.
\end{definition}
As vertices of $PG$ are paths that stabilize, we have graph maps $\face{}{1,0}, \face{}{1,1} \from PG \to G$ which send a vertex $v \from I_\infty \to G$ to its left and right endpoints, respectively.
% Note that the path graph is given by the following colimit.

\begin{proposition} \label{thm:path_graph_colim}
    For a graph $G$, we have an isomorphism
    \[ PG \cong \colim \left( \begin{tikzcd}
        \ghom{I_1}{G} \ar[r, "{l^*}"] & \ghom{I_2}{G} \ar[r, "r^*"] & \ghom{I_3}{G} \ar[r, "l^*"] & \dots
    \end{tikzcd} \right) \]
    natural in $G$.
\end{proposition}
\begin{proof}
For each $m \geq 1$, we write $m = 2k + b$ where $k \geq 0$ and $b \in \{0, 1\}$ and define a graph map $I_\infty \to I_m$ by 
\[ v \mapsto \begin{cases}
    0 & v \leq -k \\
    k + v & -k \leq v \leq k+b \\
    m & v \geq k+b.
\end{cases} \]
Geometrically, this function maps the subinterval $[-k, k+b]$ surjectively onto $I_m$, and collapses all other vertices to the endpoints.
By pre-composition, this induces a cone
% \[ \begin{tikzcd}
%     {} & {} & I_\infty \ar[ld] \ar[d] \ar[rd] & {} \\
%     \dots \ar[r, "l"] & I_3 \ar[r, "r"] & I_2 \ar[r, "l"] & I_1
% \end{tikzcd} \]
% and induce a cone
\[ \begin{tikzcd}
    \ghom{I_1}{G} \ar[r, "l^*"] \ar[rd] & \ghom{I_2}{G} \ar[r, "r^*"] \ar[d] & \ghom{I_3}{G} \ar[r, "l^*"] \ar[ld] & \dots \\
    {} & PG & {} & {}
\end{tikzcd} \]
which one verifies is a colimit cone.
\end{proof}

\begin{definition}
  For a pointed graph $(G, v)$, the \emph{loop graph} $\loopsp(G, v)$ is the subgraph of $PG$ of paths $I_\infty \to G$ whose left and right endpoints are $v$.
\end{definition}
From the definition, it is immediate that
\begin{proposition} \label{thm:graph_loopsp_pb}
  For a pointed graph $(G, v)$, the square
  \[ \begin{tikzcd}
    \loopsp(G, v) \ar[r] \ar[d] \ar[rd, phantom, "\ulcorner" very near start] & PG \ar[d, "{(\face{}{1,0}, \face{}{1,1})}"] \\
    I_0 \ar[r, "{(v, v)}"] & G \times G
  \end{tikzcd} \]
  is a pullback. \noproof
\end{proposition}
The loop graph of a pointed graph $(G, v)$ has a distinguished vertex which is the constant path at $v$.
This gives an endofunctor $\loopsp \from \Graph_* \to \Graph_*$.
From this, we define the notion of $n$-th loop graphs.
\begin{definition}
    For $n \geq 0$, we define the \emph{$n$-th loop graph} to be
    \[ \loopsp[n](X,x) := \begin{cases}
        (X,x) & n = 0 \\
        \loopsp(\loopsp[n-1](X,x)) & \text{otherwise.}
    \end{cases} \]
\end{definition}
In \cite{babson-barcelo-longueville-laubenbacher}, it is shown that the $n$-th homotopy groups of a graph correspond to the connected components of $\loopsp[n](G,v)$.
\begin{proposition}[{\cite[Prop.~7.4]{babson-barcelo-longueville-laubenbacher}}]
    For $n \geq 0$, we have an isomorphism
    \[ A_n(G, v) \cong A_0(\loopsp[n](G, v)). \eqno{\qed} \]
\end{proposition}

\section{Cubical sets and their homotopy theory} \label{sec:cset}

\subsection*{Cubical sets}
We begin by defining the box category $\Box$.
As used in this paper, the box category will include both positive and negative connections.
This variant of the box category was introduced in \cite{brown-higgins:algebra-of-cubes} and was later used in \cite{doherty-kapulkin-lindsey-sattler} to model $(\infty, 1)$-categories.
The objects of $\Box$ are posets of the form $[1]^n = \{ 0 \leq 1 \}^n$ and the maps are generated (inside the category of posets) under composition by the following four special classes:
\begin{itemize}
  \item \emph{faces} $\partial^n_{i,\varepsilon} \colon [1]^{n-1} \to [1]^n$ for $i = 1, \ldots , n$ and $\varepsilon = 0, 1$ given by:
  \[ \partial^n_{i,\varepsilon} (x_1, x_2, \ldots, x_{n-1}) = (x_1, x_2, \ldots, x_{i-1}, \varepsilon, x_i, \ldots, x_{n-1})\text{;}  \]
  \item \emph{degeneracies} $\sigma^n_i \colon [1]^n \to [1]^{n-1}$ for $i = 1, 2, \ldots, n$ given by:
  \[ \sigma^n_i ( x_1, x_2, \ldots, x_n) = (x_1, x_2, \ldots, x_{i-1}, x_{i+1}, \ldots, x_n)\text{;}  \]
  \item \emph{negative connections} $\gamma^n_{i,0} \colon [1]^n \to [1]^{n-1}$ for $i = 1, 2, \ldots, n-1$ given by:
  \[ \gamma^n_{i,0} (x_1, x_2, \ldots, x_n) = (x_1, x_2, \ldots, x_{i-1}, \max\{ x_i , x_{i+1}\}, x_{i+2}, \ldots, x_n) \text{.} \]
  \item \emph{positive connections} $\gamma^n_{i,1} \colon [1]^n \to [1]^{n-1}$ for $i = 1, 2, \ldots, n-1$ given by:
  \[ \gamma^n_{i,1} (x_1, x_2, \ldots, x_n) = (x_1, x_2, \ldots, x_{i-1}, \min\{ x_i , x_{i+1}\}, x_{i+2}, \ldots, x_n) \text{.} \]
\end{itemize}

These maps obey the following \emph{cubical identities}:

\[ \begin{array}{l l}
    \partial_{j, \varepsilon'} \partial_{i, \varepsilon} = \partial_{i+1, \varepsilon} \partial_{j, \varepsilon'} \quad \text{for } j \leq i; & 
    \sigma_j \partial_{i, \varepsilon} = \begin{cases}
        \partial_{i-1, \varepsilon} \sigma_j & \text{for } j < i; \\
        \id                                                       & \text{for } j = i; \\
        \partial_{i, \varepsilon} \sigma_{j-1} & \text{for } j > i;
    \end{cases} \\
    \sigma_i \sigma_j = \sigma_j \sigma_{i+1} \quad \text{for } j \leq i; &
    \gamma_{j,\varepsilon'} \gamma_{i,\varepsilon} = \begin{cases}
    \gamma_{i,\varepsilon} \gamma_{j+1,\varepsilon'} & \text{for } j > i; \\
    \gamma_{i,\varepsilon}\gamma_{i+1,\varepsilon} & \text{for } j = i, \varepsilon' = \varepsilon;
    \end{cases} \\
    \gamma_{j,\varepsilon'} \partial_{i, \varepsilon} = \begin{cases} 
        \partial_{i-1, \varepsilon} \gamma_{j,\varepsilon'}   & \text{for } j < i-1 \text{;} \\
        \id                                                         & \text{for } j = i-1, \, i, \, \varepsilon = \varepsilon' \text{;} \\
        \partial_{j, \varepsilon} \sigma_j         & \text{for } j = i-1, \, i, \, \varepsilon = 1-\varepsilon' \text{;} \\
        \partial_{i, \varepsilon} \gamma_{j-1,\varepsilon'} & \text{for } j > i;
    \end{cases} &
    \sigma_j \gamma_{i,\varepsilon} = \begin{cases}
        \gamma_{i-1,\varepsilon} \sigma_j  & \text{for } j < i \text{;} \\
        \sigma_i \sigma_i           & \text{for } j = i \text{;} \\
        \gamma_{i,\varepsilon} \sigma_{j+1} & \text{for } j > i \text{.} 
    \end{cases}
\end{array} \]

This category enjoys many good properties, making it suitable for modelling homotopy theory.
Formally speaking, these properties can be encapsulated in saying that $\boxcat$ is an Eilenberg--Zilber category, which in particular implies that the object $[1]^n$ has no non-identity automorphisms.
However, we will not explicitly rely on the notion of Eilenberg--Zilber categories.

A \emph{cubical set} is a presheaf $X \colon \boxcat^\op \to \Set$.
A \emph{cubical map} is a natural transformation of such presheaves.
We write $\cSet$ for the category of cubical sets and cubical maps.
%\begin{definition} 
%    \begin{enumerate}
%        \item The \emph{category of cubical sets}, denoted $\cSet$, is the functor category $\fcat{\boxcat^\op}{\Set}$.
%        \item A \emph{cubical set} is an object in $\cSet$.
%        \item A \emph{cubical map} is a morphism in $\cSet$.
%        \item An \emph{$n$-cube} of a cubical set $X$ is an element of the set $X ([1]^n)$.
%    \end{enumerate}
%\end{definition}

Given a cubical set $X$, we write $X_n$ for the value of $X$ at $[1]^n$ and refer to the elements of $X_n$ as \emph{$n$-cubes} of $X$.
We write cubical operators on the right, e.g.~given an $n$-cube $x \in X_n$ of $X$, we write $x\face{}{1,0}$ for the $\face{}{1,0}$-face of $x$.

By a \emph{degenerate cube}, we always mean a cube that is in the image of a degeneracy or a connection map.
This nomenclature is borrowed from the theory of Reedy categories, where one can speak abstractly of degenerate elements in a presheaf (\cite{riehl-verity:theory-and-practice-of-reedy-categories}, \cite[\S 5.2]{hovey}, \cite[Ch.~15]{hirschhorn}).
\begin{definition}
    Let $n \geq 0$.
    \begin{itemize}
        \item The \emph{combinatorial $n$-cube} $\cube{n}$ is the representable functor $\boxcat(-, [1]^n) \from \boxcat^\op \to \Set$;
        \item The \emph{boundary of the $n$-cube} $\bd \cube{n}$ is the subobject of $\cube{n}$ defined by
        \[ \bd \cube{n} := \bigcup\limits_{\substack{j=0,\dots,n \\ \eta = 0, 1}} \image \face{}{j,\eta}. \]
        \item When $n \geq 1$, given $i = 0, \dots, n$ and $\varepsilon = 0, 1$, the \emph{$(i,\varepsilon)$-open box} $\dfobox$ is the subobject of $\bd \cube{n}$ defined by
        \[ \dfobox := \bigcup\limits_{(j,\eta) \neq (i,\varepsilon)} \image \face{}{j,\eta}. \]
    \end{itemize}
\end{definition}
Observe $\cube{0} \in \cSet$ is the terminal object in $\cSet$.

\begin{example}
    Define a functor $\boxcat \to \Top$ from the box category to the category of topological spaces which sends $[1]^n$ to $[0, 1]^n$ where $[0, 1]$ is the unit interval.
    The face and degeneracy maps are sent to the face inclusion and product projection maps, respectively.
    The negative connection $\conn{}{i,0}$ is sent to the map $[0, 1]^n \to [0, 1]^{n-1}$ defined by
    \[ (x_1, \dots, x_n) \mapsto (x_1, \dots, x_{i-1}, \max(x_i, x_{i+1}), x_{i+2}, \dots, x_n) \]
    and the image of the positive connection $\conn{}{i,1}$ is defined analogously.

    Left Kan extension along the Yoneda embedding gives the \emph{geometric realization} functor $\reali{\uvar} \from \cSet \to \Top$.
    \[ \begin{tikzcd}[column sep = large]
        \boxcat \ar[r, "{[1]^n \mapsto [0, 1]^n}"] \ar[d] & \Top \\
        \cSet \ar[ur, "\reali{\uvar}"']
    \end{tikzcd} \]
    This functor is left adjoint to the \emph{singular cubical complex} functor $\sing{} \from \Top \to \cSet$ defined by
    \[ (\sing{S})_n := \Top ([0, 1]^n, S). \]
\end{example}

Define a functor $\gprod \from \boxcat \times \boxcat \to \boxcat$ on the cube category which sends $([1]^m, [1]^n)$ to $[1]^{m+n}$.
Postcomposing with the Yoneda embedding and left Kan extending gives a monoidal product on cubical sets.
\[ \begin{tikzcd}
    \boxcat \times \boxcat \ar[r] \ar[d] & \boxcat \ar[r] & \cSet \\
    \cSet \times \cSet \ar[urr, "\gprod"']
\end{tikzcd} \]
This is the \emph{geometric product} of cubical sets.
Although, for $[1]^m, [1]^n \in \boxcat$, there is an isomorphism $[1]^m \gprod [1]^n \cong [1]^n \gprod [1]^m$, this isomorphism is not natural.
As a result, the geometric product of cubical sets is not symmetric, i.e.~$X \gprod Y$ is not in general isomorphic to $Y \gprod X$.

This product is however biclosed. 
For a cubical set $X$, we write $\lhom(X, \uvar) \from \cSet \to \cSet$ and $\rhom (X, \uvar) \from \cSet \to \cSet$ for the right adjoints to the functors $\uvar \gprod X$ and $X \gprod \uvar$, respectively.
As the geometric product is not symmetric, the functors $\lhom(X, -)$ and $\rhom(X, -)$ are not naturally isomorphic. 

% We will only make use of the functor $\rhom(X, \uvar)$, thus we write $\hom(X, \uvar)$ for $\rhom(X, \uvar)$.

\subsection*{Kan complexes}
\begin{definition}\leavevmode
    \begin{enumerate}
        \item A cubical map $X \to Y$ is a \emph{Kan fibration} if it has the right lifting property with respect to open box inclusions.
        That is, if for any commutative square,
        \[ \begin{tikzcd}
            \dfobox \ar[d, hook] \ar[r] & X \ar[d, "f"] \\
            \cube{n} \ar[r] & Y
        \end{tikzcd} \]
        there exists a map $\cube{n} \to X$ so that the triangles
        \[ \begin{tikzcd}
            \dfobox \ar[d, hook] \ar[r] & X \ar[d, "f"] \\
            \cube{n} \ar[ur, dotted] \ar[r] & Y
        \end{tikzcd} \]
        commutes.
        \item A cubical set $X$ is a \emph{Kan complex} if the unique map $X \to \cube{0}$ is a Kan fibration.
    \end{enumerate}
\end{definition}
We write $\Kan$ for the full subcategory of $\cSet$ consisting of Kan complexes.
\begin{example} \label{ex:space-kan}
    For any $S \in \Top$, the cubical set $\sing{S}$ is a Kan complex.
    A map $\dfobox \to \sing{S}$ is, by adjointness, a map $\reali{\dfobox} \to S$.
    The inclusion $\reali{\dfobox} \ito \reali{\cube{n}}$ has a retract in $\Top$.
    Pre-composing with this retract gives a map $\reali{\cube{n}} \to S$ which restricts to the open box map $\reali{\dfobox} \to S$.
    \[ \begin{tikzcd}
        \reali{\dfobox} \ar[d, hook] \ar[r] & S \\
        \reali{\cube{n}} \ar[u, bend right]
    \end{tikzcd} \]
    By adjointess, this gives a suitable map $\cube{n} \to \sing{S}$.
\end{example}
\begin{definition}
    A map $f \from X \to Y$ is a \emph{weak equivalence} if the map $\reali{f} \from \reali{X} \to \reali{Y}$ is a weak homotopy equivalence, i.e.~for any $n \geq 0$ and $x \in \reali{X}$, the map $\pi_n \reali{f} \from \pi_n(\reali{X},x) \to \pi_n(\reali{Y}, \reali{f}(x))$ is an isomorphism.
\end{definition}
We move towards describing the fibration category of Kan complexes.
\begin{definition}[{\cite[Def.~1.1]{brown}}]
    A \emph{fibration category} is a category $\cat{C}$ with two subcategories of \emph{fibrations} and \emph{weak equivalences} such that (in what follows, an \emph{acyclic fibration} is a map that is both a fibration and a weak equivalence):
    \begin{enumerate}
     \item weak equivalences satisfy two-out-of-three property; that is, given two composable morphisms:
     \[ X \overset{f}\longrightarrow Y \overset{g}\longrightarrow Z \]
     if two of $f, g, gf$ are weak equivalences then all three are;
     \item all isomorphisms are acyclic fibrations;
     \item pullbacks along fibrations exist; fibrations and acyclic fibrations are stable under pullback;
     \item $\cat{C}$ has a terminal object 1; the canonical map $X \to 1$ is a fibration for any object $X \in \cat{C}$ (that is, all objects are \emph{fibrant});
     \item every map can be factored as a weak equivalence followed by a fibration.
    \end{enumerate}
\end{definition}
\begin{example}[{\cite[Thm.~2.4.19]{hovey}}]
    The category $\Top$ of topological spaces is a fibration where
    \begin{itemize}
        \item fibrations are Serre fibrations;
        \item weak equivalences are weak homotopy equivalences; i.e.~maps $f \from S \to S'$ such that, for all $s \in S$ and $n \geq 0$, the map $\pi_n f \from \pi_n(S,s) \to \pi_n(S', f(s))$ is an isomorphism.
    \end{itemize}
\end{example}
\begin{definition}
    A functor $F \from \cat{C} \to \cat{D}$ between fibration categories is \emph{exact} if it preserves fibrations, acyclic fibrations, pullbacks along fibrations, and the terminal object.
\end{definition}
Given a fibration category $\cat{C}$ with finite coproducts and a terminal object, the category of pointed objects $1 \downarrow \cat{C}$ is a fibration category as well.
\begin{proposition} \label{thm:fib_pt}
    Let $\cat{C}$ be a fibration category with finite coproducts and a terminal object $1 \in \cat{C}$.
    \begin{enumerate}
        \item The slice category $1 \downarrow \cat{C}$ under 1 is a fibration category where a map is a fibration/weak equivalence if the underlying map in $\cat{C}$ is;
        \item the projection functor $1 \downarrow \cat{C} \to \cat{C}$ is exact.
    \end{enumerate} 
\end{proposition}
\begin{proof}
    The first statement follows from \cite[Prop.~1.1.8]{hovey}.

    For the second statement, the projection functor preserves fibrations/weak equivalences by definition.
    It is a right adjoint to the functor $\uvar \sqcup 1 \from \cat{C} \to 1 \downarrow \cat{C}$ which adds a disjoint basepoint, hence preserves finite limits.
\end{proof}

\begin{theorem} \label{thm:kan_fib} \leavevmode
    \begin{enumerate}
        \item The category $\Kan$ of Kan complexes is a fibration category where fibrations are Kan fibrations and weak equivalences are as defined above.
        \item $\sing{} \from \Top \to \Kan$ is an exact functor
        \item The category $\Kan_*$ of pointed Kan complexes is a fibration category where a map is a fibration/weak equivalence if the underlying map in $\Kan$ is.
    \end{enumerate}
\end{theorem}
\begin{proof}
    \begin{enumerate}
        \item This is shown in \cite[Thm.~2.17]{carranza-kapulkin:homotopy-cubical}.
        \item This is \cite[Cor.~2.25]{carranza-kapulkin:homotopy-cubical}.
        \item Follows from \cref{thm:fib_pt}. \qedhere
    \end{enumerate}
\end{proof}

\subsection*{Anodyne maps}
We move towards defining anodyne maps of cubical sets; that is, maps which are both monomorphisms and weak equivalences.
\begin{definition}
    A class of morphisms $S$ in a cocomplete category $\cat{C}$ is \emph{saturated} if it is closed under
    \begin{itemize}
        % \item coproducts: given a collection $\{s_i \from A_i \to B_i \mid i \in I \}$ of morphisms in $\sat{S}$, the coproduct
        % \[ \coprod\limits_{i \in I} s_i \from \coprod\limits_{i \in I} A_i \to \coprod\limits_{i \in I} B_i \]
        % is in $\sat{S}$;
        \item pushouts: if $s \from A \to B$ is in $S$ and $f \from A \to C$ is any map then the pushout $C \to B \push C$ of $s$ along $f$ is in $S$;
        \[ \begin{tikzcd}
            A \ar[r, "f"] \ar[d, "s"'] & C \ar[d, "\in {S}"] \\
            B \ar[r] & B \push C
        \end{tikzcd} \]
        \item retracts: if $s \from A \to B$ is in ${S}$ and $r \from C \to D$ is a retract of $s$ in $\arr{\cat{C}}$ then $r$ is in ${S}$;
        \[ \begin{tikzcd}
            A \ar[d, "r"'] \ar[r] \ar[rr, bend left, "{\id[A]}"] & B \ar[d, "s"] \ar[r] & A \ar[d, "r"] \\
            C \ar[r] \ar[rr, bend right, "{\id[C]}"] & D \ar[r] & C
        \end{tikzcd} \]
        \item transfinite composition: given a limit ordinal $\lambda$ and a diagram $\lambda \to \cat{C}$ whose morphisms lie in $S$,
        \[ \begin{tikzcd}
            A_1 \ar[r, "s_1"] & A_2 \ar[r, "s_2"] & A_3 \ar[r, "s_3"] & \dots
        \end{tikzcd} \]
        writing $A$ for the colimit of this diagram, the components of the colimit cone $\lambda_i \from A_i \to A$ are in ${S}$.
        \[ \begin{tikzcd}
            A_1 \ar[r] \ar[dr, "\lambda_1"'] & A_2 \ar[r] \ar[d, "\lambda_2" description, near start] & A_3 \ar[r] \ar[dl, "\lambda_3"] & \dots \ar[dll, bend left, "\dots" description, near start] \\
            & A
        \end{tikzcd} \]
    \end{itemize}
\end{definition}
\begin{definition}
    For a set of morphisms $S$ in a cocomplete category $\cat{C}$, the \emph{saturation} $\sat{S}$ of $S$ is the smallest saturated class containing $S$.
\end{definition}
For a saturated class of maps, closure under pushouts and transfinite composition gives that
\begin{proposition}[{\cite[Lem.~2.1.13]{hovey}}] \label{thm:sat_coprod}
    Saturated classes of maps are closed under coproduct.
    That is, given a collection $\{s_i \from A_i \to B_i \mid i \in I \}$ of morphisms in a saturated class ${S}$, the coproduct
    \[ \coprod\limits_{i \in I} s_i \from \coprod\limits_{i \in I} A_i \to \coprod\limits_{i \in I} B_i \]
    is in ${S}$. \noproof
\end{proposition}

\begin{definition}
    A map of cubical sets is \emph{anodyne} if it is in the saturation of open box inclusions
    \[ \sat{\{ \dfobox \ito \cube{n} \mid n \geq 1, i = 0, \dots, n, \varepsilon = 0, 1 \}}. \]
\end{definition}
We use the following property of anodyne maps, which follows since saturations are closed under the left lifting property.
\begin{theorem}[{\cite[Thm.~1.34]{doherty-kapulkin-lindsey-sattler}}] \label{thm:anodyne_lift}
    Let $g \from A \to B$ be an anodyne map and $f \from X \to Y$ be a Kan fibration.
    Given a commutative square,
    \[ \begin{tikzcd}
        A \ar[r] \ar[d, "g"'] & X \ar[d, "f"] \\
        B \ar[r] & Y
    \end{tikzcd} \]
    there exists a map $B \to X$ so that the triangles
    \[ \begin{tikzcd}
        A \ar[r] \ar[d, "g"'] & X \ar[d, "f"] \\
        B \ar[ur, dotted] \ar[r] & Y
    \end{tikzcd} \]
    commute. \noproof
\end{theorem}
\begin{theorem}
    A cubical map is anodyne if and only if it is a monomorphism and a weak equivalence.
\end{theorem}
\begin{proof}
    This follows from \cite[Thm.~1.34]{doherty-kapulkin-lindsey-sattler}, as the saturation of open box inclusions is exactly the class of maps which have the left lifting property with respect to fibrations.
\end{proof}
In particular, we use that anodyne maps are sent to weak homotopy equivalences under geometric realization.
\begin{corollary} \label{thm:cset_we_to_top}
    If $f \from X \to Y$ is anodyne then, for all $n \geq 0$ and $x \in \reali{X}$, the map $\pi_n \reali{f} \from \pi_n(\reali{X}, x) \to \pi_n(\reali{Y}, \reali{f}(x))$ is an isomorphism. \noproof
\end{corollary}
\subsection*{Homotopies and homotopy groups}
Using the geoemtric product, we may define a notion of homotopy between cubical maps.
\begin{definition}
    Given cubical maps $f, g \from X \to Y$, a \emph{homotopy} from $f$ to $g$ is a map $H \from X \gprod \cube{1} \to G$ such that the diagram
    \[ \begin{tikzcd}
        X \gprod \cube{0} \ar[d, left, "\face{}{1,0}"'] \ar[dr, "f"] & \\
        X \gprod \cube{1} \ar[r, "H"] & Y \\
        X \gprod \cube{0} \ar[u, left, "\face{}{1,1}"] \ar[ur, "g"'] & 
    \end{tikzcd} \]
    commutes.
\end{definition}

Let $\Mono{\cSet}$ denote the full subcategory of $\arr{\cSet}$ spanned by monomorphisms.
Explicitly, its objects are monic cubical maps $A \ito X$.
A morphism from $A \ito X$ to $B \ito Y$ is a pair of maps $(f, g)$ which form a commutative square of the following form.
\[ \begin{tikzcd}
    A \ar[r, "g"] \ar[d, hook] & B \ar[d, hook] \\
    X \ar[r, "f"] & Y
\end{tikzcd} \]

We refer to the objects and morphisms of $\Mono{\cSet}$ as \emph{relative cubical sets} and \emph{relative cubical map}, respectively.
\iffalse
\begin{definition}
    \begin{enumerate}
        \item A \emph{relative cubical set} is an object in $\Mono{\cSet}$.
        \item A \emph{relative cubical map} is a morphism in $\Mono{\cSet}$.
    \end{enumerate}
\end{definition}
\fi
Following our convention for relative graph maps, we denote a relative cubical set $A \ito X$ by $(X, A)$, supressing the data of the map itself.
If the domain of the map $A \ito X$ is the 0-cube $A = \cube{0}$, we denote it by $(X, x)$, where $x$ is the unique 0-cube in the image of the map $\cube{0} \ito X$.

For a relative cubical map $(f, g) \from (X, A) \to (Y, B)$, the map $g$ is uniquely determined by $f$ since $B \ito Y$ is monic.
As a result, we denote a relative cubical map by $f \from (X, A) \to (Y, B)$.
Mirroring our convention for relative graph maps, we write $f$ for the map $X \to Y$ and $\restr{f}{A}$ for the map $A \to B$.

We have a corresponding notion of homotopy between relative cubical maps.
\begin{definition}
    Let $f, g \from (X, A) \to (Y, B)$ be relative cubical maps.
    A \emph{relative homotopy} from $f$ to $g$ is a morphism $(X \gprod \cube{1}, A \gprod \cube{1}) \to (Y, B)$ in $\Mono{\cSet}$ such that 
    \begin{itemize}
        \item the map $A \gprod \cube{1} \to B$ is a homotopy from $\restr{f}{A}$ to $\restr{g}{A}$;
        \item the map $X \gprod \cube{1} \to Y$ is a homotopy from $f$ to $g$.
    \end{itemize}
\end{definition}
\begin{proposition}[{\cite[Prop.~2.30]{carranza-kapulkin:homotopy-cubical}}] \label{thm:rel_htpy_eq_rel}
    If $B, Y$ are Kan complexes then relative homotopy is an equivalence relation on relative cubical maps $(X, A) \to (B, Y)$. \noproof
\end{proposition}
It is essential in \cref{thm:rel_htpy_eq_rel} that $B$ and $Y$ are Kan complexes.
If not, the relation of relative homotopy is neither symmetric nor transitive (and in this case, one considers the symmetric transitive closure of relative homotopy).

We write $[(X, A), (Y, B)]$ for the set of relative homotopy classes of relative maps $(X, A) \to (Y, B)$.
With this, we define the homotopy groups of a Kan complex.
\begin{definition}[{\cite[Cor.~3.16]{carranza-kapulkin:homotopy-cubical}}]
    Let $(X,x)$ be a pointed Kan complex.
    We define the $n$-th homotopy group $\pi_n (X,x)$ of $(X,x)$ as the relative homotopy classes of relative maps $(\cube{n}, \bd \cube{n}) \to (X, x)$.
\end{definition}
We give an explicit description of multiplication in the first homotopy group $\pi_1(X,x)$.
\begin{definition}
    Given two 1-cubes $u, v \from \cube{1} \to X$ in a cubical set $X$, 
    \begin{enumerate}
        \item a \emph{concatenation square} for $u$ and $v$ is a map $\eta \from \cube{2} \to X$ such that
        \begin{itemize}
            \item $\eta\face{}{1,0} = u$;
            \item $\eta\face{}{1,1} = v\face{}{1,1}\degen{}{1}$;
            \item $\eta\face{}{2,1} = v$.
        \end{itemize}
        \item a \emph{concatenation} of $u$ and $v$ is a 1-cube $w \from \cube{1} \to X$ which is the $\face{}{2,0}$-face of some concatenation square for $u$ and $v$.
    \end{enumerate}
\end{definition}
\begin{proposition}[{\cite[Thm.~3.11]{carranza-kapulkin:homotopy-cubical}}]
    If $(X,x)$ is a pointed Kan complex then composition induces a well-defined binary operation 
    \[ [(\cube{1}, \bd \cube{1}), (X,x)] \times [(\cube{1}, \bd \cube{1}), (X,x)] \to [(\cube{1}, \bd \cube{1}), (X,x)] \]
    on relative homotopy classes of relative maps which gives a group structure on $[(\cube{1}, \bd \cube{1}), (X,x)]$. \noproof
\end{proposition}

As with spaces, the homotopy groups of a Kan complex are the connected components of its \emph{loop space}, which we define.
\begin{definition}
    For a pointed Kan complex $(X,x)$, 
    \begin{itemize}
        \item the \emph{loop space} $\loopsp(X,x)$ of $X$ is the pullback
        \[ \begin{tikzcd}
            \loopsp(X,x) \ar[r] \ar[d] \ar[rd, phantom, "\ulcorner" very near start] & \hom(\cube{1}, X) \ar[d, "{(\face{*}{1,0}, \face{*}{1,1})}"] \\
            \cube{0} \ar[r, "x"] & X \times X
        \end{tikzcd} \]
        with a distinguished 0-cube $x\degen{}{1} \from \cube{0} \to \loopsp(X,x)$.
        \item for $n \geq 0$, the $n$-th \emph{loop space} $\loopsp[n](X,x)$ of $X$ is defined to be
        \[ \loopsp[n](X,x) := \begin{cases}
            (X,x) & n = 0 \\
            \loopsp(\loopsp[n-1](X,x)) & n > 0.
        \end{cases} \]
    \end{itemize} 
\end{definition}
\begin{proposition}[{\cite[Cor.~3.16]{carranza-kapulkin:homotopy-cubical}}] \label{thm:loopsp_pi}
    For a pointed Kan complex $(X,x)$ and $0 \leq k \leq n$, we have an isomorphism
    \[ \pi_{n}(X,x) \cong \pi_{n-k} (\loopsp[k](X,x)) \]
    natural in $X$. \noproof
\end{proposition}
Using \cref{thm:loopsp_pi}, the group structure on higher homotopy groups is induced by the bijection $\pi_n(X,x) \cong \pi_1(\loopsp[n-1](X,x))$.

This definition of homotopy groups agrees with the homotopy groups of its geometric realization.
\begin{theorem}[{\cite[Thm.~3.25]{carranza-kapulkin:homotopy-cubical}}] \label{thm:cset_pi_eq_pi}
    There is an isomorphism 
    \[ \pi_n(X, x) \cong \pi_n(\reali{X}, x) \]
    natural in $X$. \noproof
\end{theorem}

We know that the loop space functor is exact.
\begin{theorem}[{\cite[Thm.~3.6]{carranza-kapulkin:homotopy-cubical}}] \label{thm:cset_loopsp_exact}
    The loop space functor $\loopsp \from \Kan_* \to \Kan_*$ is exact. \noproof
\end{theorem}

% The conjecture of \cite{babson-barcelo-longueville-laubenbacher} is an isomorphism
% \[ A_n(G, v) \cong \pi_n(\reali{\sing[1]{G}}, v). \]

\section{Cubical nerve of a graph}\label{sec:nerves}
Let $m \geq 1$. 
Geometrically, we view the graph $\gexp{I_m}{n}$ as an $n$-dimensional cube.
Making this intuition formal, we have face, degeneracy, and connection maps defined as follows.
\begin{itemize}
    \item the face map $\face{n}{i,\varepsilon} \from \gexp{I_m}{n-1} \to \gexp{I_m}{n}$ for $1 \leq i \leq n$ and $\varepsilon = 0$ or $1$ is given by 
    \[ \face{n}{i\varepsilon}(v_1, \dots, v_n) = (v_1, \dots, v_{i-1}, \varepsilon m, v_{i}, \dots, v_n); \]
    \item the degeneracy map $\degen{n}{i} \from \gexp{I_m}{n+1} \to \gexp{I_m}{n}$ for $1 \leq i \leq n$ is given by
    \[ \degen{n}{i}(v_1, \dots, v_n) = (v_1, \dots, v_{i-1}, v_{i+1}, \dots, v_n); \]
    \item the negative connection map $\conn{n}{i,0} \from \gexp{I_m}{n} \to \gexp{I_m}{n-1}$ for $1 \leq i \leq n-1$ is given by
    \[ \conn{n}{i,0}(v_1, \dots, v_n) = (v_1, \dots, v_{i-1}, \max(v_i, v_{i+1}), v_{i+2}, \dots, v_n); \]
    \item the positive connection map $\conn{n}{i,1} \from \gexp{I_m}{n} \to \gexp{I_m}{n-1}$ for $1 \leq i \leq n-1$ is given by
    \[ \conn{n}{i,1}(v_1, \dots, v_n) = (v_1, \dots, v_{i-1}, \min(v_i, v_{i+1}), v_{i+2}, \dots, v_n). \]
\end{itemize}
It is straightfoward to verify that these maps satisfy cubical identites.
This defines a functor $\boxcat \to \Graph$ which sends $[1]^n$ to $\gexp{I_m}{n}$.
Left Kan extension along the Yoneda embedding gives an adjunction $\cSet \rightleftarrows \Graph$.
\[ \begin{tikzcd}
    \boxcat \ar[r] \ar[d] & \Graph \ar[dl, yshift=-0.5ex, xshift=1ex] \\
    \cSet \ar[ur, yshift=0.5ex] 
\end{tikzcd} \]
\begin{definition}
    For $m \geq 1$,
    \begin{enumerate}
        \item the \emph{$m$-realization} functor $\reali[m]{\uvar} \from \cSet \to \Graph$ is the left Kan extension of the functor $\boxcat \to \Graph$ which sends $[1]^n$ to $\gexp{I_m}{n}$.
        \item the \emph{$m$-nerve} functor $\gnerve[m] \from \Graph \to \cSet$ is the right adjoint of the $m$-realization functor defined by
        \[ (\gnerve[m]G)_n := \Graph(\gexp{I_m}{n}, G). \]
    \end{enumerate}
\end{definition}
\begin{remark}
    The 1-nerve of a graph $G$ is constructed in \cite{babson-barcelo-longueville-laubenbacher} as the \emph{cubical set associated to $G$}, denoted $M_*(G)$. The geometric realization of this cubical set is referred to as the \emph{cell complex associated to $G$}, denoted $X_G$.
\end{remark}
For a cubical set $X \in \cSet$, the graph $\reali[m]{X}$ may be explicitly described as the colimit
\[ \reali[m]{X} := \colim (G_{m, 0} \ito G_{m, 1} \ito G_{m, 2} \ito \dots ) \]
where
\begin{itemize}
    \item $G_{m, 0}$ is the discrete graph whose vertices are $X_0$;
    \item $G_{m, n+1}$ is obtained from $G_{m, n}$ via the following pushout (where $(X_n)_{\textsf{nd}}$ is the subset of $X_n$ consisting of non-degenerate cubes)
    \[ \begin{tikzcd}
        \coprod\limits_{x \in (X_n)_{\textsf{nd}}} \bd \gexp{I_m}{n} \ar[r] \ar[d] \ar[rd, phantom, "\ulcorner" very near end] & G_{m, n} \ar[d] \\
        \coprod\limits_{x \in (X_n)_{\textsf{nd}}} \gexp{I_m}{n} \ar[r] & G_{m, n+1}
    \end{tikzcd} \]
    where $\bd \gexp{I_m}{n}$ is the subgraph of $\gexp{I_m}{n}$ defined by
    \[ \bd \gexp{I_m}{n} := \{ (v_1, \dots, v_n) \in \gexp{I_m}{n} \mid v_i = 0 \text{ or } m \text{ for some } i=0, \dots, n \} \]
    where the edge set is discrete if $m = n = 1$ and full otherwise.
\end{itemize}
\begin{example} \leavevmode
    \begin{enumerate}
        \item We describe the graph $\reali[1]{\obox{2}{2,1}}$.
        The cubical set $\obox{2}{2,1}$ has four 0-cubes, thus, the graph $G_{1, 0}$ is the discrete graph with four vertices.
        \begin{figure}[H]
            \centering
            \begin{tikzpicture}[node distance=25pt]
                % vertices
                \node[vertex] (00a) {};
                \node[vertex] (10a) [right of=00a] {};
                \node[vertex] (01a) [below of=00a] {};
                \node[vertex] (11a) [below of=10a] {};
                \node[vertex] (00b) [right=40pt of 10a] {};
                \node[vertex] (10b) [right of=00b] {};
                \node[vertex] (01b) [below of=00b] {};
                \node[vertex] (11b) [below of=10b] {};
                % edges
                \draw (00b) to (10b)
                      (00b) to node (m2) {} (01b)
                      (10b) to (11b);
                % draw the arrow between them
                \path (10a) to node (m1) {} (11a);
                \draw[right hook->, shorten >=10pt, shorten <=10pt] (m1) to (m2);
            \end{tikzpicture}
            \caption{The embedding of $G_{1, 0} \ito G_{1, 1}$ for $\reali[1]{\obox{2}{2,1}}$.}
        \end{figure}
    
        The open box $\obox{2}{2,1}$ has three non-degenerate 1-cubes.
        The pushout constructed to obtain $G_{1, 1}$ glues three copies of $I_1$ to $G_{1, 0}$.
        As $\obox{2}{2,1}$ contains only degenerate cubes above dimension 1, constructing $G_{1, 1}$ completes the construction of the graph $\reali[1]{\obox{2}{2,1}}$.    

        \item To construct $\reali[3]{\obox{2}{2,1}}$, we instead glue three copies of $I_3$.
        Observe that this process adds new vertices to the graph.
        \begin{figure}[H]
            \centering
            \begin{tikzpicture}[node distance=20pt]
                % vertices
                \node[vertex, fill=gray!35] (00) {};
                \node[vertex] (10) [right of=00] {};
                \node[vertex] (20) [right of=10] {};
                \node[vertex, fill=gray!35] (30) [right of=20] {};
                \node[vertex] (01) [below of=00] {};
                \node[vertex] (02) [below of=01] {};
                \node[vertex, fill=gray!35] (03) [below of=02] {};
                \node[vertex] (31) [below of=30] {};
                \node[vertex] (32) [below of=31] {};
                \node[vertex, fill=gray!35] (33) [below of=32] {};
    
                % edges
                \draw (00) to (10)
                      (10) to (20)
                      (20) to (30)
                      (00) to (01)
                      (01) to (02)
                      (02) to (03)
                      (30) to (31)
                      (31) to (32)
                      (32) to (33);
            \end{tikzpicture}
            \caption{The graph $\reali[3]{\obox{2}{2,1}}$. The image of the embedding $G_{3, 0} \ito \reali[3]{\obox{2}{2,1}}$ is highlighted.}
        \end{figure}    
    \end{enumerate}
\end{example}

For a graph $G$, let $l^*, r^* \from \gnerve[m]{G} \to \gnerve[m+1]{G}$ denote the cubical maps obtained by precomposition with the surjections $\gexp{l}{n}, \gexp{r}{n} \from \gexp{I_{m+1}}{n} \to \gexp{I_m}{n}$.
We think of these maps as inclusions of $n$-cubes of size $m$ into $n$-cubes of size $m+1$.
\begin{figure}[ht]
    \centering
    \begin{tikzpicture}[node distance=20pt]
        \node[vertex, fill=red] (0) {};
        \node[vertex, fill=cyan] (1) [right=of 0] {};
        
        \node[vertex, fill=cyan] (l2) [below left=of 0] {};
        \node[vertex, fill=red] (l1) [left=of l2] {};
        \node[vertex, fill=red] (l0) [left=of l1] {};

        \node[vertex, fill=red] (r0) [below right=of 1] {};
        \node[vertex, fill=cyan] (r1) [right=of r0] {};
        \node[vertex, fill=cyan] (r2) [right=of r1] {};

        \draw (0) to (1);
        \draw (l0) to (l1) to (l2);
        \draw (r0) to (r1) to (r2);
        \draw[|->, shorten >= 10pt, shorten <= 10pt] (0) to[bend right] node[above] {$l^*$} (l1);
        \draw[|->, shorten >= 10pt, shorten <= 10pt] (1) to[bend left] node[above] {$r^*$} (r1);
    \end{tikzpicture}
    \caption{For a 1-cube $f \from I_1 \to G$ of $\gnerve[1]{G}$, the map $l^* \from \gnerve[1]G \to \gnerve[2] G$ sends $f$ to the 1-cube $fl \from I_2 \to G$ of $\gnerve[2]{G}$ whereas $r^* \from \gnerve[1] G \to \gnerve[2] G$ sends $f$ to the 1-cube $fr \from I_2 \to G$ of $\gnerve[2]{G}$.}
\end{figure}
\begin{figure}[ht]
    \centering
    \begin{tikzpicture}[node distance=15pt]
        \node[vertex, fill=red] (00) {};
        \node[vertex, fill=cyan] (10) [right=of 00] {};
        \node[vertex, fill=yellow] (01) [below=of 00] {};
        \node[vertex, fill=green] (11) [right=of 01] {};

        \node (lspace) [below left=of 01] {};
        \node (rspace) [below right=of 11] {};
        
        \node[vertex, fill=cyan] (l20) [left=of lspace] {};
        \node[vertex, fill=red] (l10) [left=of l20] {};
        \node[vertex, fill=red] (l00) [left=of l10] {};
        \node[vertex, fill=cyan] (l21) [below=of l20] {};
        \node[vertex, fill=red] (l11) [left=of l21] {};
        \node[vertex, fill=red] (l01) [left=of l11] {};
        \node[vertex, fill=green] (l22) [below=of l21] {};
        \node[vertex, fill=yellow] (l12) [left=of l22] {};
        \node[vertex, fill=yellow] (l02) [left=of l12] {};
        
        \node[vertex, fill=red] (r00) [right=of rspace] {};
        \node[vertex, fill=cyan] (r10) [right=of r00] {};
        \node[vertex, fill=cyan] (r20) [right=of r10] {};
        \node[vertex, fill=yellow] (r01) [below=of r00] {};
        \node[vertex, fill=green] (r11) [right=of r01] {};
        \node[vertex, fill=green] (r21) [right=of r11] {};
        \node[vertex, fill=yellow] (r02) [below=of r01] {};
        \node[vertex, fill=green] (r12) [right=of r02] {};
        \node[vertex, fill=green] (r22) [right=of r12] {};

        \draw (00) to (01);
        \draw (01) to (11);
        \draw (00) to (10);
        \draw (10) to (11);

        \draw (l00) to (l10) to (l20);
        \draw (l01) to (l11) to (l21);
        \draw (l02) to (l12) to (l22);
        \draw (l00) to (l01) to (l02);
        \draw (l10) to (l11) to (l12);
        \draw (l20) to (l21) to (l22);
        
        \draw (r00) to (r10) to (r20);
        \draw (r01) to (r11) to (r21);
        \draw (r02) to (r12) to (r22);
        \draw (r00) to (r01) to (r02);
        \draw (r10) to (r11) to (r12);
        \draw (r20) to (r21) to (r22);

        \draw[|->, shorten >= 10pt, shorten <= 10pt] (01) to[bend right] node[above] {$l^*$} (l20);
        \draw[|->, shorten >= 10pt, shorten <= 10pt] (11) to[bend left] node[above] {$r^*$} (r00);
    \end{tikzpicture}
    \caption{For a 2-cube $g \from \gexp{I_1}{2} \to G$ of $\gnerve[1]{G}$, the map $l^* \from \gnerve[1]G \to \gnerve[2] G$ sends $g$ to the 1-cube $g\gexp{l}{2} \from \gexp{I_2}{2} \to G$ of $\gnerve[2]{G}$ whereas $r^* \from \gnerve[1] G \to \gnerve[2] G$ sends $g$ to the 1-cube $g\gexp{r}{2} \from \gexp{I_2}{2} \to G$ of $\gnerve[2]{G}$.}
\end{figure}
We write $c^* \from \gnerve[m]{G} \to \gnerve[m+2]{G}$ for the composite $l^*r^* = r^* l^*$.

\begin{remark}
    While the maps $\gexp{l}{n}, \gexp{r}{n}$ have sections $\gexp{I_m}{n} \to \gexp{I_{m+1}}{n}$, these maps do not commute with face maps, hence do not give retractions $\gnerve[m+1]{G} \to \gnerve[m]{G}$.
    To demonstrate this, we show the map $l^* \from \gnerve[1]{I_2} \to \gnerve[2]{I_2}$ does not have a retraction.
    The identity map $\id[I_2] \from I_2 \to I_2$ gives a 1-cube of $\gnerve[2]{I_2}$ whose faces are the 0-cubes 0 and 2.
    Observe a retraction of $l^*$ must send the 0-cubes 0 and 2 to 0 and 2, respectively.
    There is no map $f \from I_1 \to I_2$ such that $f\face{}{1,0} = 0$ and $f\face{}{1,1} = 2$.
    That is, there is no 1-cube of $\gnerve[1] I_2$ which $\id[I_2] \in (\gnerve[2] I_2)_1$ can be mapped to.
    Thus, the map $l^*$ does not have a retraction.
\end{remark} 

For a cubical set $X$, we analogously have maps $l_*, r_* \from \reali[m+1]{X} \to \reali[m]{X}$.
We write $c_* \from \reali[m+2]{X} \to \reali[m]{X}$ for the composite $l_* r_* = r_* l_*$.

We define the \emph{nerve} functor $\gnerve \from \Graph \to \cSet$ by
\[ (\gnerve G)_n := \{ f \from \gexp{I_\infty}{n} \to G \mid f \text{ is stable in all directions } (i, \varepsilon) \}. \]
Cubical operators of $\gnerve G$ are given as follows.
\begin{itemize}
    \item The map $\face{n}{i,\varepsilon} \from (\gnerve G)_n \to (\gnerve G)_{n-1}$ for $i = 1, \dots, n$ and $\varepsilon = 0, 1$ is given by $f\face{n}{i,\varepsilon} \from \gexp{I_\infty}{n-1} \to G$ defined by
    \[ f\face{n}{i,\varepsilon}(v_1, \dots, v_{n-1}) = f(v_1, \dots, v_{i-1}, (2\varepsilon - 1) M, v_i, \dots, v_{n-1}) \]
    where $M$ is such that $f$ is stable in direction $(i,\varepsilon)$;
    \item The map $\degen{n}{i} \from (\gnerve G)_n \to (\gnerve G)_{n+1}$ for $i=1, \dots, n$ is given by $f\degen{n}{i} \from \gexp{I_\infty}{n+1} \to G$ defined by
    \[ f\degen{n}{i}(v_1, \dots, v_{n+1}) = f(v_1, \dots, v_{i-1}, v_{i+1}, \dots, v_{n+1}); \]
    \item The map $\conn{n}{i,0} \from (\gnerve G)_n \to (\gnerve G)_{n+1}$ for $i=1, \dots, n-1$ is given by $f\conn{n}{i,0} \from \gexp{I_\infty}{n+1} \to G$ defined by
    \[ f\conn{n}{i,0}(v_1, \dots, v_{n+1}) = f(v_1, \dots, v_{i-1}, \max(v_i, v_{i+1}), v_{i+1}, \dots, v_{n+1}); \]
    \item The map $\conn{n}{i,1} \from (\gnerve G)_n \to (\gnerve G)_{n+1}$ for $i=1, \dots, n-1$ is given by $f\conn{n}{i,0} \from \gexp{I_\infty}{n+1} \to G$ defined by
    \[ f\conn{n}{i,0}(v_1, \dots, v_{n+1}) = f(v_1, \dots, v_{i-1}, \min(v_i, v_{i+1}), v_{i+1}, \dots, v_{n+1}). \]
\end{itemize}
One verifies these maps satisfy cubical identities, thus $\gnerve G$ is a cubical set. 
A straightforward computation gives the following statement.
% For $m \geq 0$, let $m = 2k + r$ where $0 \leq r < 2$.
% It is straightforward to verify that $\gexp{I_\infty}{n}$ is the limit of the diagram
% \[ \begin{tikzcd}
%     \dots \ar[r, "\gexp{r}{n}"] & \gexp{I_4}{n} \ar[r, "\gexp{l}{n}"] & \gexp{I_3}{n} \ar[r, "\gexp{r}{n}"] & \gexp{I_2}{n} \ar[r, "\gexp{l}{n}"] & \gexp{I_1}{n}
% \end{tikzcd} \]
% Any map $\gexp{I_m}{n} \to G$ gives a map $\gexp{I_\infty}{n} \to G$ by precomposing with the respective component of the cone.
% Conversely, a map $\gexp{I_\infty}{n} \to G$ stabilizes in all directions exactly when it factors through a component of this cone.
% This shows that
\begin{proposition} \label{thm:nerve_colim}
    We have an isomorphism
    \[ \gnerve{G} \cong \colim (\gnerve[1]{G} \xrightarrow{l^*} \gnerve[2]{G} \xrightarrow{r^*} \gnerve[3]{G} \xrightarrow{l^*} \gnerve[4]{G} \xrightarrow{r^*} \dots) \]
    natural in $G$.
    \noproof
\end{proposition}

The nerve and realization functors satisfy the following categorical properties.

\begin{proposition} \label{thm:reali_prod}
    For cubical sets $X, Y$ and $m \geq 1$, we have an isomorphism
    \[ \reali[m]{X \gprod Y} \cong \reali[m]{X} \gtimes \reali[m]{Y} \]
    natural in $X$ and $Y$.
  \end{proposition}
  \begin{proof}
    % As $m$-realization $\reali[m]{\uvar}\from \cSet \to \Graph$ is a left Kan extension, its product with itself $\reali[m]{\uvar} \times \reali[m]{\uvar} \from \cSet \times \cSet \to \Graph \times \Graph$ is a left Kan extension.
    % This implies it preserves all colimits, hence the composite
    % \[ \reali[m]{\uvar} \gtimes \reali[m]{\uvar} \from \cSet \times \cSet \to \Graph \]
    % is a left Kan extension.
    % As both the geometric product $\gprod \from \cSet \times \cSet \to \cSet$ and the graph product $\gtimes \from \Graph \times \Graph \to \Graph$ preserve colimits, they preserve left Kan extension.
    % With this, the composites
    The composite functors
    \[ \reali[m]{\uvar \gprod \uvar}, \reali[m]{\uvar} \gtimes \reali[m]{\uvar} \from \cSet \times \cSet \to \Graph \]
    preserve all colimits.
    As $\cSet$ is a presheaf category, every cubical set is a colimit of representable presheaves.
    Thus, it suffices to show these composites are naturally isomorphic on pairs $(\cube{a}, \cube{b})$ for $a, b \geq 0$. We compute
    \begin{align*}
      \reali[m]{\cube{a} \gprod \cube{b}} &\cong \reali[m]{\cube{a + b}} \\
      &\cong \gexp{I_{m}}{a + b} \\
      &\cong \gexp{I_m}{a} \gtimes \gexp{I_m}{b} \\
      &\cong \reali[m]{\cube{a}} \gtimes \reali[m]{\cube{b}}. \qedhere
    \end{align*}
  \end{proof}
  \begin{corollary} \label{thm:nerve_hom_iso}
    Let $X$ be a cubical set and $G$ be a graph.
    For $m \geq 1$, we have isomorphisms
    \[ \lhom(X, \gnerve[m] G) \cong \gnerve[m] (\ghom{\reali[m]{X}}{G}) \cong \rhom(X, \gnerve[m] G) \]
    natural in $X$ and $G$.
  \end{corollary}
  \begin{proof}
      The square
      \[ \begin{tikzcd}[column sep = large]
        \cSet \ar[r, "{\uvar \gprod X}"] \ar[d, "{\reali[m]{\uvar}}"'] & \cSet \ar[d, "{\reali[m]{\uvar}}"] \\
        \Graph \ar[r, "{\uvar \gtimes \reali[m]{X}}"] & \Graph
      \end{tikzcd} \]
      commutes up to natural isomorphism by \cref{thm:reali_prod}, thus the corresponding square of right adjoints
      \[ \begin{tikzcd}[column sep = huge]
          \Graph \ar[r, "{\ghom{\reali[m]{X}}{\uvar}}"] \ar[d, "{\gnerve[m]}"'] & \Graph \ar[d, "{\gnerve[m]}"] \\
          \cSet \ar[r, "{\lhom(X, \uvar)}"] & \cSet
      \end{tikzcd} \]
      commutes up to natural isomorphism.
      For naturality in $X$, the required square commutes by faithfulness of the Yoneda embedding $\cSet \to \fcat{\cSet^\op}{\Set}$.
      A similar argument involving $X \gprod -$ constructs the isomorphism involving $\rhom(X, -)$.
  \end{proof}
\begin{proposition} \label{thm:nerve_fin_lim}
    The nerve functor $\gnerve \from \Graph \to \cSet$ preserves finite limits.
\end{proposition}
\begin{proof}
    By \cref{thm:nerve_colim}, the nerve of $G$ is a filtered colimit.
    This then follows as filtered colimits commute with finite limits and $\gnerve[m] \from \Graph \to \cSet$ is a right adjoint for all $m \geq 1$. 
\end{proof}
We prove that the nerve functors preserve filtered colimits, which we use to give an analogue of \cref{thm:nerve_hom_iso} for the nerve functor $\gnerve \from \Graph \to \cSet$.
\begin{proposition} \label{thm:nerve_colim_filtered}
    For $m \geq 0$, the functors $\gnerve[m], \gnerve \from \Graph \to \cSet$ preserve filtered colimits.
\end{proposition}
\begin{proof}
    For $\gnerve[m]$, it suffices to show filtered colimits are preserved component-wise, i.e.~that 
    \[ \Graph(\gexp{I_m}{n}, \uvar) \from \Graph \to \Set \]
    preserves filtered colimits.
    This follows from \cref{thm:fin_graph_compact}.

    For $\gnerve$, this follows since $\gnerve[m]$ preserves filtered colimits and colimits commute with colimits.
\end{proof}
For a cubical set $X$, define a functor $\indP{X} \from \Graph \to \Graph$ by
\[ \indP{X} G := \colim \left( \begin{tikzcd}
    \ghom{\reali[1]{X}}{G} \ar[r, "(l_*)^*"] & \ghom{\reali[2]{X}}{G} \ar[r, "(r_*)^*"] & \ghom{\reali[3]{X}}{G} \ar[r, "(l_*)^*"] & \dots
\end{tikzcd} \right). \] 
As an example, the path graph $PG$ of a graph is exactly $\indP{\cube{1}} G$.
\begin{proposition} \label{thm:nerve_hom_iso_fin}
    Let $X$ be a cubical set with finitely many non-degenerate cubes.
    We have isomorphisms
    \[ \gnerve(\indP{X}{G}) \cong \lhom(X, \gnerve G) \cong \rhom(X, \gnerve G) \]
    natural in $X$ and $G$.
\end{proposition}
\begin{proof}
    By \cref{thm:nerve_colim}, the left term $\gnerve(\indP{X}{G})$ is the colimit
    \[ \gnerve(\indP{X}{G}) \cong \colim \left( \begin{tikzcd}
        \gnerve[1] \ghom{\reali[1]{X}}{G} \ar[r, hook, "l^*"] \ar[d, hook, "(l_*)^*"] & \gnerve[2] \ghom{\reali[1]{X}}{G} \ar[r, hook, "r^*"] \ar[d, hook, "(l_*)^*"] & \dots \\
        \gnerve[1] \ghom{\reali[2]{X}}{G} \ar[r, hook, "l^*"] \ar[d, hook, "(r_*)^*"] & \gnerve[2] \ghom{\reali[2]{X}}{G} \ar[r, hook, "r^*"] \ar[d, hook, "(r_*)^*"] & \dots \\
        \dots & \dots & \dots
    \end{tikzcd} \right), \]
    where the vertical maps are monomorphisms since $l_*, r_* \from \reali[m+1]{X} \to \reali[m]{X}$ are epimorphisms for all $m \geq 1$.
    Computing this colimit component-wise in $\Set$, this colimit is naturally isomorphic to the colimit
    \[ \gnerve(\indP{X}{G}) \cong \colim \left( \begin{tikzcd}
        \gnerve[1] \ghom{\reali[1]{X}}{G} \ar[r] & \gnerve[2] \ghom{\reali[2]{X}}{G} \ar[r] & \dots 
    \end{tikzcd} \right) \]
    along the diagonal.
    Applying \cref{thm:nerve_hom_iso}, we may write
    \[ \gnerve(\indP{X}{G}) \cong \colim \left( \begin{tikzcd}
        \lhom(X, \gnerve[1] G) \ar[r] & \lhom(X, \gnerve[2] G) \ar[r] & \dots
    \end{tikzcd} \right). \]
    As $X$ has finitely many non-degenerate cubes, the functor $\lhom(X, \uvar)$ preserves filtered colimits.
    Thus,
    \[ \gnerve(\indP{X}{G}) \cong \lhom(X, \gnerve G). \]
    An analogous proof applies in the case of $\rhom(X, \gnerve G)$.
\end{proof}
% We make use of the following corollary of \cref{thm:nerve_hom_iso_fin}.
% \begin{corollary} \label{thm:path_graph_path_space}
%     For a graph $G$, there are isomorphisms
%     \[ \gnerve(PG) \cong \rhom(\cube{1}, G), \quad \gnerve (G \times G) \cong \gnerve G \times \gnerve G \]
%     natural in $G$ such that square
%     \[ \begin{tikzcd}
%         {\gnerve (PG) \ar[r, "\cong"]} \ar[d, "{\gnerve(\face{}{1,0}, \face{}{1,1})}"'] & {\rhom(\cube{1}, G)} \ar[d, "{(\face{*}{1,0}, \face{*}{1,1})}"] \\
%         {\gnerve (G \times G)} \ar[r, "\cong"] & {\gnerve G \times \gnerve G}
%     \end{tikzcd} \]
%     commutes.
% \end{corollary}
% \begin{proof}
%     The isomorphisms follow from \cref{thm:nerve_hom_iso_fin}, where the square is given by naturality applied to the map $\bd \cube{1} \to \cube{1}$.
% \end{proof}

We show that the nerve functors ``detect'' concatenation of paths. 
That is, they contain all possible composition squares.
\begin{proposition} \label{thm:nerve_conc}
    Let $f \from (I_\infty, I_{\geq M}) \to (G, v)$ and $g \from (I_\infty, I_{\geq N}) \to (G, v)$ be based graph maps for some $M, N \geq 0$ such that $f\face{}{1,1} = g\face{}{1,0}$.
    The concatenation $f \cdot g$ of $f$ followed by $g$ is a concatenation of $f$ and $g$ in the $(M+N)$-nerve $\gnerve[(M + N)]{G}$.
\end{proposition}
\begin{proof}
    The horizontal concatenation of $f\conn{}{1,0} \from \gexp{I_M}{2} \to G$ on the left and $g\degen{}{2} \from \gexp{I_N}{2} \to G$ on the right gives a square $\eta \from \gexp{I_{M+N}}{2} \to G$ such that
    \[ \begin{array}{l l}
        \eta\face{}{1,0} = f & \eta\face{}{1,1} = g\face{}{1,1}\degen{}{1} \\
        \eta\face{}{2,0} = f \cdot g & \eta\face{}{2,1} = g.
    \end{array} \]
    This is exactly a composition square witnessing $f \cdot g$ as a composition of $f$ and $g$ in $\gnerve[(M+N)]{G}$.
\end{proof}
As well, the nerve functors reflect isomorphisms.
  \begin{proposition}
    The nerve functors reflect isomorphisms.
    That is, given a graph map $f \from G \to H$, if either
    \begin{enumerate}
        \item $\gnerve[m] f \from \gnerve[m] G \to \gnerve[m] H$ is an isomorphism for some $m \geq 1$; or
        \item $\gnerve f \from \gnerve G \to \gnerve H$ is an isomorphism
    \end{enumerate}
    then $f$ is.
  \end{proposition}
  \begin{proof}
    We prove that $\gnerve \from \Graph \to \cSet$ reflects isomorphisms as the case for $\gnerve[m] \from \Graph \to \cSet$ is analogous.

    The inverse of $\gnerve f$ is, in particular, an inverse on 0-cubes $(\gnerve H)_0 \to (\gnerve G)_0$, i.e.~an inverse of $f$ on vertices $H_V \to G_V$.
    It suffices to show this map $g \from H_V \to G_V$ is a graph map.

    An edge in $H$ gives a map $e \from I_1 \to H$.
    This gives a 1-cube $\overline{e} \from \cube{1} \to \gnerve G$ by the inclusion $\gnerve[1] H \ito \gnerve H$.
    As $\gnerve f \from \gnerve G \to \gnerve H$ is an isomorphism, there is a unique 1-cube $\overline{p} \from \cube{1} \to \gnerve G$ such that $\gnerve f(\overline{p}) = \overline{e}$.
    This corresponds to a map $p \from I_\infty \to G$ such that $p$ stabilizes in both directions and $fp = e$.

    As $f$ is injective on vertices and $e$ is a path of length 1 (i.e.~it consists of 2 vertices and 1 edge), we deduce that $p$ is a path of length 1.
    That is, $p$ is an edge, hence $g$ is a graph map.
    % Via the inclusion $\gnerve[1] H \ito \gnerve H$, an edge in $H$ corresponds to a map $p \from I_\infty \to H$ such that $p(-i) = p(0)$ and $p(i) = p(1)$ for all $i \geq 1$.
    % Applying $g$ gives a map $gp \from I_\infty \to G$ which stabilizes to $gp(0)$ in direction $(1, 0)$ and $gp(1)$ in direction $(1,1)$.  
    % Suppose there exists $i \in I_\infty$ such that $gp(i) \in G$ is neither $gp(0)$ nor $gp(1)$.
    % Then, we may write $gp$ as a concatenation of two maps $q, r \from I_\infty \to G$ such that $q$ stabilizes to $gp(i)$ in direction $(1,1)$ and $r$ stabilizes to $gp(i)$ in direction $(1,0)$.
    % By \cref{thm:nerve_conc}, there is a composition square in $\gnerve G$ with three distinct vertices witnessing $gp$ as a composition of 1-cubes $q$ and $r$.
    % As $f$ is an isomorphism, there is a unique square in $\gnerve H$ with three distinct vertices witnessing $p$ as a composition of non-degenerate cubes $fq, fr \from I_\infty \to H$.
    % As $f$ is a graph map, $f(q \cdot r) = fq \cdot fr$.
    % However, $fq \cdot fr \neq p$ as $p$ only contains two distinct vertices.
    % Thus, such an $i$ cannot exist.
    % From this, we conclude that there is an edge between $f(p(0))$ and $f(p(1))$, thus $f$ is a graph map.
    % This same argument also shows that $f^{-1}$ is a graph map.
  \end{proof}
  Recall that a functor which reflects isomorphisms also reflects any (co)limits which it preserves.
  Thus,
  \begin{corollary} \label{thm:nerve_ref_fin_lim} \leavevmode
    \begin{enumerate}
        \item For $m \geq 1$, the $m$-nerve $\gnerve[m] \from \Graph \to \cSet$ reflects all limits.
        \item The nerve functor $\gnerve \from \Graph \to \cSet$ reflects finite limits. \noproof
    \end{enumerate} 
  \end{corollary}

\section{Main result} \label{sec:main}

\subsection*{Statement}
Our main theorem is the following.
\begin{restatable}{theorem}{mainthm} \label{thm:main}
    For any graph $G$,
    \begin{enumerate}
        \item the nerve $\gnerve G$ of $G$ is a Kan complex;
        \item the natural inclusion $\gnerve[1] G \ito \gnerve G$ is anodyne.
    \end{enumerate}
\end{restatable}
Before proving \cref{thm:main}, we explain the proof strategy and establish some auxilliary lemmas.

To show that the nerve of any graph is a Kan complex, we construct a map $\f \from \gexp{I_{3m}}{n} \to \reali[m]{\dfobox}$ so that the triangle
\[ \begin{tikzcd}
    \reali[3m]{\dfobox} \ar[r, "{c_*^m}"] \ar[d, hook] & \reali[m]{\dfobox} \\
    \gexp{I_{3m}}{n} \ar[ur, dotted, "\f"]
\end{tikzcd} \]
commutes.
We show that every map $\dfobox \to \gnerve G$ must factor through some $m$-nerve $\dfobox \to \gnerve[m] G \to \gnerve G$.
The map $\f$ gives a filler for the composite $\dfobox \to \gnerve[m] G \to \gnerve[3m] G$, thus a filler in $\gnerve G$.

To show that the map $\gnerve[1] G \ito \gnerve G$ is anodyne, we show the maps $l^*, r^* \from \gnerve[m] G \ito \gnerve[m+1] G$ are anodyne.
This is done by an explicit construction establishing $l^*$ and $r^*$ as a transfinite composition of pushouts along coproducts of open box inclusions.
This implies $\gnerve[1] G \ito \gnerve G$ is anodyne as the nerve of $G$ is a transfinite composition of $l^*$ and $r^*$.

\iffalse
We define a map $p_n \from I_{n+2} \to I_n$ by
\[ p_n(v_i) = \begin{cases}
    0 & v_i = 0 \\
    n & v_i = n + 2 \\
    n - 1 & \text{otherwise}.
\end{cases} \]
\fi

\subsection*{Proof of part (1)}
\begin{construction}
    Fix $m, n \geq 1$, $i \in \{ 1, \dots, n \}$, and $\varepsilon \in \{ 0, 1 \}$.
    We construct the map $\f \from \gexp{I_{3m}}{n} \to \reali[m]{\dfobox}$. For $n = 1$, we compute $\reali[m]{\dfobox[1]} \cong I_0$.
    In this case, the map $\f$ is immediate. 
    Thus, we assume $n \geq 2$.

    Define a map $\varphi \from I_{3m} \to I_{3m}$ by
    \[ \varphi v := \begin{cases}
        m & v \leq m \\
        i & m \leq v \leq 2m \\
        2m & v \geq 2m.
    \end{cases} \]
    It is straightforward to verify this definition gives a graph map.
    From this, we have a map $\gexp{\varphi}{n-1} \from \gexp{I_{3m}}{n-1} \to \gexp{I_{3m}}{n-1}$ which sends $(v_1, \dots, v_{n-1})$ to $(\varphi v_1, \dots, \varphi v_{n-1})$.
    
    For $v = (v_1, \dots, v_{n-1}) \in \gexp{I_{3m}}{n-1}$, let $d(v)$ denote the node distance between $v$ and $\gexp{\varphi}{n-1}v$.
    That is,
    \[ d(v) := \sum\limits_{i=1}^{n-1} |v_i - \varphi v_i|. \]
    Observe this gives a graph map $d \from \gexp{I_{3m}}{n-1} \to I_\infty$.

    \begin{figure}[ht]
        \centering
        \begin{tikzpicture}[node distance=20pt]
            % First row of vertices
            \node[vertex, minimum size=16pt] (00) {4};
            \node[vertex, minimum size=16pt] (01) [right=of 00] {3};
            \node[vertex, minimum size=16pt] (02) [right=of 01] {2};
            \node[vertex, minimum size=16pt] (03) [right=of 02] {2};
            \node[vertex, minimum size=16pt] (04) [right=of 03] {2};
            \node[vertex, minimum size=16pt] (05) [right=of 04] {3};
            \node[vertex, minimum size=16pt] (06) [right=of 05] {4};
            % Second row
            \node[vertex, minimum size=16pt] (10) [below=of 00] {3};
            \node[vertex, minimum size=16pt] (11) [right=of 10] {2};
            \node[vertex, minimum size=16pt] (12) [right=of 11] {1};
            \node[vertex, minimum size=16pt] (13) [right=of 12] {1};
            \node[vertex, minimum size=16pt] (14) [right=of 13] {1};
            \node[vertex, minimum size=16pt] (15) [right=of 14] {2};
            \node[vertex, minimum size=16pt] (16) [right=of 15] {3};
            % Third row
            \node[vertex, minimum size=16pt] (20) [below=of 10] {2};
            \node[vertex, minimum size=16pt] (21) [right=of 20] {1};
            \node[vertex, minimum size=16pt] (22) [right=of 21] {0};
            \node[vertex, minimum size=16pt] (23) [right=of 22] {0};
            \node[vertex, minimum size=16pt] (24) [right=of 23] {0};
            \node[vertex, minimum size=16pt] (25) [right=of 24] {1};
            \node[vertex, minimum size=16pt] (26) [right=of 25] {2};
            % Fourth row
            \node[vertex, minimum size=16pt] (30) [below=of 20] {2};
            \node[vertex, minimum size=16pt] (31) [right=of 30] {1};
            \node[vertex, minimum size=16pt] (32) [right=of 31] {0};
            \node[vertex, minimum size=16pt] (33) [right=of 32] {0};
            \node[vertex, minimum size=16pt] (34) [right=of 33] {0};
            \node[vertex, minimum size=16pt] (35) [right=of 34] {1};
            \node[vertex, minimum size=16pt] (36) [right=of 35] {2};
            % Fifth row
            \node[vertex, minimum size=16pt] (40) [below=of 30] {2};
            \node[vertex, minimum size=16pt] (41) [right=of 40] {1};
            \node[vertex, minimum size=16pt] (42) [right=of 41] {0};
            \node[vertex, minimum size=16pt] (43) [right=of 42] {0};
            \node[vertex, minimum size=16pt] (44) [right=of 43] {0};
            \node[vertex, minimum size=16pt] (45) [right=of 44] {1};
            \node[vertex, minimum size=16pt] (46) [right=of 45] {2};
            % Sixth row
            \node[vertex, minimum size=16pt] (50) [below=of 40] {3};
            \node[vertex, minimum size=16pt] (51) [right=of 50] {2};
            \node[vertex, minimum size=16pt] (52) [right=of 51] {1};
            \node[vertex, minimum size=16pt] (53) [right=of 52] {1};
            \node[vertex, minimum size=16pt] (54) [right=of 53] {1};
            \node[vertex, minimum size=16pt] (55) [right=of 54] {2};
            \node[vertex, minimum size=16pt] (56) [right=of 55] {3};
            % Seventh row
            \node[vertex, minimum size=16pt] (60) [below=of 50] {4};
            \node[vertex, minimum size=16pt] (61) [right=of 60] {3};
            \node[vertex, minimum size=16pt] (62) [right=of 61] {2};
            \node[vertex, minimum size=16pt] (63) [right=of 62] {2};
            \node[vertex, minimum size=16pt] (64) [right=of 63] {2};
            \node[vertex, minimum size=16pt] (65) [right=of 64] {3};
            \node[vertex, minimum size=16pt] (66) [right=of 65] {4};
            % Horizontal edges
            \draw (00) to (01) to (02) to (03) to (04) to (05) to (06);
            \draw (10) to (11) to (12) to (13) to (14) to (15) to (16);
            \draw (20) to (21) to (22) to (23) to (24) to (25) to (26);
            \draw (30) to (31) to (32) to (33) to (34) to (35) to (36);
            \draw (40) to (41) to (42) to (43) to (44) to (45) to (46);
            \draw (50) to (51) to (52) to (53) to (54) to (55) to (56);
            \draw (60) to (61) to (62) to (63) to (64) to (65) to (66);
            % Vertical edges
            \draw (00) to (10) to (20) to (30) to (40) to (50) to (60);
            \draw (01) to (11) to (21) to (31) to (41) to (51) to (61);
            \draw (02) to (12) to (22) to (32) to (42) to (52) to (62);
            \draw (03) to (13) to (23) to (33) to (43) to (53) to (63);
            \draw (04) to (14) to (24) to (34) to (44) to (54) to (64);
            \draw (05) to (15) to (25) to (35) to (45) to (55) to (65);
            \draw (06) to (16) to (26) to (36) to (46) to (56) to (66);
        \end{tikzpicture}
        \caption{Vertices of $\gexp{I_6}{2}$ labelled by their image under $d \from \gexp{I_6}{2} \to I_\infty$.}
    \end{figure}

    For $t \geq 0$, we have a graph map $\bound{t} \from I_\infty \to I_{t}$ defined by
    \[ \bound{t}(v) := \begin{cases}
        0 & v \leq 0 \\
        v & 0 \leq v \leq t \\
        t & v \geq t.
    \end{cases} \]
    One thinks of $\bound{t}$ as bounding the graph $I_\infty$ between 0 and $t$.
    Recall the map $c^m \from I_{3m} \to I_{m}$ is given by
    \[ c^m v := \begin{cases}
        0 & v \leq m \\
        v - m & m \leq v \leq 2m \\
        m & v \geq 2m.
    \end{cases} \]
    For a vertex $(v_1, \dots, v_n) \in \gexp{I_{3m}}{n}$, we write $\degen{}{i}v$ for the vertex $(v_1, \dots, v_{i-1}, v_{i+1}, \dots, v_n) \in \gexp{I_{3m}}{n-1}$.
    As well, let $\alpha^{v_i}_\varepsilon$ denote the value
    \[ \alpha^{v_i}_\varepsilon := \varepsilon m + (1 - 2\varepsilon)(c^m v_i). \]
    We may also write this as
    \[ \alpha^{v_i}_\varepsilon = \begin{cases}
        c^m v_i & \varepsilon = 0 \\
        m - c^m v_i & \varepsilon = 1.
    \end{cases} \]
    We define $\f \from \gexp{I_{3m}}{n} \to \reali[m]{\dfobox}$ by
    \[ \f (v_1, \dots, v_n) := (c^m v_1, \dots, c^m v_{i-1}, (1 - \varepsilon)m + (2\varepsilon - 1)(\bound{\alpha^{v_i}_{(1 - \varepsilon)}}(d(\degen{}{i}v) - \alpha^{v_i}_\varepsilon)), c^m v_{i+1}, \dots, c^m v_{n}). \]
    That is,
    \[ \f (v_1, \dots, v_n) = \begin{cases}
        (c^m v_1, \dots, c^m v_{i-1}, m - \bound{m - c^m v_i}(d(\degen{}{i}v) - c^m v_i), c^m v_{i+1}, \dots, c^m v_{n}) & \text{ if } \varepsilon = 0 \\
        (c^m v_1, \dots, c^m v_{i-1}, \bound{c^m v_i}(d(\degen{}{i}v) + c^m v_i - m), c^m v_{i+1}, \dots, c^m v_{n}) & \text{ if } \varepsilon = 1.
    \end{cases} \]
    We first show this formula is well-defined, i.e.~that this tuple lies in the subgraph $\reali[m]{\dfobox}$ of $\gexp{I_m}{n}$.
    Observe that if $c^m v_k = 0$ or $m$ for some $k \neq i$ in $\{ 1, \dots, n \}$ then this tuple indeed lies in $\reali[m]{\dfobox}$.
    For $k < i$, this tuple lies on the $(k, 0)$ or $(k, 1)$-face.
    For $k > i$, this tuple lies on the $(k+1,0)$ or $(k+1,1)$-face.
    Otherwise, if $0 < c^m v_k < m$ for all $k \neq i$ then $d(\degen{}{i}v) = 0$ since $v_k = \varphi v_k$ for all $k \neq i$.
    From this, it follows that 
    \[ \bound{\alpha^{v_i}_{1 - \varepsilon}}(d(\degen{}{i}v) - \alpha^{v_i}_{\varepsilon}) = \bound{\alpha^{v_i}_{1 - \varepsilon}}(-\alpha^{v_i}_{\varepsilon}) = 0, \]
    thus $\f (v_1, \dots, v_n)$ lies on the $(i,1-\varepsilon)$-face; that is, the face opposite the missing face.
    
    To see this formula gives a graph map, suppose $(v_1, \dots, v_n)$ and $(w_1, \dots, w_n)$ are connected vertices in $\gexp{I_{3m}}{n}$.
    By definition, there exists $k = 1, \dots, n$ so that $v_k$ and $w_k$ are connected in $I_{3m}$ and $v_j = w_j$ for all $j \neq k$.
    We first consider the case where $k \neq i$.
    Observe that if $c^m v_k = c^m w_k$ then this is immediate.
    If $c^m v_k \neq c^m w_k$ then $\degen{}{i}v = \varphi(\degen{}{i}v)$.
    This gives that $d(\degen{}{i}v) = d(\degen{}{i}w)$, hence $\f(v_1, \dots, v_n)$ and $\f(w_1, \dots, w_n)$ are equal on all components except the $k$-th component.
    That is, they are connected.
    In the case where $k = i$, we have that $d v$ and $d w$ differ by at most 1.
    This implies $(1 - e)m + (2\varepsilon - 1)(\bound{\alpha^{v_i}_{1-\varepsilon}}(dv - \alpha^{v_i}_\varepsilon))$ and $(1 - e)m + (2\varepsilon - 1)(\bound{\alpha^{v_i}_{1-\varepsilon}}(dw - \alpha^{v_i}_\varepsilon))$ differ by at most 1, thus $\f (v_1, \dots, v_n)$ and $\f (w_1, \dots, w_n)$ are connected.
\end{construction}

\begin{example}
    We look at the map $\f \from \gexp{I_{3m}}{n} \to \reali[m]{\dfobox}$ in the case of $n = 3$, $m = 2$, and $(i,\varepsilon) = (3, 1)$.
    \begin{figure}[ht]
        \centering
        \begin{tikzpicture}[node distance=20pt]
            % vertices
            % Back face
            \node[vertex, fill=red] (BTL) {};
            \node[vertex, fill=orange] (BTM) [right=of BTL] {};
            \node[vertex, fill=yellow] (BTR) [right=of BTM] {};
            \node[vertex, fill=black!40!red] (BML) [below=of BTL] {};
            \node[vertex, fill=black!40!orange] (BMM) [right=of BML] {};
            \node[vertex, fill=black!40!yellow] (BMR) [right=of BMM] {};
            % Top face
            \node[vertex, fill=white!50!red] (TL) [below=of BML] {};
            \node[vertex, fill=white!50!orange] (TM) [right=of TL] {};
            \node[vertex, fill=white!50!yellow] (TR) [right=of TM] {};
            \node[vertex, fill=red!20!white!70!pink] (ML) [below=of TL] {};
            \node[vertex, fill=white] (MM) [below=of TM] {};
            \node[vertex, fill=white!70!green] (MR) [below=of TR] {};
            \node[vertex, fill=white!70!violet] (BL) [below=of ML] {};
            \node[vertex, fill=white!70!blue] (BM) [below=of MM] {};
            \node[vertex, fill=white!70!cyan] (BR) [below=of MR] {};
            % Right face
            \node[vertex, fill=black!40!yellow] (RTM) [right=of TR] {};
            \node[vertex, fill=yellow] (RTR) [right=of RTM] {};
            \node[vertex, fill=black!40!green] (RMM) [right=of MR] {};
            \node[vertex, fill=green] (RMR) [right=of RMM] {};
            \node[vertex, fill=black!40!cyan] (RBM) [right=of BR] {};
            \node[vertex, fill=cyan] (RBR) [right=of RBM] {};
            % Left face
            \node[vertex, fill=black!40!red] (LTM) [left=of TL] {};
            \node[vertex, fill=red] (LTL) [left=of LTM] {};
            \node[vertex, fill=red!10!black!40!pink] (LMM) [left=of ML] {};
            \node[vertex, fill=red!10!pink] (LML) [left=of LMM] {};
            \node[vertex, fill=black!40!violet] (LBM) [left=of BL] {};
            \node[vertex, fill=violet] (LBL) [left=of LBM] {};
            % Front face
            \node[vertex, fill=black!40!violet] (FML) [below=of BL] {};
            \node[vertex, fill=black!40!blue] (FMM) [below=of BM] {};
            \node[vertex, fill=black!40!cyan] (FMR) [below=of BR] {};
            \node[vertex, fill=violet] (FBL) [below=of FML] {};
            \node[vertex, fill=blue] (FBM) [below=of FMM] {};
            \node[vertex, fill=cyan] (FBR) [below=of FMR] {};
            % edges
            \draw (BTL) to (BTM) % back face
                  (BTM) to (BTR)
                  (BTL) to (BML)
                  (BTM) to (BMM)
                  (BTR) to (BMR)
                  (BML) to (BMM)
                  (BMM) to (BMR)
                  (BML) to (TL)
                  (BMM) to (TM)
                  (BMR) to (TR)
                  (LTL) to (LTM) % left, top, middle face
                  (LTM) to (TL)
                  (TL) to (TM)
                  (TM) to (TR)
                  (TR) to (RTM)
                  (RTM) to (RTR)
                  (LTL) to (LML)
                  (LTM) to (LMM)
                  (TL) to (ML)
                  (TM) to (MM)
                  (TR) to (MR)
                  (RTM) to (RMM)
                  (RTR) to (RMR)
                  (LML) to (LMM)
                  (LMM) to (ML)
                  (ML) to (MM)
                  (MM) to (MR)
                  (MR) to (RMM)
                  (RMM) to (RMR)
                  (LML) to (LBL)
                  (LMM) to (LBM)
                  (ML) to (BL)
                  (MM) to (BM)
                  (MR) to (BR)
                  (RMM) to (RBM)
                  (RMR) to (RBR)
                  (LBL) to (LBM)
                  (LBM) to (BL)
                  (BL) to (BM)
                  (BM) to (BR)
                  (BR) to (RBM)
                  (RBM) to (RBR)
                  (BL) to (FML) % front face
                  (BM) to (FMM)
                  (BR) to (FMR)
                  (FML) to (FMM)
                  (FMM) to (FMR)
                  (FML) to (FBL)
                  (FMM) to (FBM)
                  (FMR) to (FBR)
                  (FBL) to (FBM)
                  (FBM) to (FBR);
            % identification lines
            \draw[dotted, shorten <=5pt, shorten >=5pt] (BTL) to (LTL);
            \draw[dotted, shorten <=5pt, shorten >=5pt] (BML) to (LTM);
            \draw[dotted, shorten <=5pt, shorten >=5pt] (BTR) to (RTR);
            \draw[dotted, shorten <=5pt, shorten >=5pt] (BMR) to (RTM);
            \draw[dotted, shorten <=5pt, shorten >=5pt] (LBL) to (FBL);
            \draw[dotted, shorten <=5pt, shorten >=5pt] (LBM) to (FML);
            \draw[dotted, shorten <=5pt, shorten >=5pt] (RBR) to (FBR);
            \draw[dotted, shorten <=5pt, shorten >=5pt] (RBM) to (FMR);
        \end{tikzpicture}
        \caption{The graph $\reali[2]{\obox{3}{3,1}}$ as a net. Vertices connected by a dotted line are identical.}
        \label{fig:obox}
    \end{figure}

    We may write the map $\f \from \gexp{I_6}{3} \to \reali[2]{\obox{3}{3,1}}$ as
    \[ \f (v_1, v_2, v_3) = \begin{cases}
        (c^2 v_1, c^2 v_2, 0) & \text{ if }v_3 \leq 2 \\
        (c^2 v_1, c^2 v_2, \beta[1](d(v_1, v_2) - 1)) & \text{ if } v_3 = 3 \\
        (c^2 v_1, c^2 v_2, \beta[2](d(v_1, v_2))) & \text{ if } v_3 \geq 4.
    \end{cases} \]
    \begin{figure}[ht]
        \centering
        \begin{subfigure}[b]{0.3\textwidth}
            \centering
            \begin{tikzpicture}[node distance=10pt]
                % vertices
                % Row 0
                \node[vertex, fill=white!50!red] (00) {};
                \node[vertex, fill=white!50!red] (10) [right=of 00] {};
                \node[vertex, fill=white!50!red] (20) [right=of 10] {};
                \node[vertex, fill=white!50!orange] (30) [right=of 20] {};
                \node[vertex, fill=white!50!yellow] (40) [right=of 30] {};
                \node[vertex, fill=white!50!yellow] (50) [right=of 40] {};
                \node[vertex, fill=white!50!yellow] (60) [right=of 50] {};
                % Row 1
                \node[vertex, fill=white!50!red] (01) [below=of 00] {};
                \node[vertex, fill=white!50!red] (11) [right=of 01] {};
                \node[vertex, fill=white!50!red] (21) [right=of 11] {};
                \node[vertex, fill=white!50!orange] (31) [right=of 21] {};
                \node[vertex, fill=white!50!yellow] (41) [right=of 31] {};
                \node[vertex, fill=white!50!yellow] (51) [right=of 41] {};
                \node[vertex, fill=white!50!yellow] (61) [right=of 51] {};
                % Row 2
                \node[vertex, fill=white!50!red] (02) [below=of 01] {};
                \node[vertex, fill=white!50!red] (12) [right=of 02] {};
                \node[vertex, fill=white!50!red] (22) [right=of 12] {};
                \node[vertex, fill=white!50!orange] (32) [right=of 22] {};
                \node[vertex, fill=white!50!yellow] (42) [right=of 32] {};
                \node[vertex, fill=white!50!yellow] (52) [right=of 42] {};
                \node[vertex, fill=white!50!yellow] (62) [right=of 52] {};
                % Row 3
                \node[vertex, fill=red!20!white!70!pink] (03) [below=of 02] {};
                \node[vertex, fill=red!20!white!70!pink] (13) [right=of 03] {};
                \node[vertex, fill=red!20!white!70!pink] (23) [right=of 13] {};
                \node[vertex] (33) [right=of 23] {};
                \node[vertex, fill=white!70!green] (43) [right=of 33] {};
                \node[vertex, fill=white!70!green] (53) [right=of 43] {};
                \node[vertex, fill=white!70!green] (63) [right=of 53] {};
                % Row 4
                \node[vertex, fill=white!70!violet] (04) [below=of 03] {};
                \node[vertex, fill=white!70!violet] (14) [right=of 04] {};
                \node[vertex, fill=white!70!violet] (24) [right=of 14] {};
                \node[vertex, fill=white!70!blue] (34) [right=of 24] {};
                \node[vertex, fill=white!70!cyan] (44) [right=of 34] {};
                \node[vertex, fill=white!70!cyan] (54) [right=of 44] {};
                \node[vertex, fill=white!70!cyan] (64) [right=of 54] {};
                % Row 5
                \node[vertex, fill=white!70!violet] (05) [below=of 04] {};
                \node[vertex, fill=white!70!violet] (15) [right=of 05] {};
                \node[vertex, fill=white!70!violet] (25) [right=of 15] {};
                \node[vertex, fill=white!70!blue] (35) [right=of 25] {};
                \node[vertex, fill=white!70!cyan] (45) [right=of 35] {};
                \node[vertex, fill=white!70!cyan] (55) [right=of 45] {};
                \node[vertex, fill=white!70!cyan] (65) [right=of 55] {};
                % Row 6
                \node[vertex, fill=white!70!violet] (06) [below=of 05] {};
                \node[vertex, fill=white!70!violet] (16) [right=of 06] {};
                \node[vertex, fill=white!70!violet] (26) [right=of 16] {};
                \node[vertex, fill=white!70!blue] (36) [right=of 26] {};
                \node[vertex, fill=white!70!cyan] (46) [right=of 36] {};
                \node[vertex, fill=white!70!cyan] (56) [right=of 46] {};
                \node[vertex, fill=white!70!cyan] (66) [right=of 56] {};
                
                % Edges
                \draw (00) to (10) % Row 0
                      (10) to (20)
                      (20) to (30)
                      (30) to (40)
                      (40) to (50)
                      (50) to (60)
                      (00) to (01)
                      (10) to (11)
                      (20) to (21)
                      (30) to (31)
                      (40) to (41)
                      (50) to (51)
                      (60) to (61)
                      (01) to (11) % Row 1
                      (11) to (21)
                      (21) to (31)
                      (31) to (41)
                      (41) to (51)
                      (51) to (61)
                      (01) to (02)
                      (11) to (12)
                      (21) to (22)
                      (31) to (32)
                      (41) to (42)
                      (51) to (52)
                      (61) to (62)
                      (02) to (12) % Row 2
                      (12) to (22)
                      (22) to (32)
                      (32) to (42)
                      (42) to (52)
                      (52) to (62)
                      (02) to (03)
                      (12) to (13)
                      (22) to (23)
                      (32) to (33)
                      (42) to (43)
                      (52) to (53)
                      (62) to (63)
                      (03) to (13) % Row 3
                      (13) to (23)
                      (23) to (33)
                      (33) to (43)
                      (43) to (53)
                      (53) to (63)
                      (03) to (04)
                      (13) to (14)
                      (23) to (24)
                      (33) to (34)
                      (43) to (44)
                      (53) to (54)
                      (63) to (64)
                      (04) to (14) % Row 4
                      (14) to (24)
                      (24) to (34)
                      (34) to (44)
                      (44) to (54)
                      (54) to (64)
                      (04) to (05)
                      (14) to (15)
                      (24) to (25)
                      (34) to (35)
                      (44) to (45)
                      (54) to (55)
                      (64) to (65)
                      (05) to (15) % Row 5
                      (15) to (25)
                      (25) to (35)
                      (35) to (45)
                      (45) to (55)
                      (55) to (65)
                      (05) to (06)
                      (15) to (16)
                      (25) to (26)
                      (35) to (36)
                      (45) to (46)
                      (55) to (56)
                      (65) to (66)
                      (06) to (16) % Row 6
                      (16) to (26)
                      (26) to (36)
                      (36) to (46)
                      (46) to (56)
                      (56) to (66);
                % Colored edges
                \path (11) to node[anchor=center] (P1) {} (22)
                      (51) to node[anchor=center] (P2) {} (42) 
                      (55) to node[anchor=center] (P3) {} (44)
                      (15) to node[anchor=center] (P4) {} (24);
                % \draw[blue] (P1) rectangle (P3);
                % \draw[blue] (22) to (32)
                %       (32) to (42)
                %       (22) to (23)
                %       (32) to (33)
                %       (42) to (43)
                %       (23) to (33)
                %       (33) to (43)
                %       (23) to (24)
                %       (33) to (34)
                %       (43) to (44)
                %       (24) to (34)
                %       (34) to (44);
            \end{tikzpicture}
            \caption{$v_3 \leq 2$.}
        \end{subfigure}
        \hfill
        \begin{subfigure}[b]{0.3\textwidth}
            \centering
            \begin{tikzpicture}[node distance=10pt]
                % vertices
                % Row 0
                \node[vertex, fill=black!40!red] (00) {};
                \node[vertex, fill=black!40!red] (10) [right=of 00] {};
                \node[vertex, fill=black!40!red] (20) [right=of 10] {};
                \node[vertex, fill=black!40!orange] (30) [right=of 20] {};
                \node[vertex, fill=black!40!yellow] (40) [right=of 30] {};
                \node[vertex, fill=black!40!yellow] (50) [right=of 40] {};
                \node[vertex, fill=black!40!yellow] (60) [right=of 50] {};
                % Row 1
                \node[vertex, fill=black!40!red] (01) [below=of 00] {};
                \node[vertex, fill=black!40!red] (11) [right=of 01] {};
                \node[vertex, fill=white!50!red] (21) [right=of 11] {};
                \node[vertex, fill=white!50!orange] (31) [right=of 21] {};
                \node[vertex, fill=white!50!yellow] (41) [right=of 31] {};
                \node[vertex, fill=black!40!yellow] (51) [right=of 41] {};
                \node[vertex, fill=black!40!yellow] (61) [right=of 51] {};
                % Row 2
                \node[vertex, fill=black!40!red] (02) [below=of 01] {};
                \node[vertex, fill=white!50!red] (12) [right=of 02] {};
                \node[vertex, fill=white!50!red] (22) [right=of 12] {};
                \node[vertex, fill=white!50!orange] (32) [right=of 22] {};
                \node[vertex, fill=white!50!yellow] (42) [right=of 32] {};
                \node[vertex, fill=white!50!yellow] (52) [right=of 42] {};
                \node[vertex, fill=black!40!yellow] (62) [right=of 52] {};
                % Row 3
                \node[vertex, fill=red!10!black!40!pink] (03) [below=of 02] {};
                \node[vertex, fill=red!20!white!70!pink] (13) [right=of 03] {};
                \node[vertex, fill=red!20!white!70!pink] (23) [right=of 13] {};
                \node[vertex] (33) [right=of 23] {};
                \node[vertex, fill=white!70!green] (43) [right=of 33] {};
                \node[vertex, fill=white!70!green] (53) [right=of 43] {};
                \node[vertex, fill=black!40!green] (63) [right=of 53] {};
                % Row 4
                \node[vertex, fill=black!40!violet] (04) [below=of 03] {};
                \node[vertex, fill=white!70!violet] (14) [right=of 04] {};
                \node[vertex, fill=white!70!violet] (24) [right=of 14] {};
                \node[vertex, fill=white!70!blue] (34) [right=of 24] {};
                \node[vertex, fill=white!70!cyan] (44) [right=of 34] {};
                \node[vertex, fill=white!70!cyan] (54) [right=of 44] {};
                \node[vertex, fill=black!40!cyan] (64) [right=of 54] {};
                % Row 5
                \node[vertex, fill=black!40!violet] (05) [below=of 04] {};
                \node[vertex, fill=black!40!violet] (15) [right=of 05] {};
                \node[vertex, fill=white!70!violet] (25) [right=of 15] {};
                \node[vertex, fill=white!70!blue] (35) [right=of 25] {};
                \node[vertex, fill=white!70!cyan] (45) [right=of 35] {};
                \node[vertex, fill=black!40!cyan] (55) [right=of 45] {};
                \node[vertex, fill=black!40!cyan] (65) [right=of 55] {};
                % Row 6
                \node[vertex, fill=black!40!violet] (06) [below=of 05] {};
                \node[vertex, fill=black!40!violet] (16) [right=of 06] {};
                \node[vertex, fill=black!40!violet] (26) [right=of 16] {};
                \node[vertex, fill=black!40!blue] (36) [right=of 26] {};
                \node[vertex, fill=black!40!cyan] (46) [right=of 36] {};
                \node[vertex, fill=black!40!cyan] (56) [right=of 46] {};
                \node[vertex, fill=black!40!cyan] (66) [right=of 56] {};
                
                % Edges
                \draw (00) to (10) % Row 0
                      (10) to (20)
                      (20) to (30)
                      (30) to (40)
                      (40) to (50)
                      (50) to (60)
                      (00) to (01)
                      (10) to (11)
                      (20) to (21)
                      (30) to (31)
                      (40) to (41)
                      (50) to (51)
                      (60) to (61)
                      (01) to (11) % Row 1
                      (11) to (21)
                      (21) to (31)
                      (31) to (41)
                      (41) to (51)
                      (51) to (61)
                      (01) to (02)
                      (11) to (12)
                      (21) to (22)
                      (31) to (32)
                      (41) to (42)
                      (51) to (52)
                      (61) to (62)
                      (02) to (12) % Row 2
                      (12) to (22)
                      (22) to (32)
                      (32) to (42)
                      (42) to (52)
                      (52) to (62)
                      (02) to (03)
                      (12) to (13)
                      (22) to (23)
                      (32) to (33)
                      (42) to (43)
                      (52) to (53)
                      (62) to (63)
                      (03) to (13) % Row 3
                      (13) to (23)
                      (23) to (33)
                      (33) to (43)
                      (43) to (53)
                      (53) to (63)
                      (03) to (04)
                      (13) to (14)
                      (23) to (24)
                      (33) to (34)
                      (43) to (44)
                      (53) to (54)
                      (63) to (64)
                      (04) to (14) % Row 4
                      (14) to (24)
                      (24) to (34)
                      (34) to (44)
                      (44) to (54)
                      (54) to (64)
                      (04) to (05)
                      (14) to (15)
                      (24) to (25)
                      (34) to (35)
                      (44) to (45)
                      (54) to (55)
                      (64) to (65)
                      (05) to (15) % Row 5
                      (15) to (25)
                      (25) to (35)
                      (35) to (45)
                      (45) to (55)
                      (55) to (65)
                      (05) to (06)
                      (15) to (16)
                      (25) to (26)
                      (35) to (36)
                      (45) to (46)
                      (55) to (56)
                      (65) to (66)
                      (06) to (16) % Row 6
                      (16) to (26)
                      (26) to (36)
                      (36) to (46)
                      (46) to (56)
                      (56) to (66);
            \end{tikzpicture}
            \caption{$v_3 = 3$}
        \end{subfigure}
        \hfill
        \begin{subfigure}[b]{0.3\textwidth}
            \centering
            \begin{tikzpicture}[node distance=10pt]
                % vertices
                % Row 0
                \node[vertex, fill=red] (00) {};
                \node[vertex, fill=red] (10) [right=of 00] {};
                \node[vertex, fill=red] (20) [right=of 10] {};
                \node[vertex, fill=orange] (30) [right=of 20] {};
                \node[vertex, fill=yellow] (40) [right=of 30] {};
                \node[vertex, fill=yellow] (50) [right=of 40] {};
                \node[vertex, fill=yellow] (60) [right=of 50] {};
                % Row 1
                \node[vertex, fill=red] (01) [below=of 00] {};
                \node[vertex, fill=red] (11) [right=of 01] {};
                \node[vertex, fill=black!40!red] (21) [right=of 11] {};
                \node[vertex, fill=black!40!orange] (31) [right=of 21] {};
                \node[vertex, fill=black!40!yellow] (41) [right=of 31] {};
                \node[vertex, fill=yellow] (51) [right=of 41] {};
                \node[vertex, fill=yellow] (61) [right=of 51] {};
                % Row 2
                \node[vertex, fill=red] (02) [below=of 01] {};
                \node[vertex, fill=black!40!red] (12) [right=of 02] {};
                \node[vertex, fill=white!50!red] (22) [right=of 12] {};
                \node[vertex, fill=white!50!orange] (32) [right=of 22] {};
                \node[vertex, fill=white!50!yellow] (42) [right=of 32] {};
                \node[vertex, fill=black!40!yellow] (52) [right=of 42] {};
                \node[vertex, fill=yellow] (62) [right=of 52] {};
                % Row 3
                \node[vertex, fill=red!10!pink] (03) [below=of 02] {};
                \node[vertex, fill=red!10!black!40!pink] (13) [right=of 03] {};
                \node[vertex, fill=red!20!white!70!pink] (23) [right=of 13] {};
                \node[vertex] (33) [right=of 23] {};
                \node[vertex, fill=white!70!green] (43) [right=of 33] {};
                \node[vertex, fill=black!40!green] (53) [right=of 43] {};
                \node[vertex, fill=green] (63) [right=of 53] {};
                % Row 4
                \node[vertex, fill=violet] (04) [below=of 03] {};
                \node[vertex, fill=black!40!violet] (14) [right=of 04] {};
                \node[vertex, fill=white!70!violet] (24) [right=of 14] {};
                \node[vertex, fill=white!70!blue] (34) [right=of 24] {};
                \node[vertex, fill=white!70!cyan] (44) [right=of 34] {};
                \node[vertex, fill=black!40!cyan] (54) [right=of 44] {};
                \node[vertex, fill=cyan] (64) [right=of 54] {};
                % Row 5
                \node[vertex, fill=violet] (05) [below=of 04] {};
                \node[vertex, fill=violet] (15) [right=of 05] {};
                \node[vertex, fill=black!40!violet] (25) [right=of 15] {};
                \node[vertex, fill=black!40!blue] (35) [right=of 25] {};
                \node[vertex, fill=black!40!cyan] (45) [right=of 35] {};
                \node[vertex, fill=cyan] (55) [right=of 45] {};
                \node[vertex, fill=cyan] (65) [right=of 55] {};
                % Row 6
                \node[vertex, fill=violet] (06) [below=of 05] {};
                \node[vertex, fill=violet] (16) [right=of 06] {};
                \node[vertex, fill=violet] (26) [right=of 16] {};
                \node[vertex, fill=blue] (36) [right=of 26] {};
                \node[vertex, fill=cyan] (46) [right=of 36] {};
                \node[vertex, fill=cyan] (56) [right=of 46] {};
                \node[vertex, fill=cyan] (66) [right=of 56] {};
                
                % Edges
                \draw (00) to (10) % Row 0
                      (10) to (20)
                      (20) to (30)
                      (30) to (40)
                      (40) to (50)
                      (50) to (60)
                      (00) to (01)
                      (10) to (11)
                      (20) to (21)
                      (30) to (31)
                      (40) to (41)
                      (50) to (51)
                      (60) to (61)
                      (01) to (11) % Row 1
                      (11) to (21)
                      (21) to (31)
                      (31) to (41)
                      (41) to (51)
                      (51) to (61)
                      (01) to (02)
                      (11) to (12)
                      (21) to (22)
                      (31) to (32)
                      (41) to (42)
                      (51) to (52)
                      (61) to (62)
                      (02) to (12) % Row 2
                      (12) to (22)
                      (22) to (32)
                      (32) to (42)
                      (42) to (52)
                      (52) to (62)
                      (02) to (03)
                      (12) to (13)
                      (22) to (23)
                      (32) to (33)
                      (42) to (43)
                      (52) to (53)
                      (62) to (63)
                      (03) to (13) % Row 3
                      (13) to (23)
                      (23) to (33)
                      (33) to (43)
                      (43) to (53)
                      (53) to (63)
                      (03) to (04)
                      (13) to (14)
                      (23) to (24)
                      (33) to (34)
                      (43) to (44)
                      (53) to (54)
                      (63) to (64)
                      (04) to (14) % Row 4
                      (14) to (24)
                      (24) to (34)
                      (34) to (44)
                      (44) to (54)
                      (54) to (64)
                      (04) to (05)
                      (14) to (15)
                      (24) to (25)
                      (34) to (35)
                      (44) to (45)
                      (54) to (55)
                      (64) to (65)
                      (05) to (15) % Row 5
                      (15) to (25)
                      (25) to (35)
                      (35) to (45)
                      (45) to (55)
                      (55) to (65)
                      (05) to (06)
                      (15) to (16)
                      (25) to (26)
                      (35) to (36)
                      (45) to (46)
                      (55) to (56)
                      (65) to (66)
                      (06) to (16) % Row 6
                      (16) to (26)
                      (26) to (36)
                      (36) to (46)
                      (46) to (56)
                      (56) to (66);
            \end{tikzpicture}
            \caption{$v_3 \geq 4$}
        \end{subfigure}
        \caption{The map $\f \from \gexp{I_6}{3} \to \reali[2]{\obox{3}{3,1}}$ split into cross-sections. Vertices are colored by their image as in \cref{fig:obox}.}
    \end{figure}
    If $v_3 \leq 2$ then $(v_1, v_2, v_3)$ is contained in the $\face{}{3,0}$-face of $\reali[2]{\obox{3}{3,1}}$, which is the face opposite the missing face.
    For the cross-section where $v_3 = 3$, if $d(v_1, v_2) \geq 2$ then $(v_1, v_2, v_3)$ is sent to a vertex in the $v_3 = 1$ cross-section of $\reali[2]{\obox{3}{3,1}}$.
    For $v_3 \geq 3$, if $d(v_1, v_2) = 1$ then $(v_1, v_2, v_3)$ is sent to a vertex in the $v_3 = 1$ cross-section.
    If $d(v_1, v_2) \geq 2$ then $(v_1, v_2, v_3)$ is sent to a vertex in the $v_3 = 2$ cross-section.
    In all three cross-section, if $d(v_1, v_2) = 0$ then $(v_1, v_2, v_3)$ is contained in the face opposite the missing face of $\reali[2]{\obox{3}{3,1}}$.
\end{example}
\begin{lemma} \label{thm:cube_obox_map}
    The diagram
    \[ \begin{tikzcd}
        \reali[3m]{\dfobox} \ar[r, "(c_*)^m"] \ar[d, hook] & \reali[m]{\dfobox} \\
        \gexp{I_{3m}}{n} \ar[ur, "\f"]
    \end{tikzcd} \]
    commutes.
\end{lemma}
\begin{proof}
    For $n = 1$, we have $\reali[m]{\dfobox[1]} \cong I_0$.
    The diagram then commutes as $I_0$ is terminal in $\Graph$.

    For $n \geq 2$, fix $(v_1, \dots, v_n) \in \reali[3m]{\dfobox}$.
    It suffices to show
    \[ (1 - \varepsilon)m + (2\varepsilon - 1)(\bound{\alpha^{v_i}_{1 - \varepsilon}}(d(\degen{}{i}v) - \alpha^{v_i}_\varepsilon)) = c^m v_i. \]  
    If $v_i = (1 - \varepsilon)(3m)$ then $d(\degen{}{i}v) = 0$.
    Thus,
    \begin{align*}
        (1 - \varepsilon)m + (2\varepsilon - 1)(\bound{\alpha^{v_i}_{1 - \varepsilon}}(d(\degen{}{i}v) - \alpha^{v_i}_\varepsilon)) &= (1 - \varepsilon)m + (2\varepsilon - 1)(\bound{\alpha^{v_i}_{1 - \varepsilon}}(-\alpha^{v_i}_\varepsilon)) \\
        &= (1 - \varepsilon)m + (2\varepsilon - 1)(0) \\
        &= (1 - \varepsilon)m \\
        &= c^m ((1 - \varepsilon)(3m)) \\
        &= c^m v_i.
    \end{align*}
    Otherwise, if $v_i \neq (1 - \varepsilon)(3m)$ then $d(\degen{}{i}v) \geq m$ (since $v_k = 0, 3m$ for some $k \neq i$).
    For $\varepsilon = 0$, we compute
    \begin{align*}
        (1 - \varepsilon)m + (2\varepsilon - 1)(\bound{\alpha^{v_i}_{1 - \varepsilon}}(d(\degen{}{i}v) - \alpha^{v_i}_\varepsilon)) &= m - \bound{\alpha^{v_i}_1}(d(\degen{}{i}v) - \alpha^{v_i}_0) \\
        &= m - \bound{m - c^m v_i}(d(\degen{}{i}v) - c^m v_i) \\
        &= m - (m - c^m v_i) \\
        &= c^m v_i.
    \end{align*}
    For $\varepsilon = 1$, we compute
    \begin{align*}
        (1 - \varepsilon)m + (2\varepsilon - 1)(\bound{\alpha^{v_i}_{1 - \varepsilon}}(d(\degen{}{i}v) - \alpha^{v_i}_\varepsilon)) &= \bound{\alpha^{v_i}_0}(d(\degen{}{i}v) - \alpha^{v_i}_1) \\
        &= \bound{c^m v_i}(d(\degen{}{i}v) + c^m v_i - m) \\
        &= c^m v_i. \qedhere
    \end{align*}
\end{proof}
\begin{theorem} \label{thm:nerve_kan}
    For any graph $G$, the nerve $\gnerve G$ of $G$ is a Kan complex.
\end{theorem}
\begin{proof}
    Fix a map $f \from \dfobox \to \gnerve{G}$.
    We know that $\gnerve{G} \cong \colim(\gnerve[1]{G} \ito \gnerve[3]{G} \ito \gnerve[5]{G} \ito \dots)$ by \cref{thm:nerve_colim}.
    Recall that in a presheaf category, any map from a representable presheaf to a colimit must factor through some component of the colimit cone by the Yoneda lemma.
    Thus, for any $k \geq 0$ and $x \from \cube{k} \to \dfobox$, the map $fx \from \cube{k} \to \gnerve{G}$ factors through an inclusion $\gnerve[m] G \ito \gnerve G$ for some $m \geq 1$.
    As $\dfobox$ has only finitely many non-degenerate cubes, $f$ factors through the natural inclusion $\gnerve[m]{G} \ito \gnerve G$ as a map $g \from \dfobox \to \gnerve[m]{G}$ for some $m \geq 0$.

    By adjointness, $g$ corresponds to a map $\overline{g} \from \reali[m]{\dfobox} \to G$.
    \cref{thm:cube_obox_map} shows that $\overline{g}\f \from \gexp{I_{3m}}{n} \to G$ is a lift of the composite map $\overline{g}(c^*)^m  \from \reali[3m]{\dfobox} \to G$.
    \[ \begin{tikzcd}
        \reali[3m]{\dfobox} \ar[r, "(c_*)^m"] \ar[d, hook] & \reali[m]{\dfobox} \ar[r, "\overline{g}"] & G \\
        \reali[3m]{\cube{n}} \ar[ur, "\f"']
    \end{tikzcd} \]
    By adjointness, this gives a filler $\cube{n} \to \gnerve[3m]{G}$ for the map $(c^*)^m g \from \dfobox \to \gnerve[3m]{G}$.
    Post-composing with the natural inclusion $\gnerve[3m]{G} \ito \gnerve G$, this gives a filler $\cube{n} \to \gnerve{G}$ of $f$.
    % Suppose we have a lift $\cube{n} \to \gnerve[3m]{G}$ of $(c^*)^m g \from \dfobox \to \gnerve[3m]{G}$.
    % That is, we have a map $\cube{n} \to \gnerve[3m]{G}$ which makes the following diagram commute.
    % \[ \begin{tikzcd}
    %     \dfobox \ar[r, "g"] \ar[d, hook] & \gnerve[m] G \ar[r, hook, "(c^*)^m"] & \gnerve[3m] G \\
    %     \cube{n} \ar[urr, dotted]
    % \end{tikzcd} \]
    % This gives a lift of $f$ by post-composing with the natural map $\gnerve[3m] G \ito \gnerve G$. 
    % By adjointness, this corresponds to a map $\reali[3m]{\cube{n}} \cong \gexp{I_{3m}}{n} \to G$ which makes the following diagram commute.
    % \[ \begin{tikzcd}
    %     \reali[3m]{\dfobox} \ar[r, "(c_*)^m"] \ar[d, hook] & \reali[m]{\dfobox} \ar[r, "\overline{g}"] & G \\
    %     \reali[3m]{\cube{n}} \ar[urr, dotted]
    % \end{tikzcd} \]
    % \cref{thm:cube_obox_map} shows that $\overline{g}\f \from \gexp{I_{3m}}{n} \to G$ is such a lift.
\end{proof}
% For $n \geq 0$ and a graph $G$ with distinguished vertex $v \in G$.
% Suppose we have the following commutative square.
% \[ \begin{tikzcd}
    
% \end{tikzcd} \]
Via \cref{thm:nerve_kan}, we may speak of the homotopy groups of $\gnerve G$.
We prove that the A-homotopy groups of $G$ are exactly the cubical homotopy groups of $\gnerve G$. 
\begin{theorem}\label{thm:a_eq_cset_pi}
    We have an isomorphism $A_n(G, v) \cong \pi_n(\gnerve{G}, v)$ natural in $(G, v)$.
\end{theorem} 
\begin{proof}
    It is straightforward to verify that a relative graph map $(\gexp{I_\infty}{n}, \gexp{I_{\geq M}}{n}) \to (G, v)$ is exactly a relative cubical map $(\cube{n}, \bd \cube{n}) \to (\gnerve G, v)$ and that a relative homotopy between two such relative graph maps is exactly a relative homotopy between two such relative cubical maps.
    This gives a set bijection $A_n(G, v) \to \pi_n(\gnerve{G}, v)$.
    \cref{thm:nerve_conc} shows this map is a group homomorphism.
\end{proof}

\begin{remark}
  \cref{thm:a_eq_cset_pi} shows why the functor $\gnerve \from \Graph \to \cSet$ fails to preserve arbitrary limits, even though, by \cref{thm:nerve_fin_lim}, it preserves finite ones.
  To see that, we first note that 
  \[A_1\left(\prod\limits_{k \geq 5} (C_k, 0)\right) \cong \bigoplus\limits_{k \geq 5} \mathbb{Z}\text{,}\]
   since a map $I_\infty \to \prod_{k \geq 5} C_k$ that stabilizes outside of a finite interval must necessarily be null-homotopic in all but finitely many $C_k$'s.
   That is, $A_1 \from \Graph_* \to \mathsf{Grp}$ does not preserve infinite products.
  \cref{thm:a_eq_cset_pi} shows that $A_1 \cong \pi_1 \circ \gnerve$, and $\pi_1 \colon \cSet_* \to \mathsf{Grp}$ preserves infinite products by \cite[Prop.~4.1]{carranza-kapulkin:homotopy-cubical}.
  Thus, we conclude that
  \[  \gnerve\left(\prod\limits_{k \geq 5} (C_k, 0)\right) \not \cong \prod\limits_{k \geq 5} (\gnerve C_k, 0).  \]
%   i.e.~$\gnerve$ does not preserve infinite products.
\end{remark}

\subsection*{Proof of part (2)}
Now we show that the inclusions $l^*, r^* \from \gnerve[m] G \ito \gnerve[m+1] G$ are anodyne.
We first explain the intuition for why this statement holds.

The 1-nerve $\gnerve[1] G$ of a graph $G$ contains, as 1-cubes, all paths of length 1 in $G$ (i.e.~paths with 2 vertices and 1 edge).
Consider the image of the embedding $l^* \from \gnerve[1] G \to \gnerve[2] G$ as a cubical subset $X \subseteq \gnerve[2] G$ of the 2-nerve of $G$.
A 1-cube of $\gnerve[2] G$ is exactly a path of length 2 in $G$; a 1-cube of $X$ is a path of length 1 regarded as a path of length 2 whose first two vertices are the same.
Given a path $f \from I_2 \to G$ of length 2 in $G$, we may define a $3 \times 3$ square $g \from \gexp{I_2}{2} \to G$ in $G$ by
\[ g(v_1, v_2) := \begin{cases}
    f(l(v_1)) & v_2 < 2 \\
    f(v_1) & v_2 = 2.
\end{cases} \]
\begin{figure}[H]
    \centering
    \begin{tikzpicture}[node distance = 15pt]
        \node[vertex, fill=red] (0) {};
        \node[vertex, fill=yellow] (1) [right=of 0] {};
        \node[vertex, fill=cyan] (2) [right=of 1] {};

        \node (arrow) [right=of 2] {$\leadsto$};
        
        \node[vertex, fill=red] (01) [right=of arrow] {};
        \node[vertex, fill=red] (11) [right=of 01] {};
        \node[vertex, fill=yellow] (21) [right=of 11] {};

        \node[vertex, fill=red] (00) [above=of 01] {};
        \node[vertex, fill=red] (10) [right=of 00] {};
        \node[vertex, fill=yellow] (20) [right=of 10] {};
        
        \node[vertex, fill=red] (02) [below=of 01] {};
        \node[vertex, fill=yellow] (12) [right=of 02] {};
        \node[vertex, fill=cyan] (22) [right=of 12] {};

        \draw (0) to (1) to (2);

        \draw (00) to (01) to (02);
        \draw (10) to (11) to (12);
        \draw (20) to (21) to (22);
        \draw (00) to (10) to (20);
        \draw (01) to (11) to (21);
        \draw (02) to (12) to (22);
    \end{tikzpicture} 
    \caption{The square $g \from \gexp{I_2}{2} \to G$ constructed from the path $f \from I_2 \to G$.}
\end{figure}
Observe the $\face{}{1,0}$-, $\face{}{1,1}$-, and $\face{}{2,0}$-faces of $g$ are 1-cubes of $X$, i.e.~they are paths of length 2 whose first two vertices are the same, whereas the $\face{}{2,1}$-face of $g$ is $f$. 
That is, we have shown the restriction $\restr{g}{\reali[2]{\obox{2}{2,1}}} \from \reali[2]{\obox{2}{2,1}} \to G$ of $g$ to the open box corresponds to a map $\obox{2}{2,1} \to \gnerve[2] G$ whose image is contained in $X$.
Let $Y \subseteq \gnerve[2] G$ denote the cubical subset generated by $X$ and $g$; that is, the cubical subset containing $X$, the 2-cube $g \in (\gnerve[2] G)_2$, and all faces, degeneracies and connections of $g$.
The square
\[ \begin{tikzcd}
    \obox{2}{2,1} \ar[r, "{\restr{g}{\reali[2]{\obox{2}{2,1}}}}"] \ar[d, hook] \ar[rd, phantom, "\ulcorner" very near end] & X \ar[d, hook] \\
    \cube{2} \ar[r, "g"] & Y
\end{tikzcd} \]
is a pushout by definition of $Y$.
This gives exactly that the inclusion $X \ito Y$ is anodyne.
Following this approach, one may construct an anodyne inclusion $X \ito X_n$ such that $X_n$ contains all $n$-cubes of $\gnerve[2] G$.
With this, the colimit of the sequence $X \ito X_1 \ito X_2 \ito \dots$ is exactly $\gnerve[2] G$ and the inclusion $X \ito \gnerve[2] G$ is anodyne by closure under transfinite composition.

% Postcomposing these maps with the embedding $I_{m+1} \to I_m$ gives maps $l', r' \from I_{m+1} \to I_{m+1}$ defined by
% \[ l'(i) = \min(i, m), \quad r'(i) = \max(i-1, 0). \]
For $n > 0$ and $j \in \{ 0, \dots, n \}$, the maps $l, r \from I_{m+1} \to I_m$ yield maps 
\[ (\gexp{\id[I_{m+1}]}{j} \gtimes \gexp{l}{n-j}), (\gexp{\id[I_{m+1}]}{j} \gtimes \gexp{r}{n-j}) \from \gexp{I_{m+1}}{n} \to \gexp{I_{m+1}}{j} \gtimes \gexp{I_m}{n-j}. \]
For $j \in \{ 0, \dots, n-1 \}$, we define maps $\lstep{m,j}, \rstep{m,j} \from \gexp{I_{m+1}}{n} \gtimes I_1 \to \gexp{I_{m+1}}{j+1} \gtimes \gexp{I_m}{n-j-1}$ by
\begin{align*}
    \lstep{m,j}(v_1, \dots, v_n, v_{n+1}) &= \begin{cases}
        (v_1, \dots, v_j, lv_{j+1}, lv_{j+2}, \dots, lv_n) & v_{n+1} = 0 \\
        (v_1, \dots, v_{j+1}, lv_{j+2}, \dots, lv_n) & v_{n+1} = 1.
    \end{cases} \\
    \rstep{m,j}(v_1, \dots, v_n, v_{n+1}) &= \begin{cases}
        (v_1, \dots, v_j, rv_{j+1}, rv_{j+2}, \dots, rv_n) & v_{n+1} = 0 \\
        (v_1, \dots, v_{j+1}, rv_{j+2}, \dots, rv_n) & v_{n+1} = 1.
    \end{cases}
\end{align*}
That is, the restriction of $\lstep{m,j}$ to $\gexp{I_{m+1}}{n} \gtimes \{ 0 \}$ is the map $\gexp{\id[I_{m+1}]}{j} \gtimes \gexp{l}{n-j}$ and its restriction to $\gexp{I_{m+1}}{n} \gtimes \{ 1 \}$ is $\gexp{\id[I_{m+1}]}{j+1} \gtimes \gexp{l}{n-j-1}$ (likewise for $\rstep{m,j}$).
\begin{figure}[H]
    \centering
    \begin{tikzpicture}[node distance = 20pt]
        \node[vertex, minimum size=12pt] (00) {0};
        \node[vertex, minimum size=12pt] (10) [right=of 00] {0};
        \node[vertex, minimum size=12pt] (20) [right=of 10] {1};
        \node[vertex, minimum size=12pt] (01) [below=of 00] {0};
        \node[vertex, minimum size=12pt] (11) [right=of 01] {1};
        \node[vertex, minimum size=12pt] (21) [right=of 11] {2};

        \draw (00) to (10) to (20);
        \draw (01) to (11) to (21);
        \draw (00) to (01);
        \draw (10) to (11);
        \draw (20) to (21);
    \end{tikzpicture}
    \caption{The graph $I_2 \gtimes I_1$ with vertices labelled by their image under $\lstep[1]{1,0} \from I_2 \gtimes I_1 \to I_2$.}
\end{figure}
\begin{figure}[H]
    \centering
    \begin{subfigure}[h]{0.45\textwidth}
        \centering
        \begin{tikzpicture}[node distance=15pt]
            \node[vertex, minimum size=15pt] (00) {(0,0)};
            \node[vertex, minimum size=15pt] (10) [right=of 00] {(0,0)};
            \node[vertex, minimum size=15pt] (20) [right=of 10] {(1,0)};
            \node[vertex, minimum size=15pt] (01) [below=of 00] {(0,0)};
            \node[vertex, minimum size=15pt] (11) [right=of 01] {(0,0)};
            \node[vertex, minimum size=15pt] (21) [right=of 11] {(1,0)};
            \node[vertex, minimum size=15pt] (02) [below=of 01] {(0,1)};
            \node[vertex, minimum size=15pt] (12) [right=of 02] {(0,1)};
            \node[vertex, minimum size=15pt] (22) [right=of 12] {(1,1)};

            \draw (00) to (10) to (20);
            \draw (01) to (11) to (21);
            \draw (02) to (12) to (22);
            \draw (00) to (01) to (02);
            \draw (10) to (11) to (12);
            \draw (20) to (21) to (22);
        \end{tikzpicture}    
        \caption{The subgraph $\gexp{I_2}{2} \gprod \{0\}$ under $\lstep[2]{1,0}$}
    \end{subfigure}
    \begin{subfigure}[h]{0.45\textwidth}
        \centering
        \begin{tikzpicture}[node distance=15pt]
            \node[vertex, minimum size=15pt] (00) {(0,0)};
            \node[vertex, minimum size=15pt] (10) [right=of 00] {(1,0)};
            \node[vertex, minimum size=15pt] (20) [right=of 10] {(2,0)};
            \node[vertex, minimum size=15pt] (01) [below=of 00] {(0,0)};
            \node[vertex, minimum size=15pt] (11) [right=of 01] {(1,0)};
            \node[vertex, minimum size=15pt] (21) [right=of 11] {(2,0)};
            \node[vertex, minimum size=15pt] (02) [below=of 01] {(0,1)};
            \node[vertex, minimum size=15pt] (12) [right=of 02] {(1,1)};
            \node[vertex, minimum size=15pt] (22) [right=of 12] {(2,1)};

            \draw (00) to (10) to (20);
            \draw (01) to (11) to (21);
            \draw (02) to (12) to (22);
            \draw (00) to (01) to (02);
            \draw (10) to (11) to (12);
            \draw (20) to (21) to (22);
        \end{tikzpicture}
        \caption{The subgraph $\gexp{I_2}{2} \gprod \{1\}$ under $\lstep[2]{1,1}$}
    \end{subfigure}
    \begin{subfigure}[h]{0.45\textwidth}
        \centering
        \begin{tikzpicture}[node distance=15pt]
            \node[vertex, minimum size=15pt] (00) {(0,0)};
            \node[vertex, minimum size=15pt] (10) [right=of 00] {(1,0)};
            \node[vertex, minimum size=15pt] (20) [right=of 10] {(2,0)};
            \node[vertex, minimum size=15pt] (01) [below=of 00] {(0,0)};
            \node[vertex, minimum size=15pt] (11) [right=of 01] {(1,0)};
            \node[vertex, minimum size=15pt] (21) [right=of 11] {(2,0)};
            \node[vertex, minimum size=15pt] (02) [below=of 01] {(0,1)};
            \node[vertex, minimum size=15pt] (12) [right=of 02] {(1,1)};
            \node[vertex, minimum size=15pt] (22) [right=of 12] {(2,1)};

            \draw (00) to (10) to (20);
            \draw (01) to (11) to (21);
            \draw (02) to (12) to (22);
            \draw (00) to (01) to (02);
            \draw (10) to (11) to (12);
            \draw (20) to (21) to (22);
        \end{tikzpicture}    
        \caption{The subgraph $\gexp{I_2}{2} \gprod \{0\}$ under $\lstep[2]{1,1}$}
    \end{subfigure}
    \begin{subfigure}[h]{0.45\textwidth}
        \centering
        \begin{tikzpicture}[node distance=15pt]
            \node[vertex, minimum size=15pt] (00) {(0,0)};
            \node[vertex, minimum size=15pt] (10) [right=of 00] {(1,0)};
            \node[vertex, minimum size=15pt] (20) [right=of 10] {(2,0)};
            \node[vertex, minimum size=15pt] (01) [below=of 00] {(0,1)};
            \node[vertex, minimum size=15pt] (11) [right=of 01] {(1,1)};
            \node[vertex, minimum size=15pt] (21) [right=of 11] {(2,1)};
            \node[vertex, minimum size=15pt] (02) [below=of 01] {(0,2)};
            \node[vertex, minimum size=15pt] (12) [right=of 02] {(1,2)};
            \node[vertex, minimum size=15pt] (22) [right=of 12] {(2,2)};

            \draw (00) to (10) to (20);
            \draw (01) to (11) to (21);
            \draw (02) to (12) to (22);
            \draw (00) to (01) to (02);
            \draw (10) to (11) to (12);
            \draw (20) to (21) to (22);
        \end{tikzpicture}
        \caption{The subgraph $\gexp{I_2}{2} \gprod \{1\}$ under $\lstep[2]{1,1}$}
    \end{subfigure}
    \caption{Cross-sections of the graph $\gexp{I_2}{2} \gprod I_1$ with vertices labelled by their image under the maps $\lstep[2]{1,0} \from \gexp{I_2}{2} \gtimes I_1 \to I_2 \gtimes I_1$ and $\lstep[2]{1,1} \from \gexp{I_2}{2} \gtimes I_1 \to \gexp{I_2}{2}$.}
\end{figure}
% \begin{itemize}
%     \item the restrictions $\restr{\lstep{m,j}}{\gexp{I_{m+1}}{n} \gtimes \{ 0 \}}, \restr{\rstep{m,j}}{\gexp{I_{m+1}}{n} \gtimes \{ 0 \}} \from \gexp{I_{m+1}}{n} \to \gexp{I_{m+1}}{j+1} \gtimes \gexp{I_m}{n-j-1}$ are the maps $\gexp{\id[I_{m+1}]}{j} \gtimes \gexp{l'}{n-j}$  
    
%     and $\gexp{\id[I_{m+1}]}{j} \gtimes \gexp{r'}{n-j}$, respectively;
%     \item the restrictions $\restr{\lstep{m,j}}{\gexp{I_{m+1}}{n} \gtimes \{ 1 \}}, \restr{\rstep{m,j}}{\gexp{I_{m+1}}{n} \gtimes \{ 1 \}} \from \gexp{I_{m+1}}{n} \to \gexp{I_{m+1}}{j+1} \gtimes \gexp{I_m}{n-j-1}$ are the maps $\gexp{\id[I_{m+1}]}{j+1} \gtimes \gexp{l'}{n-j-1}$ 
    
%     and $\gexp{\id[I_{m+1}]}{j+1} \gtimes \gexp{r'}{n-j-1}$, respectively.
% \end{itemize}
The maps $l^m, r^m \from I_{m+1} \to I_1$ denote application of the maps $l, r$ a total of $m$ times.
That is,
\begin{align*}
    l^m(v) &= \begin{cases}
        0 & v < m+1 \\
        1 & v = m+1,
    \end{cases} \\
    r^m(v) &= \begin{cases}
        0 & v = 0 \\
        1 & v > 0.
    \end{cases}
\end{align*}
We write $\lstepbar{m,j} \from \gexp{I_{m+1}}{n+1} \to \gexp{I_{m+1}}{j+1} \gtimes \gexp{I_m}{n-j-1}$ for the composition of $\gexp{\id[I_{m+1}]}{n} \gtimes l^m \from \gexp{I_{m+1}}{n+1} \to \gexp{I_{m+1}}{n} \gtimes I_1$ followed by $\lstep{m,j} \from \gexp{I_{m+1}}{n} \gtimes I_1 \to \gexp{I_{m+1}}{j+1} \gtimes \gexp{I_m}{n-j-1}$ and we write $\rstepbar{m,j} \from \gexp{I_{m+1}}{n+1} \to \gexp{I_{m+1}}{j+1} \gtimes \gexp{I_m}{n-j-1}$ for the composition of $\gexp{\id[I_{m+1}]}{n} \gtimes r^n \from \gexp{I_{m+1}}{n+1} \to \gexp{I_{m+1}}{n} \gtimes I_1$ followed $\rstep{m,j} \from \gexp{I_{m+1}}{n} \gtimes I_1 \to \gexp{I_{m+1}}{j+1} \gtimes \gexp{I_m}{n-j-1}$, respectively.

% We may easily desribe most of the faces of $\lstepbar{j,m}$ and $\rstepbar{j,m}$.
% We do so in the following proposition.
\begin{proposition} \label{thm:lstep_face_inner}
    Let $m, n > 0$ and $j \in \{ 0, \dots, n \}$.
    For $i = 1, \dots, n$ such that $i \neq j+1$ and $\varepsilon = 0, 1$,
    \begin{enumerate}
        \item $\lstepbar{m,j}\face{}{i,\varepsilon} \from \gexp{I_{m+1}}{n} \to \gexp{I_{m+1}}{n}$ factors through $\lstepbar[n-1]{m,j} \from \gexp{I_{m+1}}{n} \to \gexp{I_{m+1}}{n-1}$;
        \item $\rstepbar{m,j}\face{}{i,\varepsilon} \from \gexp{I_{m+1}}{n} \to \gexp{I_{m+1}}{n}$ factors through $\rstepbar[n-1]{m,j} \from \gexp{I_{m+1}}{n} \to \gexp{I_{m+1}}{n-1}$
    \end{enumerate}
\end{proposition}
\begin{proof}
    We consider the result for $\lstepbar{m,j}$, as the result for $\rstepbar{m,j}$ is analogous.

    Fix $(v_1, \dots, v_n) \in \gexp{I_{m+1}}{n}$.
    If $i < j + 1$ then we have
    \begin{align*}
        \lstepbar{m,j}\face{}{i,\varepsilon}(v_1, \dots, v_n) &= \lstepbar{m,j}(v_1, \dots, v_{i-1}, \varepsilon (m+1), v_{i}, \dots, v_n) \\
        &= \lstep{m,j} (v_1, \dots, v_{i-1}, \varepsilon (m+1), v_i, \dots, v_{n-1}, l^m v_n) \\
        &= \begin{cases}
            (v_1, \dots, v_{i-1}, \varepsilon (m+1), v_i, \dots, v_{j}, l v_{j+1}, \dots, l v_{n-1}) & \text{if } l^m v_n = 0 \\
            (v_1, \dots, v_{i-1}, \varepsilon (m+1), v_i, \dots, v_{j+1}, l v_{j+2}, \dots, l v_{n-1}) & \text{if } l^m v_n = 1
        \end{cases} \\
        &= \face{n}{i,\varepsilon} \lstep{m,j} (v_1, \dots, v_{n-1}, l^m v_n) \\
        &= \face{n}{i,\varepsilon} \lstepbar{m,j} (v_1, \dots, v_n).
    \end{align*}
    Thus, $\lstepbar{m,j}\face{}{i,\varepsilon} = \face{}{i,\varepsilon}\lstepbar[n-1]{m,j}$.

    Otherwise, we have $i > j+1$.
    Consider the embedding $\iota \from \gexp{I_{m+1}}{n-1} \to \gexp{I_{m+1}}{n}$ given by
    \[ \iota(v_1, \dots, v_{n-1}) = (v_1, \dots, v_{i-1}, \varepsilon m, v_i, \dots, v_{n-1}). \]
    With this, we may write
    \begin{align*}
        \lstepbar{m,j}\face{n+1}{i,0}(v_1, \dots, v_n) &= \lstepbar{m,j}(v_1, \dots, v_{i-1}, \varepsilon (m+1), v_{i}, \dots, v_n) \\
        &= \lstep{m,j} (v_1, \dots, v_{i-1}, \varepsilon (m+1), v_i, \dots, v_{n-1}, l^m v_n) \\
        &= \begin{cases}
            (v_1, \dots, v_j, l v_{j+1}, \dots, l v_{i-1}, l (\varepsilon (m+1)), l v_i, \dots, l v_{n-1}) & \text{if } l^m v_n = 0 \\
            (v_1, \dots, v_{j+1}, l v_{j+2}, \dots, l v_{i-1}, l (\varepsilon (m+1)), l v_i, \dots, l v_{n-1}) & \text{if } l^m v_n = 1
        \end{cases} \\
        &= \begin{cases}
            (v_1, \dots, v_j, l v_{j+1}, \dots, l v_{i-1}, \varepsilon m, l v_i, \dots, l v_{n-1}) & \text{if } l^m v_n = 0 \\
            (v_1, \dots, v_{j+1}, l v_{j+2}, \dots, l v_{i-1}, \varepsilon m, l v_i, \dots, l v_{n-1}) & \text{if } l^m v_n = 1
        \end{cases} \\
        &= \iota \lstep{m,j} (v_1, \dots, v_{n-1}, l^m v_n) \\
        &= \iota \lstepbar{m,j} (v_1, \dots, v_n).
    \end{align*}
    Thus, $\lstepbar{m,j}\face{}{i,\varepsilon} = \iota \lstepbar[n-1]{m,j}$.
\end{proof}
\begin{proposition} \label{thm:lstep_face_outer}
    Let $m, n > 0$ and $j \in \{ 0, \dots, n \}$.
    Then,
    \begin{enumerate}
        \item $\lstepbar{m,j}\face{}{j+1,0}$ and $\lstepbar{m,j}\face{}{j+1,1}$ factor through $\gexp{\id[I_{m+1}]}{j} \gtimes \gexp{l}{n-j} \from \gexp{I_{m+1}}{n} \to \gexp{I_{m+1}}{j} \gtimes \gexp{I_m}{n-j}$;
        \item $\rstepbar{m,j}\face{}{j+1,0}$ and $\rstepbar{m,j}\face{}{j+1,1}$ factor through $\gexp{\id[I_{m+1}]}{j} \gtimes \gexp{r}{n-j} \from \gexp{I_{m+1}}{n} \to \gexp{I_{m+1}}{j} \gtimes \gexp{I_m}{n-j}$.
    \end{enumerate}
\end{proposition}
\begin{proof}
    We show the result for $\lstepbar{m,j}$ as the result for $\rstepbar{m,j}$ is analogous.

    Fix $(v_1, \dots, v_n) \in \gexp{I_{m+1}}{n}$.
    We compute
    \begin{align*}
        \lstepbar{m,j}\face{}{j+1,0}(v_1, \dots, v_n) &= \lstepbar{m,j} (v_1, \dots, v_{j}, 0, v_{j+1}, \dots, v_n) \\
        &= \lstep{m,j}(v_1, \dots, v_{j}, 0, v_{j+1}, \dots, l^m v_n) \\
        &= \begin{cases}
            (v_1, \dots, v_{j}, l 0, l v_{j+1}, \dots, l v_{n-1}) & \text{if } l^m v_n = 0 \\
            (v_1, \dots, v_j, 0, l v_{j+1}, \dots, l v_{n-1}) & \text{if } l^m v_n = 1
        \end{cases} \\
        &= (v_1, \dots, v_j, 0, l v_{j+1}, \dots, l v_{n-1}) \\
        &= \face{}{j+1,0}\degen{}{n}(\gexp{\id[I_{m+1}]}{j} \gtimes \gexp{l}{n-j})(v_1, \dots, v_n),
    \end{align*}
    thus $\lstepbar{m,j}\face{}{j+1,0} = \face{}{j+1,0}\degen{}{n}(\gexp{\id[I_{m+1}]}{j} \gtimes \gexp{l}{n-j})$.

    Consider the map $f \from \gexp{I_{m+1}}{j} \gtimes \gexp{I_m}{n-j} \to \gexp{I_{m+1}}{j+1} \gtimes \gexp{I_m}{n-j-1}$ defined by
    \[ f(a_1, \dots, a_j, b_1, \dots, b_{n-j}) = \begin{cases}
        (a_1, \dots, a_j, m, b_1, \dots, b_{n-j-1}) & \text{if } l^{m-1} b_{n-j} = 0 \\
        (a_1, \dots, a_j, m+1, b_1, \dots, b_{n-j-1}) & \text{if } l^{m-1} b_{n-j} = 1. 
    \end{cases} \]
    It is straightforward to verify this is a graph map.
    With this, we may write
    \begin{align*}
        \lstepbar{m,j} \face{}{j+1,1} (v_1, \dots, v_n) &= \lstepbar{m,j} (v_1, \dots, v_j, m+1, v_{j+1}, \dots, v_n) \\
        &= \lstep{m,j} (v_1, \dots, v_j, m+1, v_{j+1}, \dots, l^m v_n) \\
        &= \begin{cases}
            (v_1, \dots, v_j, l (m+1), l v_{j+1}, \dots, l v_{n-1}) & \text{if } l^m v_n = 0 \\
            (v_1, \dots, v_j, m+1, l v_{j+1}, \dots, l v_{n-1}) & \text{if } l^m v_n = 1
        \end{cases} \\
        &= \begin{cases}
            (v_1, \dots, v_j, m, l v_{j+1}, \dots, l v_{n-1}) & \text{if } l^m v_n = 0 \\
            (v_1, \dots, v_j, m+1, l v_{j+1}, \dots, l v_{n-1}) & \text{if } l^m v_n = 1
        \end{cases} \\
        &= f(v_1, \dots, v_j, l v_{j+1}, \dots, l v_{n}) \\
        &= f(\gexp{\id[I_{m+1}]}{j} \gtimes \gexp{l}{n-j})(v_1, \dots, v_n).
    \end{align*}
    Thus, $\lstepbar{m,j}\face{}{j+1,1} = f(\gexp{\id[I_{m+1}]}{j} \gtimes \gexp{l}{n-j})$.
\end{proof}
\begin{lemma} \label{thm:lanodyne_factor_lemma}
    Let $m, n > 0$, $j \in \{ 0, \dots, n-1 \}$, and $G$ be a graph.
    Consider a subobject $X$ of $\gnerve[m+1] G$ which contains
    \begin{itemize}
        \item all $n$-cubes of $\gnerve[m+1] G$ which factor through $\gexp{l}{n} \from \gexp{I_{m+1}}{n} \to \gexp{I_m}{n}$,
        \item (if $n > 1$) for any $x \from \gexp{I_{m+1}}{h} \gtimes \gexp{I_m}{k-h} \to G$ where $k < n$ and $h \leq k$, the $(k+1)$-cube $x\lstepbar[k]{m,h-1} \from \gexp{I_{m+1}}{k+1} \to G$,
        \item (if $j > 0$) for any $x \from \gexp{I_{m+1}}{j} \gtimes \gexp{I_m}{n-j} \to G$, the $(n+1)$-cube $x\lstepbar{m,j-1} \from \gexp{I_{m+1}}{n+1} \to G$ of $\gnerve[m+1] G$.
    \end{itemize}
    Then, for any $n$-cube of $\gnerve[m+1] G$ which factors through $\gexp{\id[I_{m+1}]}{j+1} \gtimes \gexp{l}{n-j-1} \from \gexp{I_{m+1}}{n} \to \gexp{I_{m+1}}{j+1} \gtimes \gexp{I_m}{n-j-1}$ as some $x \from \gexp{I_{m+1}}{j+1} \gtimes \gexp{I_m}{n-j-1} \to G$, the restriction of the $(n+1)$-cube $x\lstepbar{m,j} \from \gexp{I_{m+1}}{n+1} \to G$ to the open box $\restr{x\lstepbar{m,j}}{\obox{n+1}{n+1,1}} \from \obox{n+1}{n+1,1} \to \gnerve[m+1] G$ factors through the inclusion $X \ito \gnerve[m+1] G$.
\end{lemma}
\begin{proof}
    We first show that $X$ contains all $n$-cubes which factor through the map $\gexp{\id[I_{m+1}]}{j} \gtimes \gexp{l}{n-j} \from \gexp{I_{m+1}}{n} \to \gexp{I_{m+1}}{j} \gtimes \gexp{I_m}{n-j}$. 
    If $j = 0$ then this follows by assumption.
    Otherwise, consider such an $n$-cube, which we write as $y(\gexp{\id[I_{m+1}]}{j} \gtimes \gexp{l}{n-j}) \from \gexp{I_{m+1}}{n} \to G$ for some $y \from \gexp{I_{m+1}}{j} \gtimes \gexp{I_m}{n-j} \to G$.
    Recall $\lstepbar{m,j-1}\face{}{n+1,1} = \gexp{\id[I_{m+1}]}{j} \gtimes \gexp{l}{n-j}$.
    This gives that $y(\gexp{\id[I_{m+1}]}{j} \gtimes \gexp{l}{n-j}) = y\lstepbar{m,j-1}\face{}{n+1,1}$.
    By assumption, $X$ contains $y\lstepbar{m,j-1}$. 
    Thus, it contains all faces of $y\lstepbar{m,j-1}$, including $y$.

    To see that each face of $\restr{x\lstepbar{m,j}}{\obox{n+1}{n+1,1}}$ is contained in $X$, fix $(i, \varepsilon) \neq (n+1,1)$.
    For $i = n+1$ and $\varepsilon = 0$, this follows as $\lstepbar{m,j}\face{}{n+1,0} = \gexp{\id[I_{m+1}]}{j} \gtimes \gexp{l}{n-j}$.
    Otherwise, if $i \neq j+1$, this follows from \cref{thm:lstep_face_inner}.
    If $i = j+1$ then \cref{thm:lstep_face_outer} gives that $\lstepbar{m,j}$ factors through $\gexp{\id[I_{m+1}]}{j} \gtimes \gexp{l}{n-j}$.
\end{proof}
We have an analogous result for $\rstepbar{m,j}$ as well.
\begin{lemma} \label{thm:ranodyne_factor_lemma}
    Let $m, n > 0$, $j \in \{ 0, \dots, n-1 \}$, and $G$ be a graph.
    Consider a subobject $X$ of $\gnerve[m+1] G$ which contains
    \begin{itemize}
        \item all $n$-cubes of $\gnerve[m+1] G$ which factor through $\gexp{r}{n} \from \gexp{I_{m+1}}{n} \to \gexp{I_m}{n}$,
        \item (if $n > 1$) for any $x \from \gexp{I_{m+1}}{h} \gtimes \gexp{I_m}{k-h} \to G$ where $k < n$ and $h \leq k$, the $(k+1)$-cube $x\rstepbar{m,h-1} \from \gexp{I_{m+1}}{n+1} \to G$,
        \item (if $j > 0$) for any $x \from \gexp{I_{m+1}}{j} \gtimes \gexp{I_m}{n-j} \to G$, the $(n+1)$-cube $x\rstepbar{m,j-1} \from \gexp{I_{m+1}}{n+1} \to G$ of $\gnerve[m+1] G$.
    \end{itemize}
    Then, for any $n$-cube of $\gnerve[m+1] G$ which factors through $\gexp{\id[I_{m+1}]}{j+1} \gtimes \gexp{r}{n-j-1} \from \gexp{I_{m+1}}{n} \to \gexp{I_{m+1}}{j+1} \gtimes \gexp{I_m}{n-j-1}$ as some $x \from \gexp{I_{m+1}}{j+1} \gtimes \gexp{I_m}{n-j-1} \to G$, the restriction of the $(n+1)$-cube $x\rstepbar{m,j} \from \gexp{I_{m+1}}{n+1} \to G$ to the open box $\restr{x\rstepbar{m,j}}{\obox{n+1}{n+1,1}} \from \obox{n+1}{n+1,1} \to \gnerve[m+1] G$ factors through the inclusion $X \ito \gnerve[m+1] G$. \noproof
\end{lemma}
\begin{theorem} \label{thm:lr_anodyne}
    Let $m > 0$ and $G$ be a graph.
    The maps $l^*, r^* \from \gnerve[m]{G} \to \gnerve[m+1]{G}$ are anodyne.
\end{theorem}
\begin{proof}
    We show that $l^*$ is anodyne as the result for $r^*$ is analogous.
    Let $X_0$ denote the image of $\gnerve[m]{G}$ in $\gnerve[m+1]{G}$ under the embedding $l^* \from \gnerve[m]{G} \ito \gnerve[m+1]{G}$.
    We show $\gnerve[m+1]{G}$ can be obtained from $X_0$ by a transfinite composition of pushouts along coproducts of open box inclusion.
    For $n > 0$ and $j \in \{ 1, \dots, n \}$, let $X_{n,j}$ be the subobject of $\gnerve[m+1]{G}$ generated by:
    \begin{itemize}
        \item $X_0$,
        \item (if $n > 1$) for any $x \from \gexp{I_{m+1}}{l} \gtimes \gexp{I_m}{k-l} \to G$ where $k < n$ and $l \leq k$, the $(k+1)$-cube $x\lstepbar[k]{m,l-1} \from \gexp{I_{m+1}}{n+1} \to G$ of $\gnerve[m+1] G$,
        \item for any $x \from \gexp{I_{m+1}}{i} \gtimes \gexp{I_m}{n-i} \to G$ where $i \leq j$, the $(n+1)$-cube $x\lstepbar{m,i-1} \from \gexp{I_{m+1}}{n+1} \to G$ of $\gnerve[m+1] G$.
        % \item all $n$-cubes of $\gnerve[m+1] G$ which factor through $\lstepbar[n-1]{m,k} \from \gexp{I_{m+1}}{n} \to \gexp{I_{m+1}}{k} \gtimes \gexp{I_m}{n-k}$ for any $k \in \{ 0, \dots, n-2 \}$.
    \end{itemize}
    By construction, there is a sequence of embeddings
    \[ X_0 \ito X_{1,1} \ito X_{2,1} \ito X_{2,2} \ito X_{3,1} \ito \dots \]
    % An map $x \from \gexp{I_{m+1}}{n} \to G$ factors through $\gexp{l}{n-j} \from \gexp{I_{m+1}}{n} \to \gexp{I_{m+1}}{j} \gtimes \gexp{I_m}{n-j}$ for some $0 \leq j \leq n$.
    % For $j = 0$, we have $x \in X_0$ by definition.
    % For $j \geq 1$, $x$ is the $\face{}{n+1,0}$-face of $\overline{x}_l$ by construction so $x \in (X_{n,j})$. 
    % Thus, every $n$-cube of $\sing[m+1]{G}$ is in $X_{n,j}$ for some $0 \leq j \leq n$.
    % For $n > 0$ and $j \in \{ 1, \dots, n \}$, the subobject $X_{n,j}$ contains all $n$-cubes which factor through the map $(\gexp{\id[I_{m+1}]}{j} \gtimes \gexp{l}{n-j}) \from \gexp{I_{m+1}}{n} \to \gexp{I_{m+1}}{j} \gtimes \gexp{I_m}{n-j}$. 
    % This is because such an $n$-cube $x \from \gexp{I_{m+1}}{n} \to G$ is, by construction, the $\face{}{n+1,1}$-face of $x\lstepbar{m,j-1} \from \gexp{I_{m+1}}{n+1} \to G$.
    % By construction, $X_{n,j}$ contains $x\lstepbar{m,j-1}$. Thus, it contains all faces of $x\lstepbar{m,j-1}$, including $x$.
    Note that the subobject $X_{n,n}$ contains all $n$-cubes of $\gnerve[m+1] G$.
    This is because any $n$-cube $x \from \gexp{I_{m+1}}{n} \to G$ is the $\face{}{n+1,1}$-face of $x\lstepbar{m,n-1} \from \gexp{I_{m+1}}{n+1} \to G$.
    By construction, $X_{n,n}$ contains $x\lstepbar{m,n-1}$. Thus, it contains all faces of $x\lstepbar{m,n-1}$, including $x$.
    With this, we have 
    \[ \gnerve[m+1]{G} \cong \colim(X_0 \ito X_{1,1} \ito X_{2,1} \ito X_{2,2} \ito \dots). \]
    It remains to show $X_{n,j+1}$ is a pushout of $X_{n,j}$ along a coproduct of open box inclusions and $X_{n+1,1}$ is a pushout of $X_{n,n}$ along a coproduct of open box inclusions.

    Fix $n > 0$ and $j \in \{ 1, \dots, n-1 \}$.
    Let $S_{n,j+1}$ be the set of $n$-cubes $\gexp{I_{m+1}}{n} \to G$ which factor through the map $\gexp{\id[I_{m+1}]}{j+1} \gtimes \gexp{l}{n-j-1} \from \gexp{I_{m+1}}{n} \to \gexp{I_{m+1}}{j+1} \gtimes \gexp{I_m}{n-j-1}$ and are not contained in $X_{n,j}$.
    We write an element of $S_{n,j+1}$ as $x(\gexp{\id[I_{m+1}]}{j+1} \gtimes \gexp{l}{n-j-1})$ for some $x \from \gexp{I_{m+1}}{j+1} \gtimes \gexp{I_m}{n-j-1} \to G$.
    By construction, $X_{n,j+1}$ contains all $n$-cubes of $S_{n,j+1}$ as such an $n$-cube $x(\gexp{\id[I_{m+1}]}{j+1} \gtimes \gexp{l}{n-j-1})$ is the $\face{}{n+1,1}$-face of $x\lstepbar{m,j}$ (which is contained in $X_{n,j+1}$ by construction).
    This gives a map 
    \[ \coprod\limits_{x(\gexp{\id[I_{m+1}]}{j+1} \gtimes \gexp{l}{n-j-1}) \in S_{n,j+1}} \cube{n+1} \xrightarrow{x\lstepbar{m,j}} X_{n,j+1}. \]
    For each $x(\gexp{\id[I_{m+1}]}{j+1} \gtimes \gexp{l}{n-j-1}) \in S_{n,j+1}$, the restriction of the map $x\lstepbar{m,j}$ to the open box $\obox{n+1}{n+1,1}$ factors through $X_{n,j}$ by \cref{thm:lanodyne_factor_lemma}.
    \[ \coprod\limits_{x(\gexp{\id[I_{m+1}]}{j+1} \gtimes \gexp{l}{n-j-1}) \in S_{n,j+1}} \obox{n+1}{n+1,1} \xrightarrow{\restr{x\lstepbar{m,j}}{\obox{}{}}} X_{n,j}. \]
    As $x(\gexp{\id[I_{m+1}]}{j+1} \gtimes \gexp{l}{n-j-1})$ is not in $X_{n,j}$, the $n+1$-cube $x\lstepbar{m,j}$ is also not in $X_{n,j}$ (as one of its faces is $x(\gexp{\id[I_{m+1}]}{j+1} \gtimes \gexp{l}{n-j-1})$). 
    The generating cubes of $X_{n,j+1}$ which are not contained in $X_{n,j}$ are exactly those of the form $x\lstepbar{m,j}$ for some $x \in S_{n,j+1}$.
    Thus, we may write $X_{n,j+1}$ as the following pushout.
    \[ \begin{tikzcd}[column sep = large]
        \coprod\limits_{x(\gexp{\id[I_{m+1}]}{j+1} \gtimes \gexp{l}{n-j-1}) \in S_{n,j+1}} \obox{n+1}{n+1,1} \ar[r, "\restr{x\lstepbar{m,j}}{\obox{}{}}"] \ar[d, hook] \ar[rd, phantom, "\ulcorner" very near end] & X_{n,j} \ar[d, hook] \\
        \coprod\limits_{x(\gexp{\id[I_{m+1}]}{j+1} \gtimes \gexp{l}{n-j-1}) \in S_{n,j+1}} \cube{n+1} \ar[r, "x\lstepbar{m,j}", swap] & X_{n,j+1}
    \end{tikzcd} \]
    Thus, the map $X_{n,j} \to X_{n,j+1}$ is a pushout along a coproduct of open box inclusions.

    We now show the map $X_{n,n} \to X_{n+1,1}$ is a pushout along a coproduct of open box inclusions.
    Let $S_{n+1,1}$ be the set of $(n+1)$-cubes which factor through the map $\id[I_{m+1}] \gtimes \gexp{l}{n} \from \gexp{I_{m+1}}{n+1} \to I_{m+1} \gtimes \gexp{I_m}{n}$ and are not contained in $X_{n,n}$.
    % We have that $X_0$ contains all $(n+1)$-cubes which factor through $\gexp{l}{n+1} \from \gexp{I_{m+1}}{n} \to \gexp{I_m}{n}$.
    % Thus, $X_{n,n}$ contains these $(n+1)$-cubes as well.
    Similar to before, the generating cubes of $X_{n+1,1}$ which are not contained in $X_{n,n}$ are exactly those of the form $x\lstepbar[n+1]{m,0}$ for some $x \in S_{n+1,1}$.
    Thus \cref{thm:lanodyne_factor_lemma} similarly shows that $X_{n+1,1}$ may be written as the following pushout.
    \[ \begin{tikzcd}[column sep = large]
        \coprod\limits_{x(\id[I_{m+1}] \gtimes \gexp{l}{n}) \in S_{n+1,1}} \obox{n+1}{n+1,0} \ar[r, "{\restr{x\lstepbar[n+1]{m,0}}{\obox{}{}}}"] \ar[d, hook] \ar[rd, phantom, "\ulcorner" very near end] & X_{n,n} \ar[d, hook] \\
        \coprod\limits_{x(\id[I_{m+1}] \gtimes \gexp{l}{n}) \in S_{n+1,1}} \cube{n+1} \ar[r, "{x\lstepbar[n+1]{m,0}}", swap] & X_{n+1,1}
    \end{tikzcd} \]
    Thus, the map $X_{n,n} \to X_{n+1,1}$ is a pushout along a coproduct of open box inclusions.
\end{proof}
From this, we conclude that the inclusion of the 1-nerve of a graph into its nerve is anodyne.
\begin{corollary} \label{thm:1nerve_nerve_we}
    The natural map $\gnerve[1]{G} \to \gnerve{G}$ is anodyne.
\end{corollary}
\begin{proof}
    By \cref{thm:nerve_colim}, $\gnerve{G}$ is a transfinite composition of the maps $c^* \from \gnerve[m] G \to \gnerve[m+2] G$ for $m > 0$.
    By \cref{thm:lr_anodyne}, each $c^*$ is anodyne.
    Thus, each component of the colimit cone $\gnerve[m] \to \gnerve G$ is anodyne.
\end{proof}
This gives us our main theorem.
% \mainthm*
\begin{proof}[Proof of \cref{thm:main}]
    The first result is proven in \cref{thm:nerve_kan}.
    The second result is proven in \cref{thm:1nerve_nerve_we}.
\end{proof}

\section{Consequences} \label{sec:consequences}

\subsection*{Proof of the conjecture of Babson, Barcelo, de Longueville, Laubenbacher}
Using our main result, we obtain a proof of \cite[Thm.~5.2]{babson-barcelo-longueville-laubenbacher} which does not rely on the cubical approximation hypothesis \cite[Prop.~5.1]{babson-barcelo-longueville-laubenbacher}.

\begin{theorem}\label{thm:conj-bbdll}
    There is a natural group isomorphism $A_n(G, v) \cong \pi_n(\reali{\gnerve[1]{G}}, v)$. 
\end{theorem}
\begin{proof}
    We have
    \[ \begin{array}[b]{r@{ \ }l l}
        A_n(G, v) & \cong \pi_n(\gnerve{G}, v) & \text{ by \cref{thm:a_eq_cset_pi}} \\
        & \cong \pi_n(\reali{\gnerve{G}}, v) & \text{ by \cref{thm:cset_pi_eq_pi}} \\
        & \cong \pi_n(\reali{\gnerve[1]{G}}, v) & \text{ by \cref{thm:main,thm:cset_we_to_top}}.
    \end{array} \qedhere \]
\end{proof}

\subsection*{Discrete homology of graphs}

In this subsection, we prove the discrete Hurewicz theorem, relating discrete homotopy and homology groups.
To do so, we begin with a quick review of cubical homology (cf.~\cite{massey:singular-homology-theory,brown-higgins-sivera:nonabelian-algebraic-topology,barcelo-greene-jarrah-welker:connections}).

First, we recall the standard definition of homology (with integral coefficients) of a chain complex.
A (bounded) \emph{chain complex} (over $\Z$) consists of a collection $\{ C_n \mid n \geq 0 \}$ of abelian groups and, for $n \geq 1$, a group homomorphism $\cmap \from C_n \to C_{n-1}$ such that $\cmap[n-1] \cmap[n] = 0$.
%That is, $\image \cmap \subseteq \ker \cmap[n-1]$.
A map $f \from C \to D$ of chain complexes consists of maps $f_n \colon C_n \to D_n$ for $n \geq 0$ which make the respective squares commute.
We write $\Ch$ for the category of chain complexes.
Define a functor $H_* \from \Ch \to \fcat{\nat}{\Ab}$ by taking a chain complex $C$ to a sequence of homology groups given by:
\[ H_n C := \begin{cases}
    C_0 / \image \cmap[0] & n = 0 \\
    \ker \cmap[n] / \image \cmap[n+1] & n > 0.
\end{cases} \]
We refer to $H_nC$ as the \emph{$n$-th homology} of $C$ with integer coefficients.
One verifies that a map of chain complexes $C \to D$ induces maps $H_n C \to H_n D$ between their $n$-th homologies.

Next, we explain how to construct a chain complex out of a cubical set via a construction analogous to the normalized complex in simplicial singular homology.
We construct a functor $N \from \cSet_* \to \Ch$ by
\[ (NX)_n := F_*X_n / DX_n \]
where $F_* X_n$ is the free abelian group on the pointed set $X_n$ and $D X_n$ is the subgroup of $F_*X_n$ generated by degenerate cubes, i.e.~those that lie in the image of a degeneracy or a connection $X_{n-1} \to X_n$.
% Explicitly, $F_*(X, x_0) = \bigoplus_{x \in X \setminus \{x_0\}} \mathbb{Z}$. 
% Finally, 
The chain differentials $\partial \from (NX)_n \to (NX)_{n-1}$ are given by
\[ \partial(x) := \sum\limits_{\substack{i = 1, \dots, n \\ \varepsilon = 0, 1}} (-1)^{i+\varepsilon} x \face{}{i,\varepsilon}. \]
\begin{definition} 
    The \emph{reduced cubical homology with integer coefficients} functor is the functor $\tilde{H}_* \from \cSet_* \to \fcat{\nat}{\Ab}$ given by the composite
    \[ \cSet_* \xrightarrow{N} \Ch \xrightarrow{H_*} \fcat{\nat}{\Ab}. \]
\end{definition}
We contrast the definition of \emph{reduced} homology with that of \emph{unreduced homology}, which is defined on the category of non-pointed cubical sets by using non-pointed free abelian groups.

Unlike in the construction of simplicial homology, we are forced to take the quotient of $F_* X_\bullet$ by the subcomplex of cubes in the image of a degeneracy map.
We do, however, have flexibility regarding whether or not to quotient by the subcomplex of cubes in the image of a connection map, as the two complexes are quasi-isomorphic \cite[Cor.~3.10]{barcelo-greene-jarrah-welker:connections}.

%We also write $H \from \cSet \to \fcat{\nat}{\Ab}$ for the cubical homology functor.
% Note that there is an analogue of relative homology for cubical sets as well.
% We write $\tilde{H} \from \cSet_* \to \fcat{\nat}{\Ab}$ for the relative cubical homology functor from pointed cubical sets to graded abelian groups.
\begin{definition}[cf.~{\cite[\S2]{barcelo-capraro-white}}] \label{def:disc-homol}
    The \emph{reduced discrete homology} functor $\tilde{DH}_* \from \Graph_* \to \fcat{\nat}{\Ab}$ is the composite of functors
    \[ \Graph_* \xrightarrow{\gnerve[1]{}} \cSet_* \xrightarrow{\tilde{H}_*} \fcat{\nat}{\Ab}. \]
\end{definition}

By the homotopy invariance of cubical homology \cite[Thm.~3.11]{carranza-kapulkin-tonks:hurewicz-cubical}, we have the following fact.

\begin{proposition} \label{thm:homology_we}
    Let $f \from (X,x) \to (Y,y)$ be a pointed cubical map.
    If $f$ is a weak equivalence then $H_*f \from H_*(X,x) \to H_*(Y,y)$ is an isomorphism of graded abelian groups. \noproof
    % As well, for $x \in X_0$, the map $\tilde{H}f \from \tilde{H}(X,x) \to \tilde{H}(Y,fx)$ is an isomorphism of graded abelian groups.
\end{proposition}
From this, we deduce that the discrete homology of a graph is the same as the cubical homology of its nerve.
\begin{corollary} \label{thm:graph_homology_iso}
    For a pointed graph $(G, v)$, the natural map $\gnerve[1]{G} \to \gnerve{G}$ induces an isomorphism 
    \[ \tilde{H}_*(\gnerve[1]{G}, v) \cong \tilde{H}_*(\gnerve{G}, v). \] 
    of graded abelian groups.
    % Moreover, for any $v \in G$, this map induces an isomorphism
    % \[ \tilde{H} (\gnerve[1] G, v) \cong \tilde{H} (\gnerve G, v) \]
    % of graded abelian groups.
\end{corollary}
\begin{proof}
    Follows from \cref{thm:homology_we} and \cref{thm:main}.
\end{proof}

For any pointed Kan complex $(X, x)$, using the unit of the adjunction 
\[ \adjunct{F_*}{U_*}{\cSet_*}{\Ab^{\square^\op}} \]
between pointed cubical sets and cubical abelian groups, we may construct a natural map $\pi_n(X, x) \to \tilde{H}_*(X, x)$, since $\pi_n(U_*F_* (X, x)) \cong \tilde{H}_n(X, x)$ by \cite[Thm.~4.11]{carranza-kapulkin-tonks:hurewicz-cubical}.
This is the \emph{Hurewicz homomorphism}, cf.~\cite[Def.~4.14]{carranza-kapulkin-tonks:hurewicz-cubical}.

%Alternatively, one may use the usual Hurewicz homomorphism for the geometric realization of $X$ in the category of topological spaces and the fact that both homotopy groups and homology groups of a Kan complex agree with those of its geometric realization (for homotopy groups, this is \cref{thm:cset_pi_eq_pi}; for homology groups \cite[Thm.~3.17]{barcelo-greene-jarrah-welker})
We then have the classical theorem of Hurewicz, phrased in the language of cubical sets:

\begin{theorem}[{\cite[Thm.~4.16]{carranza-kapulkin-tonks:hurewicz-cubical}}] \label{thm:cset_hurewicz}
    Let $n \geq 2$ and $(X,x)$ be a pointed connected Kan complex.
    Suppose $\pi_i (X,x) = 0$ for all $i \in \{ 1, \dots, n-1 \}$, i.e., $X$ is $n$-connected.
    Then the Hurewicz homomorphism $\pi_n (X,x) \to \tilde{H}_n (X,x)$ is an isomorphism. \noproof
\end{theorem}

\begin{definition}
For any pointed connected graph $(G,v)$ and $n \geq 2$, we therefore obtain the \emph{discrete Hurewicz homomorphism} $A_n(G, v) \to \tilde{DH}_n (G, v)$ as the composite
    \[ \begin{array}{r l l}
        A_n(G, v) & \cong \pi_n(\gnerve{G}, v) & \text{ by \cref{thm:a_eq_cset_pi}} \\
        & \to \tilde{H}_n (\gnerve{G}) & \text{the Hurewicz homomorphism} \\
        & \cong \tilde{H}_n (\gnerve[1]{G}) & \text{ by \cref{thm:graph_homology_iso}} \\
        & = \tilde{DH}_n (G, v) & \text{ by \cref{def:disc-homol}.}
    \end{array} \]
\end{definition}

One may verify that this map recovers the homomorphism defined in \cite[\S5.2]{lutz}.
We then have the expected discrete analogue of the Hurewicz theorem.

\begin{theorem}[Discrete Hurewicz Theorem] \label{graph-hurewicz}
    Let $n \geq 2$ and $(G,v)$ be a pointed connected graph.
    Suppose $A_i (G,v) = 0$ for all $i \in \{ 1, \dots, n-1 \}$.
    Then the induced Hurewicz map $A_n (G,v) \to \tilde{DH}_n (G, v)$ is an isomorphism.
\end{theorem}
\begin{proof}
  This is an immediate consequence of the definition of the discrete Hurewicz homomorphism and \cref{thm:cset_hurewicz}.
\end{proof}

\subsection*{Fibration Category of Graphs}
Via \cref{thm:main}, we may view the nerve functor as a functor $\gnerve \from \Graph \to \Kan$ taking values in Kan complexes.
From this, we induce a fibration category structure on the category of graphs.
\begin{restatable}{theorem}{fibcat} \label{thm:fibcat}
  The category $\Graph$ of graphs and graph maps carries a fibration category structure where:
    \begin{itemize}
      \item the weak equivalences are the weak homotopy equivalences, i.e., maps $f \from G \to H$ such that, for all $v \in G$ and $n \geq 0$, the map $A_n f \from A_n(G, v) \to A_n (H, f(v))$ is an isomorphism;
      \item the fibrations are maps $f \from G \to H$ which are sent to fibrations $\gnerve f \from \gnerve G \to \gnerve H$ under the nerve functor $\gnerve \from \Graph \to \cSet$.
    \end{itemize}
\end{restatable}
Before proving this, we consider factorization of the diagonal map separately.
Recall that, for any graph $G$, we have a commutative triangle 
\[ \begin{tikzcd}
  & PG \ar[rd, "{(\face{*}{1,0}, \face{*}{1,1})}"] & \\
  G \ar[rr, "{(\id[G], \id[G])}"] \ar[ur, hook] && G \times G
\end{tikzcd} \]
where $G \ito PG$ sends a vertex $v$ to the constant path on $v$.
\begin{lemma} \label{thm:graph_fib_diag}
  For any graph $G$, applying $\gnerve \from \Graph \to \cSet$ to the diagram
  \[ \begin{tikzcd}
    & PG \ar[rd, "{(\face{*}{1,0}, \face{*}{1,1})}"] & \\
    G \ar[rr, "{(\id[G], \id[G])}"] \ar[ur, hook] && G \times G
  \end{tikzcd} \]
  gives a factorization of the diagonal map $\gnerve G \to \gnerve G \times \gnerve G$ as a weak equivalence followed by a fibration.
\end{lemma}
\begin{proof}
  Applying naturality of the isomorphism in \cref{thm:nerve_hom_iso_fin} to the maps $\cube{1} \to \cube{0}$ and $\bd \cube{1} \to \cube{1}$ gives that the diagram
  \[ \begin{tikzcd}
    \gnerve G \ar[r] \ar[d, "\cong"'] & \gnerve (PG) \ar[r] \ar[d, "\cong"] & \gnerve (G \times G) \ar[d, "\cong"] \\
    \gnerve G \ar[r] & \rhom(\cube{1}, \gnerve G) \ar[r] & \gnerve G \times \gnerve G
  \end{tikzcd} \]
  commutes.
  \cite[Prop.~2.22]{carranza-kapulkin:homotopy-cubical} shows the bottom composite provides a factorization of the diagonal map into a weak equivalence followed by a fibration.
  Thus, the top left map is a weak equivalence and the top right map is a fibration.
  % As the nerve preserves finite products by \cref{thm:nerve_fin_lim}, the diagonal map $G \to G \times G$ is sent to the diagonal map $\gnerve G \to \gnerve G \times \gnerve G$.
  % The map $G \ito PG$ is sent to the pre-composition map $\degen{*}{1} \from \gnerve G \ito \hom(\cube{1}, \gnerve G)$, whereas the map $PG \to G \times G$ is sent to the map $(\face{*}{1,0}, \face{*}{1,1}) \from \hom(\cube{1}, \gnerve G) \to \gnerve G \times \gnerve G$.
\end{proof}
% \fibcat*
\begin{proof}[Proof of \cref{thm:fibcat}]
    Both the two-out-of-six property for weak equivalences and closure of acyclic fibrations under isomorphisms follow from \cref{thm:kan_fib}.
    Pullbacks exist as $\Graph$ is complete; they preserve (acyclic) fibrations by \cref{thm:nerve_fin_lim}.
    \cref{thm:main} shows that the nerve of every graph is a Kan complex.
    Factorization of the diagonal map is shown in \cref{thm:graph_fib_diag}, and this gives the factorization of an arbitrary map via \cite[Factorization lemma]{brown}.
\end{proof}
The following proposition gives a useful criterion for verifying whether a map is a fibration.
\begin{proposition} \label{graph_fib_equiv}
  A graph map $f \from G \to H$ is a fibration if and only if, for any commutative square
  \[ \begin{tikzcd}
      \reali[m]{\dfobox} \ar[r] \ar[d] & G \ar[d, "f"] \\
      \gexp{I_m}{n} \ar[r] & H
  \end{tikzcd} \]
  there exists $k \geq 0$ and a lift $\gexp{I_{m + 2k}}{n} \to G$ of the following square.
  \[ \begin{tikzcd}
      \reali[m+2k]{\dfobox} \ar[r, "{c_*}^{k}"] \ar[d] & \reali[m]{\dfobox} \ar[r] \ar[d] & G \ar[d, "f"] \\
      \gexp{I_{m+2k}}{n} \ar[r, "{c_*}^k"'] \ar[urr, dotted] & \gexp{I_m}{n} \ar[r] & H
  \end{tikzcd} \]
\end{proposition}
\begin{proof}
  The map $f$ is a fibration if and only if any commutative square
  \[ \begin{tikzcd}
      \dfobox \ar[r] \ar[d] & \gnerve G \ar[d, "\gnerve f"] \\
      \cube{n} \ar[r] & \gnerve H
  \end{tikzcd} \]
  admits a lift.
  As $\gnerve$ is a sequential colimit, every such square and lift (if it exists) factors as
  \[ \begin{tikzcd}
      \dfobox \ar[r] \ar[d] & {\gnerve[m] G} \ar[r, "(c^*)^k"] \ar[d, "{\gnerve[m] f}"] & \gnerve[m + 2k] G \ar[r] \ar[d, "{\gnerve[m+2k] f}"] & \gnerve G \ar[d, "\gnerve f"] \\
      \cube{n} \ar[r] \ar[urr, dotted] & {\gnerve[m] H} \ar[r, "(c^*)^k"] & \gnerve[m+2k] H \ar[r] & \gnerve H
  \end{tikzcd} \]
  for some $m \geq 1$ and $k \geq 0$.
  The left two squares in this diagram correspond to a diagram
  \[ \begin{tikzcd}
      \reali[m+2k]{\dfobox} \ar[r, "(c_*)^k"] \ar[d] & \reali[m]{\dfobox} \ar[r] \ar[d] & G \ar[d, "f"] \\
      \gexp{I_{m+2k}}{n} \ar[r, "(c_*)^k"'] \ar[urr, dotted] & \gexp{I_m}{n} \ar[r] & H
  \end{tikzcd} \]
  in $\Graph$ by the realization-nerve adjunction.
\end{proof}
By definition, the nerve functor $\gnerve \from \Graph \to \cSet$ reflects (and preserves) fibrations.
We show that it reflects weak equivalences as well.
\begin{proposition} \label{thm:nerve_we_create}
  The nerve functor $\gnerve \from \Graph \to \cSet$ reflects weak equivalences.
  That is, given a map $f \from G \to H$, if $\gnerve f \from \gnerve G \to \gnerve H$ is a weak equivalence then $f$ is a weak equivalence.
\end{proposition}
\begin{proof}
  Follows from \cref{thm:a_eq_cset_pi} and \cite[Thm.~4.7]{carranza-kapulkin:homotopy-cubical}.
\end{proof}
As the nerve functor preserves finite limits, it is straightforward to show it is exact.
\begin{proposition} \label{thm:nerve_exact}
  The nerve functor $\gnerve \from \Graph \to \cSet$ is exact.
\end{proposition}
\begin{proof}
  The nerve functor preserves fibrations and acyclic fibrations by definition and by \cref{thm:nerve_we_create}, respectively.
  \cref{thm:nerve_fin_lim} shows that it preserves all finite limits.
\end{proof}

A consequence of \cref{thm:nerve_exact} is that \emph{fibration sequences} induce a long exact sequence of homotopy groups.
\begin{theorem} \label{thm:graph_les}
  Let $f \from (G, v) \to (H, w)$ be a fibration between pointed graphs and $(F, v)$ be the fiber of $f$ over $w$, i.e.~the pullback
  \[ \begin{tikzcd}
    (F,v) \ar[d] \ar[r, "i"] \ar[rd, phantom, "\ulcorner" very near start] & (G, v) \ar[d, "f"] \\
    (I_0, 0) \ar[r, "w"] & (H,w)
  \end{tikzcd} \]
  in $\Graph_*$.
  Then, there is a long exact sequence
  \[ \begin{tikzcd}
    \dots \rar & A_n (F,v) \ar[r, "A_n i"] & A_n (G,v) \ar[r, "A_n f"] \ar[d, phantom, ""{coordinate, name=a}] & A_n (H,w) \ar[dll, rounded corners, 
    to path={ -- ([xshift=2ex]\tikztostart.east)
              |- (a) [near end]\tikztonodes
              -| ([xshift=-2ex]\tikztotarget.west)
              -- (\tikztotarget) }] & {} \\
    & A_{n-1} (F,v) \ar[r, "A_{n-1} i"] & A_{n-1} (G,v) \ar[r, "A_{n-1} f"] \ar[d, phantom, "\dots"{description, name=b}] & A_{n-1} (H,w) \ar[dll, rounded corners, 
    to path={ -- ([xshift=2ex]\tikztostart.east)
            |- (b) [near end]\tikztonodes }] & \\
    & A_1 (F,v) \ar[r, "A_1 i"] \ar[from=b, rounded corners, to path={ -| ([xshift=-2ex]\tikztotarget.west)
            -- (\tikztotarget) }] & A_1 (G,v) \ar[d, phantom, ""{coordinate, name=c}] \ar[r, "A_1 f"] & A_1 (H,w) \ar[dll, rounded corners, 
    to path={ -- ([xshift=2ex]\tikztostart.east)
            |- (c) [near end]\tikztonodes
            -| ([xshift=-2ex]\tikztotarget.west)
            -- (\tikztotarget) }] & {} \\
    & A_0 (F, v) \ar[r, "A_0 i"] & A_0 (G, v) \ar[r, "A_0 f"] & A_0 (H, w).
  \end{tikzcd} \] 
\end{theorem}
\begin{proof}
  By \cref{thm:nerve_exact}, the map $\gnerve f \from (\gnerve G, v) \to (\gnerve H, w)$ is a fibration and its fiber is naturally isomorphic to $(\gnerve F, v)$.
  The result then follows by applying \cite[Cor.~4.6]{carranza-kapulkin:homotopy-cubical} and \cref{thm:a_eq_cset_pi}.
\end{proof}

There are two large classes of examples of fibrations: fiber bundles and $m$-fibrations.

\begin{definition}[Babson] \label{def:fiber-bundle}
  For $m \geq 1$, a graph map $f \from G \to H$ is an \emph{$m$-fiber bundle} with fiber $F$ if, for any $u \from \gexp{I_1}{n} \to H$, the left map in the pullback diagram
  \[ \begin{tikzcd}
    P \ar[r] \ar[d] \ar[rd, phantom, "\ulcorner" very near start] & G \ar[d, "f"] \\
    \gexp{I_1}{n} \ar[r, "u"] & H
  \end{tikzcd} \]
  is a product projection $\gexp{I_1}{n} \times F \to \gexp{I_1}{n}$ from the categorical product of $\gexp{I_1}{n}$ and $F$.
\end{definition}

\begin{definition}
  A fiber bundle $G \to H$ is \emph{trivial} if it is a product projection $H \times F \to H$.
\end{definition}

\begin{proposition} \label{pb-m-fib-bundle-is-m-fib-bundle}
    Any pullback of an $m$-fiber bundle is an $m$-fiber bundle.
\end{proposition}
\begin{proof}
    Follows from the two-pullback lemma. \qedhere
\end{proof}

\begin{proposition} \label{m-fib-bundle-is-1-fib-bundle}
    For $m \geq 1$, every $(m+1)$-fiber bundle is an $m$-fiber bundle.
\end{proposition}
\begin{proof}
    Fix an $(m+1)$-fiber bundle $f \from G \to H$ and a map $u \from \gexp{I_m}{n} \to H$.
    The map $\gexp{l}{n} \from \gexp{I_{m+1}}{n} \to \gexp{I_m}{n}$ admits a section $i \from \gexp{I_m}{n} \to \gexp{I_{m+1}}{n}$.
    By assumption, the left map in the pullback diagram
    \[ \begin{tikzcd}
        P \ar[rr] \ar[d] \ar[rd, phantom, "\ulcorner" very near start] & {} & G \ar[d, "f"] \\
        \gexp{I_{m+1}}{n} \ar[r, "\gexp{l}{n}"] & \gexp{I_m}{n} \ar[r, "u"] & H
    \end{tikzcd} \]
    is a product projection $\gexp{I_{m+1}}{n} \times F \to \gexp{I_{m+1}}{n}$.
    As the right map in the pullback square
    \[ \begin{tikzcd}
        \gexp{I_m}{n} \times F \ar[r, "{i \times \id}"] \ar[d] \ar[rd, phantom, "\ulcorner" very near start] & \gexp{I_{m+1}}{n} \times F \ar[d] \\
        \gexp{I_m}{n} \ar[r, "i"] & \gexp{I_{m+1}}{n}
    \end{tikzcd} \]
    is a product projection, the left map is as well.
    Thus, the outer square in
    \[ \begin{tikzcd}
        \gexp{I_m}{n} \times F \ar[r, "i \times \id"] \ar[d] \ar[rd, phantom, "\ulcorner" very near start] & \gexp{I_{m+1}}{n} \times F \ar[r] \ar[d] \ar[rd, phantom, "\ulcorner" very near start] & G \ar[d, "f"] \\
        \gexp{I_m}{n} \ar[r, "i"] & \gexp{I_{m+1}}{n} \ar[r, "u \circ \gexp{l}{n}"] & H
    \end{tikzcd} \]
    is a pullback whose left map is a product projection.
\end{proof}

Recall a graph $G$ is a \emph{retract} of a graph $H$ if there is an inclusion $i \colon G \ito H$ with a retraction $r \colon H \to G$, i.e.~$ri = \id$.

\begin{lemma} \label{retr-1-fib-bundle}
    If $H$ is a retract of $\gexp{I_1}{n}$ for some $n \geq 0$ then any 1-fiber bundle $G \to H$ is trivial.
\end{lemma}
\begin{proof}
    As $H$ is a retract of $\gexp{I_1}{n}$, we fix an inclusion $i \from H \ito \gexp{I_1}{n}$ and retraction $r \from \gexp{I_1}{n} \to H$.
    Given a fiber bundle $f \from G \to H$, the left map in the pullback diagram
    \[ \begin{tikzcd}
        F \times \gexp{I_1}{n} \ar[r] \ar[d] \ar[rd, phantom, "\ulcorner" very near start] & G \ar[d, "f"] \\
        \gexp{I_1}{n} \ar[r, "r"] & H
    \end{tikzcd} \] 
    is a product projection.
    By universal property of the pullback, the outer square in
    \[ \begin{tikzcd}
        G \ar[rrd, "\id", bend left] \ar[ddr, swap, "if", bend right] \ar[rd, dotted] & {} & {} \\
        {} & F \times \gexp{I_1}{n} \ar[r] \ar[d] \ar[rd, phantom, "\ulcorner" very near start] & H \ar[d, "f"] \\
        {} & \gexp{I_1}{n} \ar[r, "r"] & H
    \end{tikzcd} \]
    induces a map $G \to F \times \gexp{I_1}{n}$.
    By the two-pullback lemma, the left square in
    \[ \begin{tikzcd}
        G \ar[r, dotted] \ar[d, "f"'] & F \times \gexp{I_1}{n} \ar[r] \ar[d] \ar[rd, phantom, "\ulcorner" very near start] & G \ar[d, "f"] \\
        H \ar[r, hook, "i"] & \gexp{I_1}{n} \ar[r, "r"] & H
    \end{tikzcd} \]
    is a pullback.
    Therefore, the outer square in
    \[ \begin{tikzcd}
        G \ar[r, dotted] \ar[d, "f"'] \ar[rd, phantom, "\ulcorner" very near start] & F \times \gexp{I_1}{n} \ar[d] \ar[r] \ar[rd, phantom, "\ulcorner" very near start] & F \ar[d] \\
        H \ar[r, hook, "i"] & \gexp{I_1}{n} \ar[r] & I_0
    \end{tikzcd} \]
    is a pullback, which proves that $G$ is a product of $H$ and $F$, and $f$ is the projection onto the first component.
\end{proof}

\begin{theorem}
    For any $m \geq 1$, a 1-fiber bundle is an $m$-fiber bundle.
\end{theorem}
\begin{proof}
    Fix a 1-fiber bundle $f \from G \to H$ and a map $u \from \gexp{I_m}{n} \to H$.
    By \cref{pb-m-fib-bundle-is-m-fib-bundle}, the left map in the pullback diagram
    \[ \begin{tikzcd}
        u^*G \ar[r] \ar[d] \ar[rd, phantom, "\ulcorner" very near start] & G \ar[d, "f"] \\
        \gexp{I_m}{n} \ar[r, "u"] & H
    \end{tikzcd} \]
    is a 1-fiber bundle.
    By \cref{retr-1-fib-bundle}, it suffices to show $\gexp{I_m}{n}$ is a retract of $\gexp{I_1}{k}$ for some $k \geq 0$.

    Define an embedding $i \from I_m \ito \gexp{I_1}{m}$ by sending a vertex $k \in I_m$ to the tuple $(x_1, \dots, x_m)$ where $x_i = 1$ if $i \leq k$ and $x_i = 0$ if $i > k$.
    This map has a retraction $r \from \gexp{I_1}{m} \to I_m$ which sends a tuple $(x_1, \dots, x_n)$ to its sum $\sum\limits_{i=1}^{n} x_i$.
    Thus, the map $\gexp{i}{n} \from \gexp{I_m}{n} \to \gexp{I_1}{mn}$ has a retraction $\gexp{r}{n} \from \gexp{I_1}{mn} \to \gexp{I_m}{n}$. \qedhere
\end{proof}

In particular, if $f \from X \to Y$ is an $m$-fiber bundle for some $m \geq 1$ then it is a fiber bundle for all $m \geq 1$.
In the remainder of this section, we refer to such an $f$ as simply a \emph{fiber bundle}.
% \begin{remark}
%   There is a notion of $m$-fiber bundle where one consider pullbacks $\gexp{I_m}{n} \to H$ along cubes of length $m$.
%   One may show that any $m$-fiber bundle is an $(m-1)$-fiber bundle.
%   In particular, an $m$-fiber bundle is a fiber bundle in the sense of \cref{def:fiber-bundle}.
% \end{remark}

\begin{theorem} \label{1-fib-bundle-is-fib}
  Every fiber bundle is a fibration.
\end{theorem}
\begin{proof}
  Fix a fiber bundle $f \from G \to H$.
  We apply \cref{graph_fib_equiv} and consider a lifting problem
  \[ \begin{tikzcd}
      \reali[m]{\dfobox} \ar[r] \ar[d, hook] & G \ar[d, "f"] \\
      \gexp{I_m}{n} \ar[r] & H
  \end{tikzcd} \]
  Taking a pullback of the right and bottom map gives a factorization of this square as
  \[ \begin{tikzcd}
      \reali[m]{\dfobox} \ar[r, dotted] \ar[d, hook] & \gexp{I_1}{n} \times F \ar[r] \ar[d] \ar[rd, phantom, "\ulcorner" very near start] & G \ar[d, "f"] \\
      \gexp{I_m}{n} \ar[r, "\id"] & \gexp{I_m}{n} \ar[r] & H
  \end{tikzcd} \]
  The middle map is a fibration since it is a product projection.
  Hence, there exists $k \geq 0$ such that the outer square in
  \[ \begin{tikzcd}
      \reali[m+2k]{\dfobox} \ar[r, "c_*^{k}"] \ar[d, hook] & \reali[m]{\dfobox} \ar[d, hook] \ar[r, dotted] & \gexp{I_1}{n} \times F \ar[d] \\
      \gexp{I_{m+2k}}{n} \ar[r, "c_*^k"'] \ar[urr, dotted] & \gexp{I_m}{n} \ar[r, "\id"] & \gexp{I_m}{n}
  \end{tikzcd} \]
  admits a lift.
  Post-composing this lift with the map $\gexp{I_1}{n} \times F \to G$ gives a lift of the outer square in
  \[ \begin{tikzcd}[baseline=0.75em, anchor=south]
      \reali[m+2k]{\dfobox} \ar[r, "c_*^{k}"] \ar[d, hook] & \reali[m]{\dfobox} \ar[d, hook] \ar[r] & G \ar[d, "f"] \\
      \gexp{I_{m+2k}}{n} \ar[r, "c_*^k"'] \ar[urr, dotted] & \gexp{I_m}{n} \ar[r] & H
  \end{tikzcd} \qedhere \]
\end{proof}
\begin{corollary}
  A pointed fiber bundle $f \from (G, v) \to (H, w)$ with fiber $(F, v)$ induces a long exact sequence
  \[ \begin{tikzcd}
      \dots \rar & A_n (F,v) \ar[r, "A_n i"] & A_n (G,v) \ar[r, "A_n f"] \ar[d, phantom, ""{coordinate, name=a}] & A_n (H,w) \ar[dll, rounded corners, 
      to path={ -- ([xshift=2ex]\tikztostart.east)
                |- (a) [near end]\tikztonodes
                -| ([xshift=-2ex]\tikztotarget.west)
                -- (\tikztotarget) }] & {} \\
      & A_{n-1} (F,v) \ar[r, "A_{n-1} i"] & A_{n-1} (G,v) \ar[r, "A_{n-1} f"] \ar[d, phantom, "\dots"{description, name=b}] & A_{n-1} (H,w) \ar[dll, rounded corners, 
      to path={ -- ([xshift=2ex]\tikztostart.east)
              |- (b) [near end]\tikztonodes }] & \\
      & A_1 (F,v) \ar[r, "A_1 i"] \ar[from=b, rounded corners, to path={ -| ([xshift=-2ex]\tikztotarget.west)
              -- (\tikztotarget) }] & A_1 (G,v) \ar[d, phantom, ""{coordinate, name=c}] \ar[r, "A_1 f"] & A_1 (H,w) \ar[dll, rounded corners, 
      to path={ -- ([xshift=2ex]\tikztostart.east)
              |- (c) [near end]\tikztonodes
              -| ([xshift=-2ex]\tikztotarget.west)
              -- (\tikztotarget) }] & {} \\
      & A_0 (F, v) \ar[r, "A_0 i"] & A_0 (G, v) \ar[r, "A_0 f"] & A_0 (H, w).
    \end{tikzcd} \]
  of A-homotopy groups.
\end{corollary}
\begin{proof}
  Follows from \cref{1-fib-bundle-is-fib,thm:graph_les}.
\end{proof}

The second class of maps which are fibrations is the class of $m$-fibrations.
\begin{definition}
  For $m \geq 1$, a graph map $f \from G \to H$ is an \emph{$m$-fibration} if $\gnerve[m] f \from \gnerve[m] G \to \gnerve[m] H$ is a Kan fibration.
\end{definition}
\begin{proposition} \label{mult_sat}
  Let $m, k \geq 1$. \begin{enumerate}
      \item The inclusion $\{ 0, km \} \to I_{km}$ is in the saturation of $\{ \varnothing \to I_0, \{ 0, m \} \to I_m \}$.
      \item For $\varepsilon = 0, 1$, the end-point inclusion $\{ \varepsilon km \} \to I_{km}$ is in the saturation of $\left\{ \{ \varepsilon m \} \to I_{m} \right\}$.
      \item The inclusion $\bd \gexp{I_{km}}{n} \to \gexp{I_{km}}{n}$ is in the saturation of $\left\{ \bd \gexp{I_m}{j} \to \gexp{I_m}{j} \mid j \leq n \right\}$.
      \item The inclusion $\reali[km]{\dfobox} \to \gexp{I_{km}}{n}$ is in the saturation of $\left\{ \reali[m]{\dfobox[j]} \to \gexp{I_{km}}{j} \mid j \leq n \right\}$.
  \end{enumerate}
\end{proposition}
\begin{proof} \begin{enumerate}
  \item By induction, if $\bd I_{km} \to I_m$ lies in the saturation then the bottom map in the pushout
  \[ \begin{tikzcd}
      \{ 0, km \} \sqcup \{ 0, m \} \ar[r] \ar[d, "{[f, g]}"'] \ar[rd, phantom, "\ulcorner" very near end] & I_{km} \sqcup I_m \ar[d] \\
      \{ 0, km, (k+1)m \} \ar[r] & I_{(k+1)m}
  \end{tikzcd} \]
  is in the saturation as a pushout of the coproduct of maps in the saturation, where $[f, g]$ is defined by
  \[ \begin{array}{l l}
      f(0) = 0 & f(km) = km \\
      g(0) = km & g(m) = (k+1)m.
  \end{array} \]
  The inclusion $\{ 0, (k+1)m \} \to \{ 0, km, (k+1)m \}$ is a pushout along $\varnothing \to I_0$, thus the composite $\{ 0, (k+1)m \} \to I_{(k+1)m}$ lies in the saturation.
  \item By induction, if $\{ \varepsilon m \} \to I_{km}$ lies in the saturation then the bottom map in the pushout
  \[ \begin{tikzcd}
      I_0 \ar[r, "\varepsilon km"] \ar[d, "(1-\varepsilon)m"'] \ar[rd, phantom, "\ulcorner" very near end] & I_{km} \ar[d] \\
      I_m \ar[r] & I_{(k+1)m}
  \end{tikzcd} \]
  lies in the saturation.
  Pre-composing with the end-point inclusion $\{ \varepsilon m \} \to I_m$ gives one endpoint inclusion $\{ 0 \} \to I_{(k+1)m}$.
  % The other endpoint inclusion $\{ (k+1)m \} \to I_{(k+1)m}$ is naturally isomorphic, thus also lies in the saturation.
  \item The inclusion $\bd \gexp{I_{km}}{n} \to \gexp{I_{km}}{n}$ may be written as a pushout product
  \[ (\bd \gexp{I_{km}}{n} \to \gexp{I_{km}}{n}) = (\{ 0, km \} \to I_{km})^{\pp[\gtimes] n}. \]
  Noting the equality
  \[ \{ \varnothing \to I_0, \{ 0, m \} \to I_m \}^{\pp[\gtimes] n} = \{ \bd \gexp{I_m}{j} \to \gexp{I_m}{j} \mid j \leq n \}, \]
  this follows from (1) by \cite[Prop.~5.3.4]{hovey-shipley-smith}.
  \item The inclusion $\reali[km]{\dfobox} \to \gexp{I_{km}}{n}$ may be written as a pushout product
  \[ (\reali[km]{\dfobox} \to \gexp{I_m}{n}) = (\{ 0, km \} \to I_{km})^{\pp[\gtimes] (i-1)} \pp[\gtimes] (\{ \varepsilon km \} \to I_{km}) \pp[\gtimes] (\{ 0, km \} \to I_{km})^{\pp[\gtimes] n-i}. \]
  Noting the equality
  \begin{align*}
    &\{ \varnothing \to I_0, \{ 0, m \} \to I_m \}^{\pp[\gtimes] i-1} \pp[\gtimes] \{ \{ \varepsilon m \} \to I_m \} \pp \{ \varnothing \to I_0, \{ 0, m \} \to I_m \}^{\pp[\gtimes] n-i} \\
    &\qquad = \{ \reali[m]{\dfobox[j]} \to \gexp{I_m}{j} \mid j \leq n \},
  \end{align*}   
  this follows from (1) and (2) by \cite[Prop.~5.3.4]{hovey-shipley-smith}. \qedhere
\end{enumerate} \end{proof}
\begin{theorem} \label{mfib_is_fib}
  Let $k, m \geq 1$.
  Every $m$-fibration is both a $km$-fibration and a fibration.
\end{theorem}
\begin{proof}
  Let $f \from G \to H$ be an $m$-fibration.
  We have that $f$ is a $km$-fibration by \cref{mult_sat}.
  To see $f$ is a fibration, we apply \cref{graph_fib_equiv}.
  For a commutative square,
  \[ \begin{tikzcd}
      \reali[t]{\dfobox} \ar[r] \ar[d] & G \ar[d] \\
      \gexp{I_t}{n} \ar[r] & H
  \end{tikzcd} \]
  let $k \geq 0$ be the smallest non-negative integer such that $t + k$ is a multiple of $m$.
  As $f$ is a $(t+k)$-fibration, the required lift exists.
\end{proof}

% \begin{proposition} \label{thm:pg_exp}
%   For a graph $G$, we have an isomorphism
%   \[ \gnerve (PG) \cong \rhom(\cube{1}, \gnerve G) \cong \rhom(\cube{1}, \gnerve G). \]
% \end{proposition}
% \begin{proof}
%     Let $f \from G \gtimes I_\infty \to H$ be a graph map and $\overline{f} \from I_\infty \to \ghom{G}{H}$ denote its image under the adjunction bijection.
%   Then,
%   \begin{itemize}
%     \item $\overline{f}$ stabilizes in direction $(1,0)$ if and only if there exists $M \geq 0$ such that, for all $i > M$ and $v \in G$, we have $f(v, -i) = f(v, -M)$;
%     \item $\overline{f}$ stabilizes in direction $(1,1)$ if and only if there exists $M \geq 0$ such that, for all $i > M$ and $v \in G$, we have $f(v, i) = f(v, M)$.
%   \end{itemize}
% \end{proof}
\cref{thm:nerve_exact} also shows that the nerve functor preserves loop spaces.
\begin{proposition} \label{thm:loopsp_graph_cset}
  The square
  \[ \begin{tikzcd}
      \Graph_* \ar[r, "\gnerve"] \ar[d, "\Omega"'] & \Kan_* \ar[d, "\Omega"] \\
      \Graph_* \ar[r, "\gnerve"] & \Kan_*
  \end{tikzcd} \]
  commutes up to natural isomorphism.
\end{proposition}
\begin{proof}
  Fix a pointed graph $(G,v)$.
  By \cref{thm:graph_loopsp_pb}, the square
  \[ \begin{tikzcd}
    \loopsp(G, v) \ar[r] \ar[d] \ar[rd, phantom, "\ulcorner" very near start] & PG \ar[d, "{(\face{*}{1,0}, \face{*}{1,1})}"] \\
    I_0 \ar[r, "{(v,v)}"] & G \times G
  \end{tikzcd} \]
  is a pullback.
  \cref{thm:nerve_fin_lim} shows the nerve functor preserves this pullback, thus the square
  \[ \begin{tikzcd}
    \gnerve (\loopsp(G, v)) \ar[r] \ar[d] \ar[rd, phantom, "\ulcorner" very near start] & \gnerve(PG) \ar[d, "{(\face{*}{1,0}, \face{*}{1,1})}"] \\
    \gnerve I_0 \ar[r, "{(v,v)}"] & \gnerve (G \times G)
  \end{tikzcd} \]
  is a pullback.
  \cref{thm:nerve_fin_lim} also shows the nerve preserves products and the terminal object.
  \cref{thm:graph_fib_diag} implies that $\gnerve(PG) \cong \rhom(\cube{1}, \gnerve G)$, hence the square
  \[ \begin{tikzcd}
    \gnerve (\loopsp(G, v)) \ar[r] \ar[d] \ar[rd, phantom, "\ulcorner" very near start] & \rhom(\cube{1}, \gnerve G) \ar[d, "{(\face{*}{1,0}, \face{*}{1,1})}"] \\
    \cube{0} \ar[r, "{(v,v)}"] & \gnerve G \times \gnerve G
  \end{tikzcd} \]
  is a pullback.
  That is, $\gnerve (\loopsp(G,v)) \cong \loopsp(\gnerve G, v)$.
\end{proof}

By \cref{thm:fib_pt}, the category $\Graph_*$ of pointed graphs has a fibration category structure as well.
Thus, the loop graph functor $\loopsp \from \Graph_* \to \Graph_*$ is a functor between fibration categories.
\begin{theorem} \label{thm:graph_loopsp_exact}
  The loop graph functor $\loopsp \from \Graph_* \to \Graph_*$ is exact.
\end{theorem}
\begin{proof}
  By \cref{thm:nerve_exact,thm:cset_loopsp_exact}, the composite
  $\loopsp(\gnerve \uvar, \uvar) \from \Graph_* \to \cSet_*$ is exact.
  Applying \cref{thm:loopsp_graph_cset}, this gives that the composite $\gnerve(\loopsp(\uvar, \uvar)) \from \Graph_* \to \cSet_*$ is exact.
  The nerve functor reflects fibrations and acyclic fibrations as, by definition, it creates them.
  It reflects finite limits by \cref{thm:nerve_ref_fin_lim}. 
  From this, it follows that $\loopsp \from \Graph_* \to \Graph_*$ preserves fibrations, acyclic fibrations, and finite limits.
\end{proof}

\subsection*{Cubical enrichment of the category of graphs}

Recall (for instance from \cite[Def.~3.4.1]{riehl:cat-htpy}) that a functor $F \from \cat{C} \to \cat{D}$ between monoidal categories $(\cat{C}, \otimes_\cat{C}, I_\cat{C})$ and $(\cat{D}, \otimes_\cat{D}, I_\cat{D})$ is \emph{lax monoidal} if there exist natural transformations
\[ Fc \otimes_\cat{D} Fc' \to F(c \otimes c') \text{ and } I_\cat{D} \to FI_\cat{C} \]
subject to the associativity and unitality conditions, which we omit here.

\begin{lemma}\leavevmode \label{thm:nerve_monoidal}
  \begin{enumerate}
    \item For $m \geq 1$, the functor $\gnerve[m] \from \Graph \to \cSet$ is lax monoidal.
    \item The functor $\gnerve \from \Graph \to \cSet$ is lax monoidal.
  \end{enumerate}
\end{lemma}

\begin{proof}
  By \cref{thm:reali_prod}, the functors $\reali[m]{\uvar} \from \cSet \to \Graph$ are strong monoidal, and hence in particular, oplax monoidal.
  Thus, their right adjoints $\gnerve[m] \from \Graph \to \cSet$ are lax monoidal, proving (1).

  Clearly, $\gnerve I_0 \iso \cube 0$.  
  As the geometric product preserves colimits in each variable and colimits commute with colimits, the cubical set $\gnerve G \gprod \gnerve H$ (for graphs $G$ and $H$) is the colimit of the following diagram.
  \[ \gnerve G \gprod \gnerve H = \colim \left( \begin{tikzcd}
    \gnerve[1] G \gprod \gnerve[1] H \ar[d, hook] \ar[r, hook] & \gnerve[2] G \gprod \gnerve[1] H \ar[r, hook] \ar[d, hook] & \gnerve[3] G \gprod \gnerve[1] H \ar[r, hook] \ar[d, hook] & \dots \\
    \gnerve[1] G \gprod \gnerve[2] H \ar[r, hook] \ar[d, hook] & \gnerve[2] G \gprod \gnerve[2] H \ar[r, hook] \ar[d, hook] & \gnerve[3] G \gprod \gnerve[2] H \ar[r, hook] \ar[d, hook] & \dots \\
    \gnerve[1] G \gprod \gnerve[3] H \ar[r, hook] \ar[d, hook] & \gnerve[2] G \gprod \gnerve[3] H \ar[r, hook] \ar[d, hook] & \gnerve[3] G \gprod \gnerve[3] H \ar[r, hook] \ar[d, hook] & \dots \\
    \vdots & \vdots & \vdots & \ddots
  \end{tikzcd} \right) \] 
  Computing this colimit component-wise in $\Set$, one verifies that the colimit of the diagonal
  \[ \colim (\begin{tikzcd}
    \gnerve[1] G \gprod \gnerve[1] H \ar[r, hook] & \gnerve[2] G \gprod \gnerve[2] H \ar[r, hook] & \gnerve[3] G \gprod \gnerve[3] H \ar[r, hook] & \dots
  \end{tikzcd}) \]
  computes the same cubical set.
  As $\gnerve (G \gtimes H)$ is the colimit
  \[ \gnerve (G \gtimes H) = \colim (\begin{tikzcd}
    \gnerve[1] (G \gtimes H) \ar[r, hook] & \gnerve[2] (G \gtimes H) \ar[r, hook] & \gnerve[3] (G \gtimes H) \ar[r, hook] & \dots
  \end{tikzcd}) \]
  the lax monoidal maps $\gnerve[m] G \gprod \gnerve[m] H \to \gnerve[m] (G \gtimes H)$ induce a map on colimits $\gnerve G \gprod \gnerve H \to \gnerve (G \gtimes H)$ which satisfies the required associativity and unitality conditions. 
\end{proof}

\begin{remark}
   An alternative proof of (2) can be given using \cite[Prop.~1.24]{doherty-kapulkin-lindsey-sattler}, which gives an explicit description of the geometric product of cubical sets.
   Explicitly, the $n$-cubes of $\gnerve G \gprod \gnerve H$ are in bijective correspondence with pairs of cubes
   \[ (\cube k \to \gnerve G,\ \cube l \to \gnerve H)\text{,} \]
   where $k+l = n$.
   By definition of $\gnerve$, each such a pair corresponds in turn to pair of maps
   \[  (\gexp{I_\infty}{k} \to G,\ \gexp{I_\infty}{l} \to H) \]
   that stabilize in all directions.
   Taking products of these maps, we obtain a map $\gexp{I_\infty}{n} \to G \gtimes H$ that stabilizes in all directions.
\end{remark}

\begin{corollary}
  The nerve functor $\gnerve \from \Graph \to \cSet$ preserves homotopy equivalences. \noproof
\end{corollary}

We describe the notion of enriched categories informally, with a reference to \cite[Def.~3.3.1]{riehl:cat-htpy} in lieu of a fully formal statement.

\begin{definition}
  For a monoidal category $(\cat{V}, \otimes, 1)$, a \emph{$(\cat{V}, \otimes, I)$-enriched category} $\cat{C}$ consists of
  \begin{itemize}
    \item a class of objects $\ob\cat{C}$;
    \item for $X, Y \in \ob\cat{C}$, a \emph{morphism} object $\cat{C}(X, Y) \in \cat{V}$;
    \item for $X, Y, Z \in \ob\cat{C}$, a \emph{composition} morphism $\circ \from \cat{C}(Y, Z) \otimes \cat{C}(X, Y) \to \cat{C}(X, Z)$ in $\cat{V}$;
    \item for $X \in \ob\cat{C}$, an \emph{identity} morphism $1 \to \cat{C}(X, X)$ in $\cat{V}$,
  \end{itemize}
  subject to appropriate associativity and unitality axioms (cf.~\cite[Def.~3.3.1]{riehl:cat-htpy}).
\end{definition}

\begin{example} 
  Any locally-small category is a $(\Set, \times, \{ \ast \})$-enriched category, where the objects, morphism sets, composition function, and identity morphisms are as usually defined. 
\end{example}

\begin{example}
  Any closed monoidal category is enriched over itself.
  In particular, $\Graph_{\mathsf{G}}$ is a $(\Graph, \gtimes, I_0)$-enriched category where
  \begin{itemize}
    \item $\ob\Graph_{\mathsf{G}}$ is the collection of all graphs;
    \item for graphs $G, H$, the morphism graph is $\ghom{G}{H}$;
    \item for graphs $G, H, J$, the composition morphism is the graph map given by composition of graph maps regarded as vertices:
    \[ \ghom{H}{J} \gtimes \ghom{G}{H} \to \ghom{G}{J}\text{;} \]
    \item the identity morphism is the identity map on $G$ as a vertex $\id[G] \from I_0 \to \ghom{G}{G}$.
  \end{itemize}
\end{example}

\begin{example}
  Since enrichment can be transferred along lax monoidal functors \cite[Lem.~3.4.3]{riehl:cat-htpy}, \cref{thm:nerve_monoidal} implies there is a $(\cSet, \gprod, \cube{0})$-enriched category $\Graph_{\boxcat}$ of graphs where $\gnerve(\ghom{G}{H})$ is the morphism cubical set.
  Composition and identity morphisms are defined analogously.
\end{example}

\begin{definition}
  A $(\cSet, \gprod, \cube{0})$-enriched category $\cat{C}$ is \emph{locally Kan} if, for all $X, Y \in \ob\cat{C}$, the cubical set $\cat{C}(X, Y)$ is a Kan complex.
\end{definition}

\begin{theorem} \label{thm:graph-locally-kan}
  The $(\cSet, \gprod, \cube{0})$-enriched category $\Graph_{\boxcat}$ of graphs is locally Kan.
\end{theorem}

\begin{proof}
  Follows from \cref{thm:main}.
\end{proof}

By the results of \cite{kapulkin-voevodsky}, we have established a presentation of the $(\infty, 1)$-category of graphs (localized at $A$-homotopy equivalences).
In subsequent work, we use this presentation to show that A-homotopy equivalences are not part of a model structure on the category of graphs, as well as to study homotopy limits in discrete homotopy theory.

\bibliographystyle{amsalphaurlmod}
\bibliography{all-refs.bib}

\begin{bibdiv}
\begin{biblist}

\bib{atkin:i}{article}{
      author={Atkin, R.~H.},
       title={An algebra for patterns on a complex. {I}},
        date={1974},
        ISSN={0020-7373},
     journal={Internat. J. Man-Mach. Stud.},
      volume={6},
       pages={285\ndash 307},
         url={https://doi.org/10.1016/S0020-7373(74)80024-6},
}

\bib{atkin:ii}{article}{
      author={Atkin, R.~H.},
       title={An algebra for patterns on a complex. {II}},
        date={1976},
        ISSN={0020-7373},
     journal={Internat. J. Man-Mach. Stud.},
      volume={8},
      number={5},
       pages={483\ndash 498},
         url={https://doi.org/10.1016/s0020-7373(76)80015-6},
}

\bib{babson-barcelo-longueville-laubenbacher}{article}{
      author={Babson, Eric},
      author={Barcelo, H\'{e}l\`ene},
      author={de~Longueville, Mark},
      author={Laubenbacher, Reinhard},
       title={Homotopy theory of graphs},
        date={2006},
        ISSN={0925-9899},
     journal={J. Algebraic Combin.},
      volume={24},
      number={1},
       pages={31\ndash 44},
         url={https://doi.org/10.1007/s10801-006-9100-0},
}

\bib{barcelo-capraro-white}{article}{
      author={Barcelo, H\'{e}l\`ene},
      author={Capraro, Valerio},
      author={White, Jacob~A.},
       title={Discrete homology theory for metric spaces},
        date={2014},
        ISSN={0024-6093},
     journal={Bull. Lond. Math. Soc.},
      volume={46},
      number={5},
       pages={889\ndash 905},
         url={https://doi.org/10.1112/blms/bdu043},
}

\bib{barcelo-greene-jarrah-welker:connections}{article}{
      author={Barcelo, H\'{e}l\`ene},
      author={Greene, Curtis},
      author={Jarrah, Abdul~Salam},
      author={Welker, Volkmar},
       title={Homology groups of cubical sets with connections},
        date={2021},
        ISSN={0927-2852},
     journal={Appl. Categ. Structures},
      volume={29},
      number={3},
       pages={415\ndash 429},
         url={https://doi.org/10.1007/s10485-020-09621-x},
}

\bib{brown-higgins:algebra-of-cubes}{article}{
      author={Brown, Ronald},
      author={Higgins, Philip~J.},
       title={Colimit theorems for relative homotopy groups},
        date={1981},
        ISSN={0022-4049},
     journal={J. Pure Appl. Algebra},
      volume={22},
      number={1},
       pages={11\ndash 41},
         url={https://doi.org/10.1016/0022-4049(81)90080-3},
}

\bib{brown-higgins-sivera:nonabelian-algebraic-topology}{book}{
      author={Brown, Ronald},
      author={Higgins, Philip~J.},
      author={Sivera, Rafael},
       title={Nonabelian algebraic topology},
      series={EMS Tracts in Mathematics},
   publisher={European Mathematical Society (EMS), Z\"{u}rich},
        date={2011},
      volume={15},
        ISBN={978-3-03719-083-8},
         url={https://doi.org/10.4171/083},
        note={Filtered spaces, crossed complexes, cubical homotopy groupoids, With contributions by Christopher D. Wensley and Sergei V. Soloviev},
}

\bib{barcelo-kramer-laubenbacher-weaver}{article}{
      author={Barcelo, H\'{e}l\`ene},
      author={Kramer, Xenia},
      author={Laubenbacher, Reinhard},
      author={Weaver, Christopher},
       title={Foundations of a connectivity theory for simplicial complexes},
        date={2001},
        ISSN={0925-9899},
     journal={Adv. Appl. Math.},
      volume={26},
      number={1},
       pages={97\ndash 128},
         url={https://doi.org/10.1007/s10801-006-9100-0},
}

\bib{barcelo-laubenbacher}{article}{
      author={Barcelo, H\'{e}l\`ene},
      author={Laubenbacher, Reinhard},
       title={Perspectives on {$A$}-homotopy theory and its applications},
        date={2005},
        ISSN={0012-365X},
     journal={Discrete Math.},
      volume={298},
      number={1-3},
       pages={39\ndash 61},
         url={https://doi.org/10.1016/j.disc.2004.03.016},
}

\bib{buchholtz-morehouse}{incollection}{
      author={Buchholtz, Ulrik},
      author={Morehouse, Edward},
       title={Varieties of cubical sets},
        date={2017},
   booktitle={Relational and algebraic methods in computer science},
      series={Lecture Notes in Comput. Sci.},
      volume={10226},
   publisher={Springer, Cham},
       pages={77\ndash 92},
         url={https://doi.org/10.1007/978-3-319-57418-9_5},
}

\bib{barthel-riehl}{article}{
      author={Barthel, Tobias},
      author={Riehl, Emily},
       title={On the construction of functorial factorizations for model categories},
        date={2013},
        ISSN={1472-2747,1472-2739},
     journal={Algebr. Geom. Topol.},
      volume={13},
      number={2},
       pages={1089\ndash 1124},
         url={https://doi.org/10.2140/agt.2013.13.1089},
}

\bib{brown}{article}{
      author={Brown, Kenneth~S.},
       title={Abstract homotopy theory and generalized sheaf cohomology},
        date={1973},
        ISSN={0002-9947},
     journal={Trans. Amer. Math. Soc.},
      volume={186},
       pages={419\ndash 458},
         url={https://doi.org/10.2307/1996573},
}

\bib{cisinski:presheaves}{article}{
      author={Cisinski, Denis-Charles},
       title={Les pr\'{e}faisceaux comme mod\`eles des types d'homotopie},
        date={2006},
        ISSN={0303-1179},
     journal={Ast\'{e}risque},
      number={308},
       pages={xxiv+390},
}

\bib{carranza-kapulkin:homotopy-cubical}{article}{
      author={Carranza, Daniel},
      author={Kapulkin, Krzysztof},
       title={Homotopy groups of cubical sets},
        date={2023},
        ISSN={0723-0869,1878-0792},
     journal={Expo. Math.},
      volume={41},
      number={4},
       pages={Paper No. 125518, 55},
         url={https://doi.org/10.1016/j.exmath.2023.125518},
}

\bib{carranza-kapulkin-tonks:hurewicz-cubical}{article}{
      author={Carranza, Daniel},
      author={Kapulkin, Krzysztof},
      author={Tonks, Andrew},
       title={The {H}urewicz theorem for cubical homology},
        date={2023},
        ISSN={0025-5874,1432-1823},
     journal={Math. Z.},
      volume={305},
      number={4},
       pages={Paper No. 61, 20},
         url={https://doi.org/10.1007/s00209-023-03352-0},
}

\bib{doherty-kapulkin-lindsey-sattler}{unpublished}{
      author={Doherty, Brandon},
      author={Kapulkin, Krzysztof},
      author={Lindsey, Zachery},
      author={Sattler, Christian},
       title={Cubical models of $(\infty,1)$-categories},
        date={2020},
        note={to appear in Mem.~Amer.~Math.~Soc.},
}

\bib{hirschhorn}{book}{
      author={Hirschhorn, Philip~S.},
       title={Model categories and their localizations},
      series={Mathematical Surveys and Monographs},
   publisher={American Mathematical Society, Providence, RI},
        date={2003},
      volume={99},
        ISBN={0-8218-3279-4},
         url={https://doi.org/10.1090/surv/099},
}

\bib{hovey}{book}{
      author={Hovey, Mark},
       title={Model categories},
      series={Mathematical Surveys and Monographs},
   publisher={American Mathematical Society, Providence, RI},
        date={1999},
      volume={63},
        ISBN={0-8218-1359-5},
}

\bib{hovey-shipley-smith}{article}{
      author={Hovey, Mark},
      author={Shipley, Brooke},
      author={Smith, Jeff},
       title={Symmetric spectra},
        date={2000},
        ISSN={0894-0347},
     journal={J. Amer. Math. Soc.},
      volume={13},
      number={1},
       pages={149\ndash 208},
         url={https://doi.org/10.1090/S0894-0347-99-00320-3},
}

\bib{kramer-laubenbacher}{incollection}{
      author={Kramer, Xenia~H.},
      author={Laubenbacher, Reinhard~C.},
       title={Combinatorial homotopy of simplicial complexes and complex information systems},
        date={1998},
   booktitle={Applications of computational algebraic geometry ({S}an {D}iego, {CA}, 1997)},
      series={Proc. Sympos. Appl. Math.},
      volume={53},
   publisher={Amer. Math. Soc., Providence, RI},
       pages={91\ndash 118},
         url={https://doi.org/10.1090/psapm/053/1602359},
}

\bib{kapulkin-voevodsky}{article}{
      author={Kapulkin, Krzysztof},
      author={Voevodsky, Vladimir},
       title={A cubical approach to straightening},
        date={2020},
        ISSN={1753-8416},
     journal={J. Topol.},
      volume={13},
      number={4},
       pages={1682\ndash 1700},
         url={https://doi.org/10.1112/topo.12173},
}

\bib{htt}{book}{
      author={Lurie, Jacob},
       title={Higher topos theory},
      series={Annals of Mathematics Studies},
   publisher={Princeton University Press, Princeton, NJ},
        date={2009},
      volume={170},
        ISBN={978-0-691-14049-0; 0-691-14049-9},
         url={https://doi.org/10.1515/9781400830558},
}

\bib{lutz}{article}{
      author={Lutz, Bob},
       title={Higher discrete homotopy groups of graphs},
        date={2021},
     journal={Algebr. Comb.},
      volume={4},
      number={1},
       pages={69\ndash 88},
         url={https://doi.org/10.5802/alco.151},
}

\bib{malle}{article}{
      author={Malle, G.},
       title={A homotopy theory for graphs},
        date={1983},
        ISSN={0017-095X},
     journal={Glas. Mat. Ser. III},
      volume={18(38)},
      number={1},
       pages={3\ndash 25},
}

\bib{massey:singular-homology-theory}{book}{
      author={Massey, William~S.},
       title={Singular homology theory},
      series={Graduate Texts in Mathematics},
   publisher={Springer-Verlag, New York-Berlin},
        date={1980},
      volume={70},
        ISBN={0-387-90456-5},
}

\bib{may}{article}{
      author={May, J.~Peter},
       title={Classifying spaces and fibrations},
        date={1975},
        ISSN={0065-9266,1947-6221},
     journal={Mem. Amer. Math. Soc.},
      volume={1},
       pages={xiii+98},
         url={https://doi.org/10.1090/memo/0155},
}

\bib{memoli-zhou}{unpublished}{
      author={M{\'e}moli, Facundo},
      author={Zhou, Ling},
       title={Persistent homotopy groups of metric spaces},
        date={2019},
}

\bib{riehl:cat-htpy}{book}{
      author={Riehl, Emily},
       title={Categorical homotopy theory},
      series={New Mathematical Monographs},
   publisher={Cambridge University Press, Cambridge},
        date={2014},
      volume={24},
        ISBN={978-1-107-04845-4},
         url={https://doi.org/10.1017/CBO9781107261457},
}

\bib{riehl-verity:theory-and-practice-of-reedy-categories}{article}{
      author={Riehl, Emily},
      author={Verity, Dominic},
       title={The theory and practice of {R}eedy categories},
        date={2014},
        ISSN={1201-561X},
     journal={Theory Appl. Categ.},
      volume={29},
       pages={256\ndash 301},
}

\bib{szumilo}{article}{
      author={Szumi{\l}o, Karol},
       title={Homotopy theory of cofibration categories},
        date={2016},
        ISSN={1532-0073},
     journal={Homology Homotopy Appl.},
      volume={18},
      number={2},
       pages={345\ndash 357},
         url={https://doi.org/10.4310/HHA.2016.v18.n2.a19},
}

\bib{berg-garner}{article}{
      author={van~den Berg, Benno},
      author={Garner, Richard},
       title={Topological and simplicial models of identity types},
        date={2012},
        ISSN={1529-3785,1557-945X},
     journal={ACM Trans. Comput. Log.},
      volume={13},
      number={1},
       pages={Art. 3, 44},
         url={https://doi.org/10.1145/2071368.2071371},
}

\end{biblist}
\end{bibdiv}

\end{document}